\documentclass[a4paper, 10pt]{amsart}
\usepackage{times,%latexsym,
amssymb, amsmath,amsthm}

\allowdisplaybreaks
\newtheorem{theorem}{Theorem}
\newtheorem{lemma}[theorem]{Lemma}
\newtheorem{definition}[theorem]{Definition}

\numberwithin{theorem}{section}
\numberwithin{equation}{section}

\def\N{\mathbb{N}}
\def\R{\mathbb{R}}
\def\C{\mathbb{C}}

\def\H{\mathcal{H}}

\def\Z{\mathbb{Z}}
\def\HA{\mathcal{H}_{\partial A}}

\def\dx{\,dx}

\def\wto{\rightharpoonup}

\def\S{{\mathcal{S}(\Gamma)}}
\def\Ss{{\mathcal{S}^\ast(\Gamma)}}
\def\SGA{{\mathcal{S}^\ast(\Gamma,A)}}
\def\SGAs{{\mathcal{S}^\ast(\Gamma,A,\sigma)}}
\DeclareMathOperator{\spt}{spt}

\DeclareMathOperator{\dist}{dist}
\DeclareMathOperator*{\osc}{osc}
\renewcommand{\Tilde}{\widetilde}
\renewcommand{\d}{\partial}
\renewcommand{\epsilon}{\varepsilon}
\newcommand{\eps}{\varepsilon}
\renewcommand{\rho}{\varrho}
\def\edge{\hspace{0.15em}\mbox{\LARGE$\llcorner$}\hspace{0.05em}}

\def\XXint#1#2#3{{\setbox0=\hbox{$#1{#2#3}{\int}$}
                 \vcenter{\hbox{$#2#3$}}\kern-.5\wd0}}
\def\Xint#1{\mathchoice
              {\XXint\displaystyle\textstyle{#1}}
              {\XXint\textstyle\scriptstyle{#1}}
              {\XXint\scriptstyle\scriptscriptstyle{#1}}
              {\XXint\scriptscriptstyle\scriptscriptstyle{#1}}
              \!\int}
\def\mint{\Xint-}

\def\lk{[\![}
\def\rk{]\!]}

\def\eqn#1$$#2$${\begin{equation}\label#1#2\end{equation}}
\begin{document}

\renewcommand{\leftmargini}{0.7cm}
\title[The Evolution of $H$-surfaces with Plateau boundary condition]
 {The evolution of $H$-surfaces
 with\\  a Plateau boundary condition}
\date{June 26, 2012}

\author[F. Duzaar]{Frank Duzaar}
\address{Frank Duzaar\\Department Mathematik, Universit\"at
Erlangen--N\"urnberg\\ Cauerstrasse 11, 91058 Erlangen, Germany}
\email{duzaar@mi.uni-erlangen.de}

\author[C. Scheven]{Christoph Scheven}
\address{Christoph Scheven\\Department Mathematik, Universit\"at
Erlangen--N\"urnberg\\ Cauerstrasse 11, 91058 Erlangen, Germany}
\email{scheven@mi.uni-erlangen.de}

%\subjclass[2000]{}
%\keywords{}

\maketitle

\begin{abstract}
In this paper we consider the heat flow associated to
the classical Plateau problem for surfaces of prescribed mean curvature. To  be precise, for a given Jordan curve $\Gamma\subset
\R^3$, a given prescribed mean curvature function $H\colon
\R^3\to\R$ and an initial datum $u_o\colon B\to \R^3$
satisfying the Plateau boundary condition, i.e. that
$u_o\big|_{\partial B}\colon \partial B\to \Gamma$
is a homeomorphism, 
we consider the geometric flow
$$
	\partial_tu -\Delta u
     = -2(H\circ u)D_1u\times D_2u 
	\quad\mbox{in $B\times (0,\infty)$,}
$$
$$
	\mbox{$u(\cdot ,0)=u_o$ on $B$,}\qquad
	\mbox{$u(\cdot,t)\big|_{\partial B}\colon \partial B\to \Gamma$
	is weakly monotone for all $t>0$.}
$$
We show that an isoperimetric condition on $H$ ensures the existence of a global weak solution. Moreover, we establish that these global solutions sub-converge as $t\to \infty$ to a conformal solution of the classical Plateau problem for surfaces of prescribed mean curvature.

\end{abstract}

\tableofcontents

%****************************************************************
\section{Introduction}
\label{sec:intro}

\subsection{The history of the problem}
The classical Plateau problem for $H$-surfaces consists in the construction of parametric surfaces $u\colon B\to\R^3$ with prescribed mean
curvature $H$  and with boundary $\Gamma$; here $\Gamma$ is a given closed, rectifiable  Jordan curve in $\R^3$. For parametric surfaces $u\in C^2 (B,\R^3)\cap C^0(\overline B,\R^3)$ defined on the unit disk $B$ in
$\R^2$ it has the following formulation:
 \begin{equation}\label{Plateau-probl}
    \left\{
    \begin{array}{c}
     \Delta u =2 (H\circ u) D_1 u\times D_2 u\quad\mbox{on }B,\\[5pt]
     \mbox{$u\big|_{\partial B}:\partial B\to\Gamma$ is a homeomorphism,}\\[5pt]
     |D_1u|^2 -|D_2u|^2 =0=D_1u \cdot
     D_2u\quad\mbox{on $B$.}
     \end{array}
     \right.
 \end{equation}
Here \eqref{Plateau-probl}$_1$ is called the $H$-surface-equation and \eqref{Plateau-probl}$_3$ are the conformality
relations. Non-constant $C^2$-solutions $u$ to \eqref{Plateau-probl}$_1$ and \eqref{Plateau-probl}$_3$ are usually called
$H$-surfaces in $\R^3$. The geometric significance of
\eqref{Plateau-probl}$_1$ and \eqref{Plateau-probl}$_3$ is that its
solutions are 2-dimensional immersed surfaces in $\R^3$
with mean curvature given by $H$. The Plateau boundary condition
\eqref{Plateau-probl}$_2$ is a free boundary condition with one degree
of freedom. Problem \eqref{Plateau-probl} has been treated by many authors, e.g. by 
Heinz \cite{Heinz:1954}, Hildebrandt \cite{Hildebrandt:1969, Hildebrandt:1970},
Gulliver \& Spruck \cite{GulliverSpruck:1971, GulliverSpruck:1972},
Steffen \cite{Steffen:1976, Steffen:1976-2} and Wente \cite{Wente:1969}. 
Several optimal
results have been obtained in the seventies and these results essentially settle the existence problem \eqref{Plateau-probl}
for disk type surfaces in $\R^3$. One prominent example is the result of Hildebrandt \cite{Hildebrandt:1969, Hildebrandt:1970}
which ensures the existence of an $H$-surface contained
in a ball $B_R$ of radius $R$ in $\R^3$ whenever $\Gamma$ is a closed, rectifiable Jordan curve contained in $B_R$
and the prescribed mean curvature function satisfies $|H|\le\frac1R$
on $B_R$.

In contrast to the Plateau problem for $H$-surfaces, 
much less is known for the associated flow to \eqref{Plateau-probl}.
This geometric flow can be formulated as follows:
\begin{equation}\label{time-dep-Plateau-probl}
    \left\{
    \begin{array}{c}
     \partial_tu -\Delta u
     = -2(H\circ u)D_1u\times D_2u 
	\quad\mbox{in $B\times (0,\infty)$,}\\[5pt]
     \mbox{$u(\cdot ,0)=u_o$ on $B$,}\\[5pt]
     \mbox{$u(\cdot,t)\big|_{\partial B}\colon \partial B\to \Gamma$
	is weakly monotone for all $t>0$.}     \end{array}
     \right.
 \end{equation}
For the precise definition we refer to \eqref{Plateau}.
In the special case $H\equiv 0$, i.e. the evolutionary
Plateau problem for minimal surfaces, this flow was considered by
Chang and Liu 
in \cite{ChangLiu:2003,ChangLiu:2003.2,ChangLiu:2004}.
Their main result ensures the existence of a global weak solution 
% (see \eqref{def-weak} for the definition) 
which sub-converges
asymptotically as $t\to\infty$ to a conformal solution
of the Plateau problem for minimal surfaces, i.e. a solution  of \eqref{Plateau-probl} with $H\equiv 0$. Moreover, the same
authors treated the case
$H\equiv {\rm const}$, see \cite{ChangLiu:2003.2}. In this case, existence of a global
weak solution with image contained in a ball of radius $R$, was
shown under the Hildebrandt type condition $|H|< \frac1R$. Finally, in \cite{Struwe:1988}
Struwe considered the $H$-surface flow subject to a free boundary condition of the type $u(\cdot ,t)\in S$ on $\partial B$
 and the orthogonality condition $\partial_ru(\cdot ,t)
\perp T_{u(\cdot ,t)}S$ on $\partial B$ for all $t>0$. In this context
$S$ is assumed to be a sufficiently regular surface in $\R^3$ which
is diffeomorphic to standard sphere $S^2$.

With respect to the associated flow for a Dirichlet boundary condition
on the lateral boundary,
several results ensure the existence of
global weak, respectively smooth classical solutions. In this
case the problem
can be formulated as follows:
\begin{equation}\label{time-dep-Dirichlet-probl}
    \left\{
    \begin{array}{c}
     \partial_tu -\Delta u
     = -2(H\circ u)D_1u\times D_2u 
	\quad\mbox{in $B\times (0,\infty)$,}\\[5pt]
     \mbox{$u(\cdot)=u_o$ on $B\times\{0\}\cup \partial B\times (0,\infty)$,}    \end{array}
     \right.
\end{equation}
for given initial and boundary values $u_o\in W^{1,2}(B,\R^3)$. 
In \cite{Rey:1991}, Rey showed that the Hildebrandt type condition 
$|u_o|< R$ on $B$ and  $|H|< \frac1R$ for a 
Lipschitz continuous prescribed mean curvature function $H:B_R\to\R$ is sufficient to
guarantee the existence of a smooth global solution of
\eqref{time-dep-Dirichlet-probl}. For an existence result for short
time existence of classical solutions without any  assumption on $H$
and $u_o$ we refer to Chen \& Levine \cite{ChenLevine:2002}. In this paper also the bubbling phenomenon
at a first singular time is analyzed. Such a bubbling was ruled out
by Rey in \cite{Rey:1991} for the proof of the long time existence
using the Hildebrandt condition. The previous papers rely on methods
developed by Struwe \cite{Struwe:1985} for the harmonic map heat flow. In recent papers Hong \& Hsu
\cite{HongHsu:2012} respectively Leone \& Misawa \& Verde \cite{LeoneMisawaVerde:2012} established the existence of a global weak solution  for the evolutionary flow to higher dimensional $H$-surfaces by different methods; in the first paper the authors were also able to
show that the solutions are of class $C^{1,\alpha}$, which is the best regularity one can expect for systems including 
the parabolic $n$-Laplacean as leading term. 
Again a Hildebrandt type condition serves to exclude
the occurrence of $H$-bubbles during the flow. We note that all mentioned papers
rely on the strong assumption of Lipschitz continuity for $H$ and
the Hildebrandt-type condition for the existence proof of global solutions. These strong assumptions were considerably weakened in
a previous paper \cite{BoegDuzSchev:2011},
 in the sense that an {\bf isoperimetric condition} for bounded and continuous prescribed mean curvature functions $H\colon\R^3\to\R$
is sufficient for the existence of global solutions to
\eqref{time-dep-Dirichlet-probl}. Such an isoperimetric condition
relates the weighted $H$-volume of a set $E\subset \R^3$ to its perimeter via
\begin{equation}\label{iso-simple}
	2\Big|\int_E H\, d\xi\Big|\le c\mathbf P(E)
\end{equation}
for any set $E\subset\R^3$ with finite perimeter $\mathbf P(E)\le s$.
The condition \eqref{iso-simple} is termed isoperimetric condition of
type $(c,s)$. 
In \cite{Steffen:1976, Steffen:1976-2}, Steffen showed that such a condition with $c<1$ is sufficient for the existence of solutions to
\eqref{Plateau-probl}, and moreover that all known classical
existence results can be deduced from such a condition. The paper
\cite{BoegDuzSchev:2011}  gives the full parabolic analogue of this result for
the flow \eqref{time-dep-Dirichlet-probl}, which yields global solutions under a large variety of conditions.
Moreover, the same isoperimetric condition allows to
analyze the asymptotic behavior  as $t\to\infty$, to be precise,
global solutions sub-converge as $t\to\infty$ to solutions of
the stationary Dirichlet  problem for the $H$-surface equation.
Under the Dirichlet boundary condition, these solutions 
of course can not be expected to be conformal
and therefore they admit no differential geometric meaning. For this
reason we are here interested 
in the flow \eqref{time-dep-Plateau-probl} under the geometrically
more natural Plateau boundary condition. 
We prove that the free boundary condition
\eqref{time-dep-Plateau-probl}$_3$ 
allows the surfaces $u(\cdot,t)$ to adjust themselves conformally as
$t\to\infty$, so that global solutions to \eqref{time-dep-Plateau-probl}
sub-converge to classical \textbf{conformal} solutions of the Plateau problem,
which actually parametrize immersed surfaces with prescribed mean
curvature.

\subsection{Formulation of the problem and results}
The aim of the present paper is to give a suitable meaning
to the heat flow associated to the classical  Plateau
problem \eqref{Plateau-probl}. In order to formulate this evolution
problem, we need to explain to a certain extent some notations from the classical theory. Let $\Gamma\subset\R^3$ be a Jordan curve such that  a
$C^3$-parametrization $\gamma\colon S^1\to\Gamma$ exists.
By $\,\widehat\gamma\colon\R\to\Gamma$, we denote the corresponding map on
the universal cover $\R$ of $S^1$, defined by 
$\widehat\gamma(\varphi)=\gamma(e^{i\varphi})$. Associated
with the Jordan curve $\Gamma$ we consider the following class
of mappings from the unit disk $B\subset \R^2$ into $\R^3$ defined by
\begin{equation*}
   \S:=\left\{u\in W^{1,2}(B,\R^3)\,\left|\,
   \begin{array}{c}
      u|_{\partial B}\colon \partial B\to\Gamma\mbox{ is a continuous, }\\[0.5ex]
    \mbox{weakly monotone parametrization of }\Gamma
   \end{array}
   \right.\right\}.     
\end{equation*}
The monotonicity condition on $u|_{\partial B}$ means precisely that $u\big|_{\partial B}$
is the uniform limit of orientation preserving homeomorphisms from
$\partial B$ onto $\Gamma$.
This class allows the action of the non-compact M\"obius group of conformal
diffeomorphisms of the disc into itself, i.e. with $u\in\S$ we have
$u\circ g\in \S$ whenever $g\in\mathcal{G}$, where $\mathcal G$ denotes the M\"obius group defined by 
\begin{equation*}
  \mathcal{G}=\Big\{g\colon w\mapsto e^{i\varphi}\frac{a+w}{1+\overline a
    w}\,:\, a\in\C,\ |a|<1,\ \varphi\in\R\Big\}.
\end{equation*}
In order to factor out the action of the M\"obius group it is standard
to impose a \textbf{three-point-condition}.
More precisely, we fix three
arbitrary distinct points $P_1,P_2,P_3\in\partial B$ -- for
convenience we may choose $P_k=e^{i \Theta_k}$ with $\Theta_k:=\frac{2\pi k}{3}$ for $k=1,2,3$ -- and
three distinct points $Q_1,Q_2,Q_3\in\Gamma$ and impose the condition
$u(P_k)=Q_k$ for $k=1,2,3$. The corresponding function space 
we denote
by
\begin{equation}\label{def:CGamma}
  \Ss:=\big\{u\in\S\,:\, u(P_k)=Q_k\mbox{ for $k=1,2,3$}\big\}.
\end{equation}
We note that $u\in W^{1,2}(B,\R^3)$ is contained in $\Ss$ if and only
if $u(e^{i\vartheta})=\widehat\gamma(\varphi(\vartheta))$ for all
$\vartheta\in\R$ and some function
$\varphi\colon\R\to\R$ that is contained in the space
\begin{equation*}
  \mathcal{T}^*(\Gamma):=\left\{\varphi\in C^0\cap
  W^{\frac12,2}(\R)\,\left|\,
  \begin{array}{c}
    \varphi\mbox{ is non-decreasing, }
  \varphi(\cdot+2\pi)=\varphi+2\pi\\[0.5ex]
  \mbox{and }\widehat\gamma(\varphi(\Theta_k))= Q_k\mbox{ for }\ k=1,2,3
\end{array}\right.\right\},
\end{equation*}
where here, $\Theta_k\in[0,2\pi)$ is characterized by
$e^{i\Theta_k}=P_k$ for $k=1,2,3$.
We can always achieve
$\mathcal{T}^*(\Gamma)\neq\varnothing$ by 
changing the orientation of the parametrization $\gamma\colon S^1\to\Gamma$
if necessary. 
The space of admissible testing functions for a given surface $u\in\Ss$ with $u(e^{i\vartheta})=\widehat\gamma(\varphi(\vartheta))$, is then given by
\begin{equation*}
  T_{u}\mathcal{S}^\ast:=\Big\{w\in L^\infty\cap W^{1,2}(B,\R^3)\,:\,
                      w(e^{i\vartheta})=\widehat\gamma'(\varphi)(\psi-\varphi)
                      \mbox{ for some }\psi\in\mathcal{T}^\ast(\Gamma)\Big\}\,.
\end{equation*}
We note that $T_{u}\mathcal{S}^\ast$ is a convex cone. The significance of this set becomes clear from 
Lemma~\ref{lem:rand-variationen} which ensures that  a given $w\in T_u\mathcal{S}^\ast$
is the variation vector field of an admissible variation 
of $u$; here admissible has to be understood in the sense
that the variation is contained in $\Ss$ along the variation. 
The class $\Ss$ also allows so-called {\bf inner variations}.
These variations are generated by vector fields $\eta$ belonging
to the class $\mathcal C^\ast (B)$ (cf.~(\ref{inner-variations})), the class  of all
$C^1$-vector fields $\eta$ on 
$\overline B$ which are tangential along $\partial B$ and vanish at the three points $P_1, P_2$ and $P_3$. 

Finally, for a given closed, convex 
obstacle $A\subset\R^3$ with $\Gamma\subset A^\circ$, we define
\begin{equation}
  \label{def:CGA}
  \SGA:=\big\{u\in\Ss : u(x)\in A\mbox{ for a.e. }x\in B\big\}.
\end{equation}
As already mentioned before our goal is to define a geometric flow
associated with the classical Plateau problem \eqref{Plateau-probl}
for surfaces with prescribed mean curvature function $H\colon A\to \R$
that is continuous and bounded in $A$. 
This geometric flow should allow the existence of global (weak) solutions which at least sub-converge
asymptotically as $t\to\infty$ to solutions of the stationary 
Plateau problem \eqref{Plateau-probl}. Our definition of this flow is as follows: For a given obstacle $A$,
a given Jordan curve $\Gamma$ contained in $A$ and an initial datum $u_o\in\SGA$ we are looking for a global
weak solution 
\begin{equation}\label{def-weak}
	u\in L^\infty (0,\infty; W^{1,2}(B,\R^3))
	\quad
	\mbox{with}
	\quad
	\partial_tu\in L^2 (0,\infty; L^{2}(B,\R^3))
\end{equation}
to the following {\bf evolutionary Plateau problem for $H$-surfaces}:
%\begin{align}\label{Plateau}
%  \left\{
%  \begin{array}{c l}
%     \partial_tu -\Delta u
%     = -2(H\circ u)D_1u\times D_2u &\mbox{weakly in }B^2\times(0,\infty)\\[7pt]
%     \int_{B^2\times\{t\}} [Du\cdot Dw+\Delta u\cdot w]\,dx \ge 0
%      &\mbox{for a.e. $t\in(0,\infty)$ and all $w\in
%        L^\infty\cap 
%        T_{u(\cdot,t)} \mathcal{S}^\ast$ }\\[7pt]
%     u(\cdot,t)\in\SGA&\mbox{for a.e. }t\in(0,\infty),\\[7pt]
%    u(\cdot,0)=u_0&\mbox{in }B^2.
%  \end{array}\right.
%\end{align}
\begin{align}\label{Plateau}
  \left\{
  \begin{array}{cl}
    \partial_tu -\Delta u
     = -2(H\circ u)D_1u\times D_2u 
     &
     \begin{array}{l}
     \mbox{weakly in $B\times (0,\infty)$,}
     \end{array}
     \\[7pt]
     u(\cdot,0)=u_o
     &
     \begin{array}{l}
     \mbox{in $B$,}
     \end{array}
     \\[7pt]
     u(\cdot,t)\in\SGA
     &
     \begin{array}{l}
     \mbox{for a.e. $t\in(0,\infty)$,}
      \end{array}
     \\[7pt]
     \displaystyle
     \int_{B} [Du(\cdot ,t)\cdot Dw+\Delta u(\cdot ,t)\cdot w]
     \,dx \ge 0
     &
     \begin{array}{l}
     \mbox{for a.e. $t\in(0,\infty)$ and all}\\
     \mbox{$w\in T_{u(\cdot,t)} 
     \mathcal{S}^\ast$ }
     \end{array}
     \\[7pt]
     \displaystyle
     \int_{B} {\rm Re}\big(\frak h[u(\cdot ,t)]
     \overline \partial\eta\big)
     +
     \big(\partial_tu\cdot Du\big)(\cdot ,t)\eta\, dx
     =0
     &\begin{array}{l}
     \mbox{for a.e. $t\in(0,\infty)$ and all}\\
     \mbox{$\eta\in\mathcal{C}^\ast (B)$. }
     \end{array}
    \end{array}
\right.
\end{align}
In \eqref{Plateau}$_5$ we have identified $\R^2$ with $\mathbb C$ and
abbreviated $\overline\partial\eta :=\tfrac12 (D_1\eta +\mathbf i
D_2\eta)$. Further, for a map $w\in W^{1,2}(B,\R^3)$ we use the abbreviation
\begin{equation}\label{quadratic-diff}
    \frak h [w]
    :=
    |D_1w|^2-|D_2w|^2 -2\mathbf i D_1w\cdot D_2w.
\end{equation}
We point out that for sufficiently regular $u$,
by the Gauss-Green formula
the inequality
\eqref{Plateau}$_4$ is equivalent to
\begin{equation}\label{strong-bdry-cond}
 \int_{\partial B} \tfrac{\partial u}{\partial r}(x,t)
w(x,t)\,d\mathcal H^1x \ge 0
 \qquad\mbox{for all $w\in T_{u(\cdot,t)} \mathcal S^\ast$}.
\end{equation}
We therefore  interpret \eqref{Plateau}$_4$ as a weak formulation of
(\ref{strong-bdry-cond}). It is well defined in 
our situation because $\Delta u(\cdot ,t)\in L^1(B)$ 
for a.e. $t$ as a consequence of
\eqref{def-weak} and (\ref{Plateau})$_1$, while (\ref{strong-bdry-cond}) 
can not be used in the general case 
since the radial derivative $\tfrac{\partial
 u}{\partial r}$ might not be well defined on 
$\partial B$. With this respect \eqref{Plateau}$_4$ can be interpreted
as a weak form of the Neumann boundary condition
\eqref{strong-bdry-cond} and henceforth we shall denote 
\eqref{Plateau}$_4$ 
{\bf weak Neumann boundary condition}. The last property 
\eqref{Plateau}$_5$ can be viewed as a type of conformality condition.
For a stationary solution, i.e. a time independent solution,
\eqref{Plateau}$_5$ yields the conformality in $B$, that is we have $\frak h[u]\equiv 0$ in $B$ which is equivalent to \eqref{Plateau-probl}$_3$. For a weak solution of the evolutionary Plateau problem, starting with
an initial datum $u_o$, we can not expect the solution to be conformal
for every time slice $t>0$. However, the asymptotic behaviour 
as $t\to\infty$ should enforce the solution to become
conformal. This can actually be shown for a sequence
of time slices $t_j\to\infty$, since the constructed weak solutions
obey the property $\partial_tu\in L^2(B\times (0,\infty))$.
Therefore, weak solutions of \eqref{Plateau}
sub-converge as $t\to\infty$ asymptotically to
a solution of the classical Plateau problem~(\ref{Plateau-probl}). In this sense
(apart from the three-point-condition which is inherited in
\eqref{Plateau}$_3$) the flow from \eqref{Plateau} is a natural geometric flow
associated to the classical Plateau problem for surfaces of prescribed mean curvature.

We also note that (\ref{Plateau})$_1$ and (\ref{Plateau})$_4$ can
be combined to
\begin{equation}\label{var-inequality}
  \int_{B\times(0,\infty)} Du\cdot Dw
    +\partial_tu\cdot w+2(H\circ u)D_1u\times D_2u\cdot w\,dz \ge 0
\end{equation}
for all $w\in L^\infty(B\times(0,\infty),\R^3)\cap L^2(0,\infty;W^{1,2}(B,\R^3))$ with 
$w(\cdot,t)\in T_{u(\cdot,t)} \mathcal{S}^\ast$ for
a.e. $t\in(0,\infty)$. In order to keep the presentation more
intuitive we prefer to use the $H$-surface system and weak Neumann
type boundary condition separately, 
instead of the unified variational inequality
\eqref{var-inequality}. 

To explain the
main results of the present paper, we start by specifying the hypotheses.
For the obstacle $A\subseteq\R^3$ we suppose that
 \begin{equation}\label{Assum-A}
   \mbox{$A\subseteq\R^3$ is closed, convex, with $C^2$-boundary and bounded principal curvatures.}
\end{equation}
By $\mathcal{H}_{\partial A}(a)$ we denote the minimum of the
 principal curvatures of $\partial A$ in the point $a\in\partial A$,
 taken with respect to the inward pointing unit normal vector.
 Moreover, we assume that
 \begin{equation}
   \label{Assum-H}
   \mbox{$H\colon A\to\R$ is a bounded, continuous function}
 \end{equation}
 and satisfies
 \begin{equation}
   \label{Assum-H-dA}
   |H|\le \HA\qquad\mbox{on }\partial A.
 \end{equation}
 As before, we assume that 
 \begin{equation}
   \label{Assum-Jordan}
	\mbox{$\Gamma\subset A^\circ$ is a Jordan curve parametrized by
   }\gamma\in C^3(S^1,\Gamma).
 \end{equation}
 Furthermore, we suppose that $H$
 satisfies a {\bf spherical isoperimetric condition of type \boldmath $(c,s)$
   on $A$}, for parameters $0<s\le \infty$ and $0<
 c<1$. This means that for every
 spherical $2$-current $T$ (cf. Definition \ref{def-spherical}) with
$\spt T\subseteq A$ and ${\bf M}(T)\le s$ there holds
 \begin{equation}\label{sphrical-isop}
  2\big|\left\langle Q,H\Omega\right\rangle\big| =2\left| \int_A i_QH\Omega\right|\le c\,{\bf M}(T),
 \end{equation}
 where $Q$ denotes the unique integer multiplicity rectifiable
 $3$-current with $\partial Q=T$, ${\bf M}(Q)<\infty$ and $\spt
 Q\subseteq A$. Moreover, $i_Q$ denotes the integer valued multiplicity
 function of $Q$ and $\Omega$ the volume form on $\R^3$.
Finally, for the initial  values $u_o\in \SGA$, we
assume that they satisfy
\begin{equation}
  \label{Assum-uo}
  \int_B|Du_o|^2\dx\le s(1-c).
%  \qquad\mbox{with either \quad}c\le\frac{\sigma-1}{\sigma+1}\quad\mbox{or}\quad\sigma=\infty.
\end{equation}
Note that this is automatically satisfied in the case $s=\infty$.
Under this set of assumptions, we have the following
general existence result.

 \begin{theorem}\label{main}
 Assume that $A\subseteq\R^3$
 and $H\colon A\to\R$ satisfy the assumptions \eqref{Assum-A} --
 \eqref{sphrical-isop} and let $u_o\in \SGA$ be given with \eqref{Assum-uo}.  Then there exists a global weak solution
 $$
 u\in C^{0} \big([0,\infty); L^2(B,A)\big)\cap L^\infty \big((0,\infty); W^{1,2}(B,A)\big)
 $$
 with
 $\partial_tu\in L^2(B\times(0,\infty),\R^3)$ to  \eqref{Plateau}. Moreover, the initial datum is achieved as usual in the $L^2$-sense, that is
 $\lim_{t\downarrow 0}\| u(\cdot ,t)-u_o\|_{L^2(B,\R^3)}=0$.
 \end{theorem}
 
With respect to the asymptotic behaviour as $t\to\infty$ we have the following
 
  \begin{theorem}\label{main2}
 Under the assumptions of Theorem \ref{main} there exist a map $u_*\in \SGA$ and a sequence $t_j\to\infty$
 such that $u(\cdot ,t_j)\rightharpoonup u_*$ weakly in $W^{1,2}(B,\R^3)$ and such that $u_*$ is a solution of the Plateau
 problem for surfaces of prescribed mean curvature
 \begin{equation}\label{Plateau-probl-infty}
    \left\{
    \begin{array}{c}
     \Delta u_* =2 (H\circ u_*) D_1 u_*\times D_2 u_*\quad\mbox{weakly in B,}\\[5pt]
     u_* \in \SGA,\\[5pt]
     |D_1u_*|^2 -|D_2u_*|^2 =0=D_1u_* \cdot
     D_2u_*\quad\mbox{in $B$.}
     \end{array}
     \right.
 \end{equation}
The solution satisfies
$u_*\in C^0(\overline B,\R^3)\cap C^{1,\alpha}_{\rm loc}(B,\R^3)
$, and if $H$ is H\"older continuous, then
$u_*\in C^{2,\alpha}(\overline B,\R^3)$ and $u_*$ is a
 classical solution of \eqref{Plateau-probl}.
 \end{theorem}
 
\subsection{Technical aspects of the proofs}
In the present section, we briefly comment on the several different aspects
that are joined to the existence proof.

\subsubsection*{Variational formulation via Geometric Measure Theory}
The starting point 
of our considerations is the observation that the geometric
flow \eqref{Plateau} admits a variational structure. This means
that  $u\mapsto -\Delta u+2(H\circ u)D_1u\times D_2u$
can be interpreted as the Euler-Lagrange operator
of the energy functional $\mathbf E_H(v):= \mathbf D(u)+2\mathbf V_H(u,u_o)$ defined on the class $\SGA$. Here, $\mathbf V_H(u,u_o)$
measures the oriented volume (taken with multiplicities as in \eqref{sphrical-isop}) enclosed
by the surfaces $u$ and $u_o$ and weighted with respect to
the prescribed mean curvature function $H\colon A\to \R$.
The definition of the volume term can be made rigorous
by methods from Geometric Measure Theory, and at this stage we follow
ideas introduced by Steffen \cite{Steffen:1976, Steffen:1976-2}.
Minimizers of such energy functionals are in particular stationary
with respect to inner variations, i.e.
$\frac\partial{\partial s}\big|_{s=0}\mathbf E_H(u\circ\phi_s)=0$ whenever
$\phi_s$ is the flow generated by a vector field $\eta\in \mathcal
C^\ast (B)$. Since the volume term is invariant under inner
transformations, minimizers of $\mathbf{E}_H$ satisfy $\partial
\mathbf D(u;\eta)=\int_B{\rm Re}\big(\frak
h[u]\overline\partial\eta\big)\dx=0$, which leads to  conformal
solutions. The conformality is geometrically significant since it implies that the minimizers
parametrize an immersed surface with mean curvature given by the
prescribed function $H$. 
Finally,
variations which take into account the possibility to vary
minimizers along $\partial B$ tangential to $\Gamma$ give rise to
a weak Neumann type boundary condition as \eqref{Plateau}$_4$.
Therefore, \eqref{Plateau} can be interpreted as the gradient
flow associated with the classical Plateau problem \eqref{Plateau-probl}.
For the construction of solutions to this gradient flow, we use the following time
discretization approach. 

\subsubsection*{Time discretization -- Rothe's method}
This approach has been successfully carried out for the construction
of weak solutions for the harmonic map heat flow by Haga \& Hoshino \&
Kikuchi \cite{HagaHshinoKikuchi:2004} and Kikuchi \cite{Kikuchi:1991}
(see also Moser \cite{Moser:2009} for an application of the technique
to the bi-harmonic heat flow). 
 For a fixed step size $h>0$ we sub-divide $(0,\infty)$ into $((j-1)h, jh]$ for $j\in\N$. We fix a closed, convex subset
 $A\subseteq \R^3$ and a datum $u_o\in \SGA$. For $j=0$ we let $u_{o,h}:=u_o$. Then, for $j\in\N$
 we recursively define  time-discretized energy functionals according to
 $$
    {\bf F}_{j,h}(w):={\bf D}(w)+{\bf V}_H(w,u_o)+\tfrac{1}{2h}\int_B\big|w-u_{j-1,h}\big|^2\, dx.
 $$
We construct $u_{j,h}$ as
a minimizer of the functional ${\bf F}_{j,h}$ in a fixed sub-class of
$\SGA$, which may be defined for example by
a further energy restriction such as $\mathbf D(u)\le s$. 
At this stage, we impose a spherical isoperimetric condition on the prescribed
mean curvature function $H\colon A\to \R$
to ensure the existence of an ${\bf F}_{j,h}$-minimizer. 
Moreover, since the leading terms ${\bf D}(w)$ and ${\bf V}_H(w,u_o)$
of the energy functional are conformally invariant, we
impose the classical three-point-condition of the type $u(P_k)=Q_k$,
$k=1,2,3$ for three points $P_k\in\partial B$,  to factor out the action of
the M\"obius group in the leading terms of the functional. In this setting, we can ensure the existence of
minimizers in $\SGA$ to $\mathbf{F}_{j,h}$ by modifying the methods developed
in \cite{Steffen:1976} (see also \cite{DuzaarSteffen:1999,
  BoegDuzSchev:2011}). Having the sequence
of $\mathbf F_{j,h}$-minimizers $u_{j,h}$ at hand one defines an approximative solution to the Plateau $H$-flow from \eqref{Plateau}
by letting
$$
    \mbox{$u_h(x,t):= u_{j,h}(x)\quad$ for all $x\in B$ and $j\in\N$
      with $t\in ((j-1)h,jh]$.}
$$
The constructed minimizers $u_{j,h}$ are actually H\"older continuous in the interior of $B$ and  continuous up to the boundary
$\partial B$. This follows by using the ${\bf F}_{j,h}$-minimality along the lines of an old device of  Morrey  based on the harmonic replacement and comparison of energies. The lower order $L^2$-term, i.e. the term playing the role of the discrete time derivative,  is at
this stage harmless. This term has however a certain draw back. It is
responsible for the fact that the H\"older estimates can not be
achieved uniformly in $h$ when $h\downarrow 0$.

The obstacle condition $u_{j,h}(B)\subseteq A$ and the possible energy restriction of the form ${\bf D}(u_{j,h})\le s$ in principle only allow to derive certain variational inequalities for minimizers. However, if one imposes a condition relating the absolute value of the prescribed mean curvature function $H$ along the boundary $\partial A$
of the obstacle to the principle curvatures $\mathcal H_{\partial A}$
of $\partial A$, then by some sort of maximum principle the minimizers
$u_{j,h}$ fulfill the Euler-Lagrange system
associated with the functional ${\bf F}_{j,h}$. Formulated in terms of
the function $u_h$, this system reads as
% $$
% 	\frac{u_{j,h}-u_{j-1,h}}{h}-\Delta u_{j,h} +2(H\circ u_{j,h})
% 	D_1u_{j,h}\times D_2u_{j,h}=0
% $$
% holds true weakly on $B$. This is also the point where the convexity of the obstacle enters. Furthermore, variations of  minimizers tangential to $\Gamma$ along $\partial B$ yields the weak Neumann type boundary condition, in the sense that there holds
% $$
% 	\int_{B} [Du_{j,h}\cdot Dw+\Delta u_{j,h}\cdot w]
%      \,dx \ge 0,
% $$
% whenever $w\in T_{u_{j,h}}\mathcal S^\ast$. Finally, inner variations lead to the perturbed conformality in the sense that
% $$
% 	\int_{B} {\rm Re}\big(\frak h[u_{j,h}]
%      \overline \partial\eta\big)
%      +
%      Du\frac{u_{j,h}-u_{j-1,h}}{h}\,\eta\, dx
%      =0
% $$
% holds true for any $\eta\in \mathcal C^\ast (B)$. 
\begin{equation}\label{Euler-h}
    \Delta_t^hu_h -\Delta u_h +2(H\circ u_h)D_1u_h\times D_2 u_h=0\quad \mbox{weakly on $B\times (0,\infty)$}
\end{equation}
if we abbreviate 
$$
    \Delta_t^h w(x,t):=\frac{w(x,t)-w(x,t-h)}{h}
$$
for the finite difference quotient in time. We mention that $u_h(\cdot
,t)\in\SGA$ for any $t\ge 0$, by construction.
Moreover, varying the minimizers $u_{j,h}$ tangentially to $\Gamma$
along $\partial B$ yields the weak Neumann type boundary condition for the map $u_h$:
\begin{equation}\label{Neumann-h}
	0
	\le 
	\int_{B\times \{t\}} 
	\big[Du_{h}\cdot D w+\Delta u_{h}\cdot  w\big]\,dx
\end{equation}
for any $w\in T_{u_{h}(\cdot ,t)} \mathcal S^\ast$ and $t>0$.
Finally, inner variations lead to some kind of perturbed conformality
condition, more precisely  
\begin{align}\label{conform-h}
	\int_{B\times \{t\}}\big[
	{\rm Re} \big( \frak h[u_{h}] \,\overline\partial\eta\big)
	+
	\Delta_t^hu_h\cdot Du_{h} \eta\big]\, dx=0
\end{align}
whenever $\eta\in\mathcal{C}^\ast(B)$ and $t>0$. The combination of 
\eqref{Euler-h}, \eqref{Neumann-h} and \eqref{conform-h} means 
that $u_h$ solves the time discretized Plateau flow for surfaces of
prescibed mean curvature, and  the main effort of the paper is 
to show that the constructed solutions $u_h$ actually converge to a
solution of \eqref{Plateau} as $h\downarrow 0$. 

\subsubsection*{An $\varepsilon$-regularity result}

Due to the non-linear character of the time discrete
$H$-flow system \eqref{Euler-h}, the (non-linear) Plateau-type boundary condition appearing
in \eqref{Neumann-h} and the perturbed conformality condition
\eqref{conform-h}, the analysis of the convergence 
is a non-trivial task and needs several technically involved tools. 
The major obstructions stem from three facts. 
Firstly, the non-linear $H$-term, i.e. $2(H\circ w) D_1w\times D_2w$,
is not continuous with respect to weak convergence in $W^{1,2}$. 
Secondly, the weak boundary condition (\ref{Neumann-h}) associated
with the 
Plateau problem contains a hidden non-linearity in the constraint
$w\in T_{u_h(\cdot,t)}\mathcal{S}^\ast$ and therefore is also not
compatible with weak convergence. 
Finally, the non-linear term $\mathfrak h [u_h]\overline \partial\eta$
also causes problems in the limit $h\downarrow 0$. For these reasons,
one would need at least uniform local $W^{2,2}$-estimates up to the
boundary in order to achieve local strong convergence in $W^{1,2}$.
 
However, the approximation scheme only yields uniform $L^\infty$--\,$W^{1,2}$-bounds for $u_h$ and $L^2$-bounds
for the discrete time derivative $\Delta_t^hu_h$. Therefore, one can only conclude that
a subsequence $u_{h_i}$ converges in $C^0$--\,$L^2$ and weakly* in
$L^\infty$--\,$W^{1,2}$ 
to a limit map $u\in L^\infty$--\,$W^{1,2}\cap C^{0,\frac12}$--\,$L^2$, and furthermore
that the weak limit admits a time derivative $\partial_tu\in L^2$
and that $\Delta_t^{h_i}u_{h_i}$ converges weakly to $\partial_tu$ in
$L^2$. These convergence properties
are not sufficient, though, to pass to the limit neither in the non-linear $H$-term
$H(u_h)D_1u_h\times D_2u_h$, nor in the boundary
condition~(\ref{Neumann-h}), nor 
in the non-linear term $\frak h [u_h]$. 
For the treatment of these terms, we employ ideas used by Moser for the construction of a biharmonic map
heat flow \cite{Moser:2009}. These methods have been successfully
adapted in \cite{BoegDuzSchev:2011},
 where a related $H$-surface flow with a Dirichlet boundary condition on the lateral boundary has been studied (see also \cite{BoegDuzSchev:2011.2} for an application to the heat flow for $n$-harmonic maps).

First of all one argues slice-wise, that is for a fixed time $t$. Then the sequence $u_{h_i}(\cdot ,t)$ is composed by different minimizers,
all of them in $\SGA$, and each of them satisfies \eqref{Euler-h}, \eqref{Neumann-h} and \eqref{conform-h} on the fixed time slice.
In particular, the maps $u_{h_i}$ satisfy the three-point-condition
and therefore are equicontinuous on $\partial B$. The idea now is to establish some sort of $\varepsilon$-regularity result. By this we mean an assertion of the form
\begin{equation}\label{eps-o}
	\sup_{i\in \N} \int_{B_\rho^+(x_o)} |Du_{h_i}|^2\, dx
	<\varepsilon
	\quad
	\Longrightarrow
	\quad
	\sup_{i\in \N} \| u_{h_i}\|_{W^{2,2}(B_{\rho/2}^+(x_o))}
	<\infty,
\end{equation}
where $\epsilon >0$ is a universal constant which can be determined in dependence on the data.
Here $B_\rho^+(x_o)$ denotes either an interior disk $B_\rho (x_o)
\subset B$ or a half-disk centered at a boundary point $x_o\in\partial B$.
In any case we only consider disks such that $B_\rho^+(x_o)\cap
\{ P_1,P_2,P_3\}=\varnothing$. The proof of statement \eqref{eps-o} is the core
of our construction of weak solutions and consists of two steps, which
we summarize next. 

\subsubsection*{A-priori $W^{2,2}$-estimates up to the Plateau boundary}
The first step of the proof of~(\ref{eps-o}) consists of proving
apriori estimates under additional regularity assumptions. 
We establish them for general solutions which satisfy $\Delta u=F$ in
$B$, together with a Plateau type boundary condition and the weak
Neumann type condition \eqref{Neumann-h}. Here, we need to consider right-hand sides of critical growth $|F|\le C (|Du|^2+ f)$ for some
$f\in L^2(B)$. This is the reason why we can establish
$W^{2,2}$-estimates in a first step only under the additional
assumption $|Du|\in L^4_{\rm loc}$, which implies $F\in L^2_{\rm loc}$. 
In the interior, the local $W^{2,2}$ estimate~(\ref{eps-o}) then follows via the
difference quotient technique and an application of the
Gagliardo-Nirenberg interpolation inequality in a standard
way. However, the boundary version of
this result is much more involved. Here we need local versions of
global $W^{2,2}$ estimates which have been
derived for minimal surfaces with a Plateau type boundary condition by
Struwe in \cite{Struwe:1988-2, ImbuschStruwe:1999} (see als
\cite{ChangLiu:2003, ChangLiu:2003.2, ChangLiu:2004}).
The local $W^{2,2}$-estimate
follows by a technically involved angular difference quotient
argument. For its implementation, additionally to $Du\in
L^4(B_\rho^+(x_o))$ we also need to assume that the 
oscillation of $u$ on $B_\rho^+(x_o)$ is small enough. This is needed
in order to ensure that the image of $u$ is contained in a tubular
neighborhood of $\Gamma$, so that the nearest-point retraction onto
$\Gamma$ is well-defined. In this situation, it is possible to adapt
the standard variations that are used in the difference quotient
argument in such a way that they are admissible under the Plateau
boundary constraint. The additional assumption of small oscillation
can be established by a Courant-Lebesgue type argument, once 
the local interior $W^{2,2}$-estimate is known. This is a consequence
of an argument by Hildebrandt \& Kaul \cite{HildebrandtKaul:1972} and
has been exploited before in the situation 
of a free boundary condition in \cite{Scheven:2006}.
Therefore it only remains to establish the local
$W^{1,4}$-estimate at the boundary in order to justify the application
of the above $W^{2,2}$-estimates to the time-discretized $H$-surface flow.

\subsubsection*{Calder\'on-Zygmund estimates up to the boundary for systems with critical
growth}
Here we use a Calder\'on-Zygmund type argument for solutions of systems of the type
$\Delta u=F$ which satisfy a Plateau-type boundary condition, where 
the right-hand side has critical growth as above. Our arguments are inspired by 
methods which have been developed for elliptic and parabolic
$p$-Laplacean type systems by Acerbi-Mingione \cite{Acerbi-Mingione:2007} (see also the paper by Caffarelli-Peral
\cite{CaffarelliPeral:1998}). 
In order to deal with the critical growth of the inhomogeneity, we
again need a small oscillation assumption for
the derivation of suitable comparison estimates. The small oscillation
 is guaranteed by the continuity of the minimizers $u_{j,h}$.
As local comparison problems, we consider the system $\Delta w=0$ on
$B_\rho^+(x_o)$, together with the boundary condition $w=u$ on $B\cap \partial B_\rho(x_o)$
and a Plateau type condition on $\partial B\cap B_\rho(x_o)$. For such solutions local $W^{2,2}$-estimates hold, which allow an improvement of integrability of the gradient of $u$ on its level sets.
This improvement yields a quantitative Calder\'on-Zygmund estimate
of the form
\begin{equation*}
    \int_{B_{\rho/4}^+(x_o)}|Du|^4\dx
    \le  
    \frac C{\rho^2} 
    \bigg(\int_{B_{\rho/2}^+(x_o)}|Du|^2\dx\bigg)^2
    +
    C \int_{B_{\rho/2}^+(x_o)}|f|^2\dx,
\end{equation*}
for some universal constant $C$,
provided $\osc_{B_\rho^+(x_o)}u$ is small enough and $\|Du\|_{L^2}$ is
bounded from above. For our applications however, we are only
interested in the qualitative regularity $u\in
W^{1,4}(B_{\rho/2}^+(x_o),\R^3)$, which enables us to apply the
a-priori $W^{2,2}$-estimates from above and thereby to establish the
$\varepsilon$-regularity result~(\ref{eps-o}). 

\subsubsection*{Concentration compactness arguments}
Next, we apply \eqref{eps-o} to the sequence $(u_{h_i})$ on a fixed
time slice $t>0$.
Since the smallness assumption on the left-hand side of~(\ref{eps-o})
is satisfied for all but finitely many points $x_o\in \overline B\setminus \{P_1,P_2,P_3\}$ for a
sufficiently small radius $\rho(x_o)>0$, we infer uniform $W^{2,2}$-estimates and
therefore strong $W^{1,q}$-convergence for any $q\ge 1$ away from finitely many
concentration points. Since anyway we have to deal with finitely many 
exceptional points, we can also exclude the points $P_1,P_2,P_3$ from
the three-point-condition from our considerations. 
The local strong convergence suffices to conclude that the non-linear terms 
in \eqref{Euler-h}, \eqref{Neumann-h} and \eqref{conform-h} locally
converge 
to the corresponding terms for the limit map $u$.
Assuming that $\Delta_t^hu_h\to -f$
weakly in $L^2$, we infer that $u(\cdot,t)$ solves 
\eqref{Euler-h}, \eqref{Neumann-h} and \eqref{conform-h} away from finitely
many singular points if we replace 
$u_h$ by $u$ and $\Delta_t^hu_h$ by $-f$ in all three formulae. 
The finite singular set
obviously is a set of vanishing $W^{1,2}$-capacity, and this enables
us to deduce that $u(\cdot ,t)$ is a weak solution to \eqref{Euler-h},
\eqref{Neumann-h} and \eqref{conform-h} on all of $B$. 
It is worth to note, that in the capacity argument for the perturbed 
conformality relation~(\ref{conform-h}) we have to utilize the
regularity result by Rivi\`ere \cite{Riviere:2007} for the $H$-surface
equation and then the Calder\'on-Zygmund
estimate mentioned above in order to have $\frak h[u_h]\in L^2_{\rm loc}$.
To conclude that $u$ actually is a weak solution of 
\eqref{Plateau} we need to have the identification $f=-\partial_tu$. 
This assertion can be achieved along the replacement argument by Moser \cite{Moser:2009}.

\subsubsection*{Asymptotics as $t\to\infty$: Convergence to a
  conformal solution}
The strategy for the proof of the asymptotic behavior is similar,
i.e. a concentration compactness argument combined with
 a capacity argument. The only major difference occurs
since we can choose the time slices $t_i\to\infty$ in such a way that
$\int_{B}|\partial_tu|^2(\cdot, t_i)\, dx\to 0$ as $i\to\infty$. 
Therefore, for the weak limit map $u_\ast:=\lim_{i\to\infty}u(\cdot ,t_i)$ the weak conformality
condition~(\ref{conform-h}) becomes
\begin{equation*}
        \partial\mathbf{D}(u_\ast;\eta)=
	\int_B
	{\rm Re} \big( \frak h[u_\ast] \,\overline\partial\eta\big)
	\, dx=0
\end{equation*}
for all $\eta\in\mathcal{C}^\ast(B)$.  
It is well known from the theory of $H$-surfaces that this identity implies the
conformality of the limit map $u_\ast$. Moreover, the regularity result by
Rivi\`ere \cite{Riviere:2007} combined with classical arguments 
yield that $u_\ast$ is regular up to the
boundary. As a result, the flow subconverges as $t\to\infty$ to a classical solution of
the Plateau problem for $H$-surfaces, i.e. to a map that parametrizes an
immersed surface with prescribed mean curvature and
boundary contour given by $\Gamma$.

\subsection{Applications}
\label{sec:Existence-appl}  
%\subsection{Existence of solutions to the $H$-flow with a Plateau boundary condition}  
In this section we give some sufficient conditions ensuring the existence
of a weak solution to the heat flow for surfaces with prescribed mean curvature satisfying a Plateau boundary condition \eqref{Plateau}.
They follow from Theorem \ref{main} and known criteria guaranteeing  the validity of an
isoperimetric condition, cf. \cite{Steffen:1976,Steffen:1976-2,DuzaarSteffen:1996,DuzaarSteffen:1999}.

\begin{theorem}\label{thm:application}
Let $A$ be convex and the closure of a $C^2$-domain in
  $\R^3$ and let the principal curvatures of $\partial A$ be bounded. By
$\mathcal H_{\partial A}$ we denote the minimum of the principal curvatures of $\partial A$.
Further, we consider initial data $u_o\in \SGA$ and $H\in L^\infty (A)\cap C^0(A)$. Then each of the following conditions
\begin{equation}\label{H1}
    \sup_A|H|\le \sqrt{\frac{2\pi}{3\mathbf{D}(u_o)}}
\end{equation}

\begin{equation}\label{H2}
    A\subseteq B_R\quad\mbox{and}\quad\int_{\{\xi \in A: |H(\xi )|\ge \frac{3}{2R}\}} |H|^3\, dx <\tfrac{9\pi}{2}
\end{equation}

\begin{equation}\label{H3}
    \sup_A|H|< \tfrac{3}{2}\,\sqrt[3]{\frac{4\pi}{3\mathcal L^3 (A)}}
\end{equation}

\begin{equation}\label{H4}
   \mbox{for some $c< 1$ suppose that $\mathcal L^3\{ a\in A\colon |H(a)|\ge\tau\}\le c\tfrac{4\pi}{3} \tau^{-3}$}
   \mbox{for any $\tau >0$} %assume that either $c<1$ or $\mathcal L^3 ( A )<\infty$.}
\end{equation}
together with the curvature assumption
\begin{equation}\label{Rand}
    |H(a)|\le \mathcal H_{\partial A}(a)\quad\mbox{for $a\in\partial A$,}
\end{equation}
ensure the existence of a weak solution of \eqref{Plateau} with the properties described in Theorem~\ref{main}.
The same conditions guarantee the sub-convergence of $u(\cdot ,t)$ to a solution of the Plateau problem \eqref{Plateau-probl}.

\end{theorem}

In the case $A\equiv \overline B_R(0)\subseteq \R^3$ the conditions \eqref{H3} and \eqref{Rand} simplify to
$$
    \sup_{B_R(0)}|H|<\tfrac32\tfrac{1}{R},\quad |H(a)|\le\tfrac{1}{R}\quad\mbox{for $a\in \partial B_R(0)$.}
$$
Moreover, in this case we have that \eqref{H2} is fulfilled. 
Consequently, both of the assumptions \eqref{H2} and \eqref{H3}
contain the preceding Hildebrandt type assumptions as special cases and ensure
the existence of a weak solution in the sense of Theorem \ref{main} to the parabolic $H$-flow system \eqref{Plateau}.
Finally, we note that \eqref{H1} can be improved by choosing
$u_o$ to be an area minimizing disk type surface spanned by the Jordan curve $\Gamma$. Then, in \eqref{H1} the Dirichlet energy of $u_o$
equals the minimal area $A_\Gamma$ spanned by $\Gamma$ and the
condition \eqref{H1} turns into
\begin{equation}\label{H1*}
    \sup_A|H|\le \sqrt{\frac{2\pi}{3A_\Gamma}},
\end{equation}
allowing large values of $H$ for Jordan curves with small minimal
area.

\section{Notation and Preliminaries}

In this section we collect the main notation and some results
needed in the proofs later.

\subsection{Notation}
Throughout this article, we write $B$ for the open unit disk in
$\R^2$. More generally, by $B_r(x_o)\subset\R^2$ we denote the open disk with
center $x_o\in\R^2$ and radius $r>0$. Moreover, we use the notation
$
  B_r^+(x_o):=B\cap B_r(x_o)
$
for the interior part of the disk $B_r(x_o)$, which will frequently be
used in particular in the case for a center $x_o\in\partial B$. Furthermore, we use
the abbreviations $S_r^+(x_o):=\partial B_r^+(x_o)\cap \overline B$ and
$I_r(x_o):=B_r^+(x_o)\cap \partial B$, so that
\begin{equation*}
  \partial B_r^+(x_o)= S_r^+(x_o)\cup I_r(x_o).
\end{equation*}
For the \textbf{Dirichlet energy} of a map $u\in W^{1,2}(B,\R^3)$, we
write
\begin{equation*}
  \mathbf{D}(u):=\tfrac12\int_B|Du|^2\dx
  \qquad\mbox{and}\qquad
  \mathbf{D}_G(u):=\tfrac12\int_G|Du|^2\dx
\end{equation*}
for any measurable subset $G\subset B$. 

\subsection{The chord-arc condition}
Any  Jordan curve
$\Gamma$ of class $C^1$ satisfies a $(\delta,M)$-\textbf{chord-arc condition}, i.e. there are constants $\delta>0$ and $M\ge1$ 
such that for each pair of distinct points $p,q\in\Gamma$ we have 
\begin{equation}
  \label{chord-arc}
  \min\big\{L(\Gamma_{p,q}),L(\Gamma_{p,q}^*)\big\}\le M|p-q|
  \qquad\mbox{provided }|p-q|\le\delta,
\end{equation}
where $\Gamma_{p,q},\Gamma_{p,q}^*$ denote the two sub-arcs of $\Gamma$ that
connect $p$ with $q$, and $L(\cdot)$ is their length.
 
\subsection{Admissible variations and variation vector fields}
There are two possible types of variations for a given
surface $u\in \Ss$. 
The first type -- called {\bf outer variations} or {\bf variations
of the dependent variables} -- are those ones performing a deformation of the surface in the ambient space $\R^3$. The initial
vector field of the variation should be a map $w\in T_u\mathcal{S}^\ast$. However, it is not clear at this stage that such a vector field
yields a one-sided variation $u_s\in \Ss$ for values $0\le s\ll 1$
with $u_o=u$. Since we are dealing with surfaces
contained in a closed, convex subset $A\subset\R^3$
we also need a version respecting the obstacle condition
$u_s(B)\subset A$ along the variation. The existence of these kind of variations
is granted by the following

\begin{lemma}\label{lem:rand-variationen}
Let $u\in\Ss$ and $w\in T_u\mathcal{S}^\ast$ be given.
Then there hold:
\begin{enumerate}
\item There exists a one-sided variation $[0,\eps)\ni s\mapsto 
		 u_s\in\Ss$ with
      $u_0=u$ and $\frac{\partial}{\partial s}u_s\big|_{s=0}=w$.
\item If $\Gamma\subset A^\circ$, there exists a one-sided
      variation $[0,\eps)\ni
      s\mapsto u_s\in\SGA$  with $u_o=u$ and
      $\frac{\partial}{\partial s}u_s\big|_{s=0}\in 
      (w+W^{1,2}_0(B,\R^3))\cap C^0(\overline B,\R^3)$.
\end{enumerate}
In both cases, the variations $u_s$ satisfy
$\frac\partial{\partial s}u_s\in L^\infty(B,\R^3)
\cap C^0(\partial B,\R^3)$ for all $s\in[0,\eps)$ and
moreover we have the following bounds:
\begin{equation}\label{uniform-bounds}
    \sup_{0\le s<\eps}
    \Big(
    \|u_s\|_{W^{1,2}(B)}
    +
    \|\tfrac\partial{\partial s}u_s\|_{L^\infty(B)}
    +
    \|\tfrac\partial{\partial s}u_s\|_{C^0(\partial B)}\Big)
    <
    \infty.
\end{equation}
\end{lemma}
\begin{proof}
By $\varphi,\psi\in\mathcal{T}^\ast(\Gamma)$ we denote functions
that are determined by the properties
\begin{equation*}
    u(e^{i\vartheta})=\widehat\gamma(\varphi),
    \quad
    \mbox{respectively by}
    \quad
    w(e^{i\vartheta})=\widehat\gamma'(\varphi)(\psi-\varphi). 
\end{equation*}
For $s\in [0,\eps)$, we define $h_s\in W^{1,2}(B,\R^3)\cap C^0(\overline B,\R^3)$ as the harmonic extension of the boundary
data on $\partial B$ given by $\widehat\gamma(\varphi+
s(\psi-\varphi))$. 
These boundary data are bounded in $W^{\frac12,2}_{\rm loc}(\R)$, uniformly in $s\in[0,\eps)$, and therefore, its harmonic extensions satisfy
\begin{equation}\label{uniformW12}
    \sup_{0\le s<\eps}\|h_s\|_{W^{1,2}(B)}<\infty.
\end{equation}
The derivative $\frac\partial{\partial s}h_s$ is the harmonic extension of the boundary values $\widehat\gamma'(\varphi+
s(\psi-\varphi))(\psi-\varphi)$,
which are uniformly bounded with respect to $s$ in $C^0\cap W^{\frac12 ,2}$. From the maximum principle we thereby infer
\begin{equation}
\label{max-princ}
      \sup_{0\le s<\eps}
      \big\|
      \tfrac\partial{\partial s}h_s
      \big\|_{C^0(\overline B)}
      <\infty.
\end{equation}
In particular, the function $\widetilde w:=
\tfrac{\partial}{\partial s}h_s\big|_{s=0}$ is the harmonic extension of the boundary values given by $\widehat\gamma'(\varphi)(\psi-\varphi)$ and therefore $\widetilde w\in (w+W^{1,2}_0(B,\R^3))
\cap C^0(\overline B,\R^3)$. Next, 
since $\varphi,\psi\in\mathcal{T}^\ast(\Gamma)$, which is a convex
set, and $s\in[0,\epsilon)$, we also have
$\varphi+s(\psi-\varphi)\in\mathcal{T}^\ast(\Gamma)$, which means
$h_s\in\Ss$. Now we distinguish between the two cases stated in the lemma. 

For the proof of (i), we define the variation $u_s$ by  
\begin{equation*}
      u_s:=h_s+s(w-\widetilde w)-(h_0-u).
\end{equation*} 
Since $h_0-u\in W^{1,2}_0(B,\R^3)$ and $w-\widetilde w\in W^{1,2}_0(B,\R^3)$, we conclude $u_s\in\Ss$ for all
$s\in[0,\eps)$, and a straightforward calculation gives
$\frac{\partial}{\partial s}u_s\big|_{s=0}=w$. The claimed
bounds~(\ref{uniform-bounds}) follow from~(\ref{uniformW12}),
(\ref{max-princ}) and $w,\widetilde w\in L^\infty\cap
W^{1,2}(B,\R^3)$ with $w\big|_{\partial B}=
\Tilde w\big|_{\partial B}\in C^0(\partial B,\R^3)$.

In the case of (ii), we choose a cut-off function $\zeta\in
C^\infty(A,[0,1])$ with $\zeta\equiv1$ on a neighborhood of
$\Gamma$ and $\mathrm{spt}\,\zeta\subset A^\circ$, which is possible by our assumption $\Gamma\subset
A^\circ$. Then we define $u_s$ by
\begin{equation*}
  u_s:= u+\zeta(u)\,(h_s-h_0).
\end{equation*}
Because of~(\ref{max-princ}), we can choose $\eps>0$ so 
small that $\|h_s-h_0\|_{L^\infty}<\dist(\spt\zeta,\partial A)$ for
all $s\in[0,\eps)$. Distinguishing between the cases
$u(x)\in\spt\zeta$ and $u(x)\in A\setminus\spt\zeta$, we deduce
$u_s(B)\subset A$ for any $s\in[0,\eps)$.
In order to compute the boundary values of
$\frac\partial{\partial s}\big|_{s=0}u_s$, we note that $u(\partial
B)\subset\Gamma$ and therefore $\zeta (u)\equiv 1$ on
$\partial B$. We conclude $\frac\partial{\partial s}
\big|_{s=0}u_s= \frac\partial{\partial s}\big|_{s=0}h_s$
on $\partial B$ in the sense of
traces and consequently, $\frac\partial{\partial s}\big|_{s=0}u_s=\widetilde w\in w+W^{1,2}_0(B,\R^3)$, as
claimed. Again, the assertion~(\ref{uniform-bounds}) follows
from~(\ref{uniformW12}) and~(\ref{max-princ}). \end{proof}

The second class of variations are the so-called {\bf inner variations} or {\bf variations of the independent variables} , 
which are re-para\-metrizations of the surfaces $u\colon B\to\R^3$
in the domain of definition. For the variation vector fields for
this kind of variations we define the classes
\begin{align}\label{inner-variations}
	\left\{
	\begin{array}{l}
	\mathcal{C}(B)
	:=\big\{
	 \eta\in C^1(\overline B,\R^3):\mbox{$\eta$
    is tangential to $\partial B$ along $\partial B$}
    \big\},\\[5pt]
  \mathcal{C}^\ast(B)
  :=\big\{\eta\in
  \mathcal{C}(B):\eta(P_k)=0\mbox{ for }k=1,2,3.\big\}.
  \end{array}
  \right.
\end{align}
For $\eta
\in \mathcal C^\ast (B)$  we consider the associated flow
$\phi_s$  with $\phi_0={\rm id}$. 
Our assumptions on $\eta$
ensure that $\phi_s(\overline B)\subset \overline B$ and
$\phi_s(P_k)=P_k$ for all $s\in
(-\eps,\eps)$ and $k\in\{1,2,3\}$. Moreover, since $\phi_s$ is an
orientation preserving diffeomorphism for sufficiently small $|s|$,
we know for $u\in\SGA$ that $u\circ\phi_s\big|_{\partial B}$ is a weakly monotone
parametrization of $\Gamma$ and therefore
$u_s:=u\circ\phi_s\in\SGA$. The first
variation
of the Dirichlet integral with respect to such inner variations
is given by
\begin{align}\label{inner-D}
	\partial \mathbf D(u;\eta):=\frac{d}{ds}\Big|_{s =0}{\bf D}(u\circ\varphi_s )
	=
	\int_B {\rm Re}\big( \frak h[u]\overline{\partial}\eta\big)
	\, dx.
\end{align}

The following well-known compactness result is crucial for the
existence of solutions to the Plateau problem. Its proof, which is
based on the Courant-Lebesgue-Lemma, can be found e.g. in \cite[Lemma I.4.3]{Struwe:1988-2}. 
\begin{lemma}\label{lem:compact traces}
	The injection $\Ss\hookrightarrow C^0(\partial B,\R^3)$
	is compact, that is
   bounded subsets of $\Ss$ (with respect to the
   $W^{1,2}$-norm) have equicontinuous traces on $\partial B$. 
\end{lemma}

\subsection{An elementary iteration lemma}
The following standard iteration result will be used in order to re-absorb certain terms.
\begin{lemma}\label{lem:Giaq}
    For $R>0$,
    let $f\colon[r,R]\to[0,\infty)$ be a bounded function with 

    \begin{equation*}
      f(s)\le\vartheta f(t)+\frac A{(t-s)^\alpha}+B
      \qquad\mbox{for all $r\le s<t\le R$},
    \end{equation*}
    for constants $A,B\ge0$, $\alpha>0$
    and $\vartheta\in(0,1)$. Then we have
    \begin{equation*}
      f(r)\le c(\alpha,\vartheta)
     \bigg[\frac A{(R-r)^\alpha}
      +B\bigg].
    \end{equation*}
\end{lemma}

\subsection{An interpolation inequality}

The following Gagliardo-Nirenberg interpolation inequality plays a
central role in the proof of our regularity results and thereby for the construction of  global weak solutions to our parabolic free boundary  problem of Plateau type. 

\begin{lemma}[\cite{Nirenberg:1966}]\label{lem-gag}
Let $B_\rho(x_o)\subset\R^n$ with $0<\varrho\le1$ and
$B_\rho^+(x_o):=B_\rho(x_o)\cap B$. For any parameters
$1\le \sigma,q,r<\infty$ and $\vartheta\in(0,1)$ such that
$ - \frac{n}{\sigma}
\le \vartheta (1 - \frac{n}{q} ) - ( 1- \vartheta ) \frac{n}{r}$, there is a constant
$C=C(n,q,r)$ such that for any $v \in W^{1,q}(B_\rho^+(x_o))$ there holds:
\begin{equation*} 
	\mint_{B_\rho^+(x_o)} \Big|\frac{v}{\rho}\Big|^\sigma \dx
	\le
	C~\bigg(
	\mint_{B_\rho^+(x_o)}
	\Big|\frac{v}{\rho}\Big|^{q} + |Dv|^q\dx
	\bigg)^{\frac{\vartheta \sigma}{q}}
	\bigg(\mint_{B_\varrho^+(x_o)} \Big|\frac{v}{\rho}\Big|^r\dx\bigg)
	^{\frac{(1-\vartheta)\sigma}{r}} .
\end{equation*}
\end{lemma}
For a map $u\in W^{2,2}(B_\rho^+ (x_o),\R^N)$, we may apply this to 
$v=|Du|\in W^{1,2}(B_\rho^+(x_o))$, with the parameters 
$\sigma =4$, $n=2$, $q=r=2$, $\vartheta=\frac12$. This yields, with a universal constant $C$, the following
interpolation estimate:
% \begin{align}\label{appl-gag-nir}
%     &\int_{B_R(x_o)}\eta^4|Du|^4\, dx\\\nonumber
%     &\qquad\le
%     C \int_{B_R(x_o)}\eta^2|D^2u|^2 + (\tfrac1{R^2}\eta^2+|D\eta|^2)|Du|^2\, dx\,
%     \int_{B_R(x_o)}\eta^2|Du|^2\dx\,.
% \end{align}
\begin{equation}\label{appl-gag-nir}
    \int_{B_\rho^+(x_o)}|Du|^4\, dx
    \le
    C \int_{B_\rho^+(x_o)}|D^2u|^2 + \Big|\frac{Du}{\rho}\Big|^2\, dx\,
    \int_{B_\rho^+(x_o)}|Du|^2\dx\,.
\end{equation}
The following lemma is due to Morrey \cite[Lemma 5.4.1]{Morrey:1966}.
\begin{lemma}\label{lem:Morrey}
  Assume that $v\in W^{1,2}_0(\Omega)$ for a domain $\Omega\subset\R^2$ and that $w\in L^1(\Omega)$ satisfies the
Morrey growth condition
\begin{equation*}
    \int_{B_r(y)\cap\Omega}|w|\dx\le C_or^{2\alpha}
\end{equation*}
for all radii $r>0$ and center $y\in\Omega$, with constants $C_o>0$ and $\alpha>0$. Then there holds
$v^2w\in L^1(\Omega)$ with
\begin{equation*}
    \int_{B_r(y)\cap\Omega}|v^2w|\dx
    \le
    C_1C_o|\Omega|^{\alpha/2} r^{\alpha}\int_\Omega|Dv|^2\dx
\end{equation*}
 for all $r>0$, $y\in \Omega$ and a universal constant $C_1=C_1(\alpha)>0$. 
\end{lemma}

\subsection{A generalization of Rivi\`ere's result}
The following result, which is a slight improvement of
Rivi\`ere's fundamental paper \cite{Riviere:2007}, can be retrieved from
\cite{MuellSchikorra:2010}.

\begin{theorem}\label{mod-Riviere}
Let $\Omega\in L^2(B,\frak{so}(m)\otimes\R^2)$ and $f\in L^s(B,\R^m)$ with $s>1$ be given. Then, any weak solution $u\in W^{1,2}(B,\R^m)$ of
\begin{equation*}
	-\Delta u=\Omega\cdot Du +f\quad\mbox{on $B$}
\end{equation*}
is H\"older continuous in $B$ for some H\"older exponent $\alpha \in (0,1)$. Moreover, if $u$ admits a continuous boundary trace $u\big|_{\partial B}$, then $u$ is also continuous up to the boundary, that is
$u\in C^{0,\alpha}(B,\R^m)\cap C^0(\overline B,\R^m)$. 
\end{theorem}
This result is important for our purposes since as noted by Rivi\`ere \cite{Riviere:2007}, the right-hand side of the
$H$-surface equation~(\ref{Plateau-probl})$_1$ can be written in the form
$\Omega\cdot Du$. The difference of the above statement to the one in \cite{Riviere:2007} stems from the fact that
an $L^s$-pertubation with $s>1$ of the critical right-hand side $\Omega
\cdot Du\in L^1$ is considered. This generalization is
necessary for our purposes. 
In our setting $f$ plays the role of the time derivative $\partial_t
u$ which by our construction will be an $L^2$-map on almost every time
slice $B\times\{t\}$.
The statement concerning the boundary 
regularity goes indeed back to \cite[Lemma 3]{HildebrandtKaul:1972}.
Once the interior regularity is established the assumption
of a continuous boundary trace can be used to conclude the regularity up to the boundary by a simple lemma concerning Sobolev maps.

\section{The $H$-volume functional}\label{sec:H-volume}
Here, we briefly recall the definition of the $H$-volume functional
and some of its properties. For a more detailed treatment of the
topic, we refer to \cite{Steffen:1976} or \cite{DuzaarSteffen:1999}.
The definition of the $H$-volume functional that we present here
relies on the theory of currents. The
standard references  are \cite{Federer} and
\cite{Simon}.

\subsection{Definitions}
We write $\mathcal{D}^k(\R^3)$, $k\in\{0,1,2,3\}$, for the space of 
smooth $k$-forms with compact support in $\R^3$. 
A distribution
$T\colon\mathcal{D}^k(\R^3)\to\R$ is called 
\textbf{$k$-current} on $\R^3$. The \textbf{mass} of $T$
is defined by
\begin{equation*}
  \mathbf{M}(T):=\sup\left\{ T(\omega): \omega\in \mathcal
    D^k(\R^3),\; \|\omega\|_\infty\le 1\right\}.
\end{equation*}
The boundary of a $k$-current $T$ is the $(k-1)$-current $\partial T$ given by 
$\partial T(\alpha):=T(d\alpha)$ for $\alpha\in \mathcal D^{k-1}(\R^3)$. A
current $T$ is called \textbf{closed} if $\partial T\equiv0$.
For the definition of the $H$-volume functional, the following
subclass of currents will be crucial. 
\begin{definition}\upshape
  A $k$-current $T$ on $\R^3$ is called an \textbf{integer multiplicity
  rectifiable $k$-current} if it can be represented as
  \begin{equation*}
    T(\omega):=\int_{M}\langle\omega(x),\xi(x)\rangle\,\theta(x)\,d\mathcal{H}^k(x)\qquad
    \mbox{ for all }\omega\in \mathcal{D}^k(\R^3),
  \end{equation*}
  where $\mathcal{H}^k$ denotes the $k$-dimensional
  Hausdorff measure, $M\subset\R^3$ is an
  $\mathcal{H}^k$-measurable, countably $k$-rectifiable subset,
  $\theta\colon M\to\N$ is a locally $\mathcal{H}^k$-integrable
  function and $\xi\colon M\to\bigwedge_k\R^3$ is an
  $\mathcal{H}^k$-measurable function of the form
  $\xi(x)=\tau_1(x)\wedge\ldots\wedge\tau_k(x)$, where 
  $\tau_1(x),\ldots,\tau_k(x)$
  form an orthonormal basis of the approximate tangent space $T_xM$ for
  $\mathcal{H}^k$-a.e. $x\in M$.\hfill $\Box$
\end{definition}

The preceding definition follows the terminology of
Simon \cite{Simon}. In the
language of Federer \cite{Federer}, 
the currents defined above are called locally
rectifiable $k$-currents. Examples of integer multiplicity rectifiable $2$-currents are induced
by any map $u\in W^{1,2}(B,\R^3)$ via integration of $2$-forms 
over the surface $u$ as follows.
\begin{equation*}
    J_u(\omega ):=\int_Bu^\#\omega =\int_B\langle \omega\circ u, D_1u\wedge D_2u\rangle\, dx \qquad \forall\, \omega\in \mathcal D^2(\R^3).
\end{equation*}
The fact that $J_u$ is an integer multiplicity rectifiable $2$-current in $\R^3$ can be checked by a Lusin-type approximation argument as in 
\cite[Sect. 6.6.3]{EvansGariepy}. Moreover, the 
current $J_u$ has finite mass since
$$
 {\bf M}(J_u)
 :=
 \sup\left\{ J_u(\omega): \omega\in \mathcal D^2(\R^3),\,
 \|\omega\|_\infty\le 1\right\}
 \le 
 \int_B\big|D_1u\wedge D_2u\big|\, dx
 \le 
 {\bf D}(u).
$$
If $v\in W^{1,2}(B,\R^3)$ is a parametric surface with associated
$2$-current $J_v$ then $(J_u-J_v)(\omega)$ is determined by integration of $u^\#\omega-v^\#\omega$ over the set $G:=\{ x\in B: u(x)\not= v(x)\}$, and therefore
we have
\begin{equation}\label{mass-Ju-Jv}
    {\bf M}(J_u-J_v)\le {\bf D}_G(u)+{\bf D}_G(v).
\end{equation}
The main idea for the definition
of the oriented $H$-volume ${\bf V}_H(u,v)$ enclosed
by two  surfaces $u,v\in \SGA$ is to interpret the $2$-current
$J_u-J_v$ as the boundary of an integer
multiplicity rectifiable $3$-current $Q$ of finite mass in $\R^3$,
i.e. to write $J_u-J_v=\partial Q$. Such $3$-currents can
be interpreted as a set with integer multiplicities and finite
(absolute) volume, more precisely, they can be written as 
\begin{equation*}
    Q (\gamma )=\int_{\R^3} i_Q\gamma\qquad\mbox{ for all }\gamma\in\mathcal D^3(\R^3)
\end{equation*}
with an integer valued multiplicity function $i_Q\in
L^1(\R^3,\Z)$. Since in the present situation, the boundary $\partial
Q$ has finite mass, the multiplicity function $i_Q$ turns out to be a 
${\rm BV}$-function on $\R^3$.
The oriented $H$-volume enclosed by $u$ and $v$ can then be
defined by 
\begin{equation*}
    {\bf V}_H(u,v):=\int_{\R^3}i_QH\Omega\, ,
\end{equation*}
where $\Omega$ denotes the standard volume form on $\R^3$.
We interpret this term as the volume of the 
set $\mathrm{spt}\, i_Q$, whose boundary is parametrized by the 
mappings $u$ and $v$, where the multiplicities and the orientation are 
taken into account.
% Of course this has to be understood in duality to the Stokes theorem, that is
% \begin{equation*}
%  \int_{\R^3} i_Qd\omega =\int_B u^\# \omega -\int_B v^\# \omega\quad\forall\, \omega\in\mathcal D^2(\R^3).
% \end{equation*}
In order to make this idea precise, we need to ensure the
existence and the uniqueness of the $3$-current $Q$ with $\partial
Q=J_u-J_v$ from above. We first note that the $2$-currents $J_u-J_v$
considered here 
are spherical in the sense of 

\begin{definition} \label{def-spherical}\upshape
A $2$-current $T$ with support in $A$ is called {\bf spherical} iff it can be represented by a map $f\in W^{1,2}(S^2,A)$ in the form $T=f_\#\lk S^2\rk$, i.e.
\begin{equation}\label{spherical-current}    
    T(\omega )
    =
    \int_{S^2}f^\#\omega\quad \mbox{for all $\omega\in\mathcal
    D^2(\R^3)$.}
\end{equation}
\hfill$\Box$
\end{definition}

From \cite[Lemma 3.3]{DuzaarSteffen:1999} we recall the following
fact.

\begin{lemma}\label{lem-spherical}
For any $u,v\in \SGA$ the current $J_u-J_v$ is a spherical $2$-current in $A$.
\end{lemma}

Since $T:=J_u-J_v$ can be written in the form
\eqref{spherical-current}, it is in particular closed because
\begin{equation*}
  \partial T(\omega)
  =
  \int_{S^2}f^\#d\omega=\int_{S^2}d(f^\#\omega)=0
    \quad \mbox{ for all $\omega\in\mathcal D^1(\R^3)$.}
\end{equation*}
Therefore, for all $u,v\in\SGA$, the current 
$T=J_u-J_v$ is a closed, integer multiplicity
rectifiable $2$-current of finite mass with $\spt T\subseteq A$. By
the deformation theorem, we conclude the existence of 
an integer multiplicity rectifiable $3$-current $Q$ of finite mass with $\partial Q =T$ (see \cite[Thm. 29.1]{Simon} or \cite[4.2.9]{Federer}). Furthermore, the constancy theorem implies that $Q$ is unique up to integer multiples of $\lk\R^3\rk$, which makes $Q$ the unique current of finite mass with $\partial Q=T$.
In order to prove $\spt Q\subseteq A$, we consider the nearest-point-retraction $\pi\colon\R^3\to A$ onto the convex set $A$. From
$\partial \pi_\#Q=\pi_\#\partial Q=T$ and 
${\bf M}(\pi_\# Q)\le ({\rm Lip}\,\pi)^3
{\bf M}(Q)\le {\bf M}(Q)$, we infer in view of the uniqueness
established above that $\pi_\#Q =Q$. This means that $\spt
Q\subseteq A$, as claimed. The above reasoning leads us to 
\begin{lemma}\label{lem:deform}
  Let $A\subseteq\R^3$ be a closed convex set. Then for every spherical
  $2$-current $T$ on $\R^3$ with $\spt T\subseteq A$, there exists a
  \textbf{unique} integer multiplicity rectifiable $3$-current $Q$
  with the properties $\mathbf{M}(Q)<\infty$, $\partial Q=T$ and $\spt Q\subseteq A$.
\end{lemma}

This result allows us to define the oriented $H$-volume enclosed by
two maps $u,v\in\SGA$. 

\begin{definition} \label{def-H-volume}
\upshape
For $u,v\in\SGA$, we write $J_u-J_v$ for the associated 
spherical $2$-current and $I_{u,v}$ for 
the unique integer multiplicity rectifiable
$3$-current with boundary $\partial I_{u,v}=J_u-J_v$, finite mass ${\bf M}(I_{u,v})<\infty$ and $\spt I_{u,v}\subseteq A$.
Then the $H$-{\bf volume}
enclosed by $u$ and $v$ is defined by
\begin{equation*}
     {\bf V}_H(u,v):= I_{u,v}(H\Omega)=\int_{A}i_{u,v}H\Omega\, .
\end{equation*}
Here, $i_{u,v}$ denotes the multiplicity function of $I_{u,v}$, and
$\Omega$ the standard volume form of $\R^3$. 

\hfill $\Box$
\end{definition}

\subsection{Some important properties of the $H$-volume}
Throughout this work, we assume that $H$ satisfies 
a spherical isoperimetric condition of type  $(c,s)$ on $A$ as defined
in \eqref{sphrical-isop}. This condition can be re-written in terms of
the $H$-volume as follows: Consider any $u,v\in\SGA$ 
with ${\bf D}(u)+{\bf D}(v)\le s$, so that in particular
$\mathbf{M}(J_u-J_v)\le s$. Then  
the $H$-volume enclosed by $u$ and $v$ is bounded by
 \begin{equation}\label{isop-Ju-Jv}
     2\big| {\bf V}_H(u,v)\big|
     \le
     c\,{\bf M}(J_u-J_v)\le c\big({\bf D}_G(u)+{\bf D}_G(v)\big),
 \end{equation}
where $G=\{ x\in B: u(x)\not= v(x)\}$. For the second inequality we
refer to \eqref{mass-Ju-Jv}.
Next, we state the following well-known invariance of the volume
functional, cf. \cite[(2.12)]{DuzaarSteffen:1999}.
\begin{lemma}\label{lem-invariant-volume}
  The $H$-volume is invariant under orientation preserving
  $C^1$-diffeo\-mor\-phisms $\varphi,\psi\colon
  \overline B\to\overline B$ in
  the sense that for all $u,v\in\SGA$, there holds
\begin{equation*}
    \mathbf{V}_H(u\circ \varphi,v\circ\psi)=\mathbf{V}_H(u,v).
\end{equation*}
\end{lemma}
%\begin{proof}
%We observe that for all $u,v\in\SGA$ and all
%$\omega\in\mathcal{D}^2(\R^3)$, there holds
%\begin{equation*}
%    J_{u\circ\varphi}(\omega)=\int_B (u\circ\varphi)^\#\omega
%    =\int_B\varphi^\# u^\#\omega=\int_Bu^\#\omega=J_u(\omega) 
%\end{equation*}
%by the transformation formula, since $\varphi$ is
%orientation preserving. Analogously, we infer $J_{v\circ\psi}=J_v$,
%and the claim follows from the definition of the $H$-volume. 
%\end{proof}

The next lemma states that  the $H$-volume functional admits all the properties to derive the variational
(in-)equality (first variation formula) later on. We have
\begin{lemma}\label{prop-h-vol}
Let $u,v\in \SGA$ so that the 
$H$-volume ${\bf V}_H(u,v)$ is defined (cf. Definition \ref{def-H-volume}). Then there hold:

\begin{enumerate}
\item Assume that $\tilde u\in \SGA$ is given.
Then ${\bf V}_H(\tilde u, v)$ and ${\bf V}_H(\tilde u,u)$ (that are
also well-defined by Lemma \ref{lem:deform}) satisfy
$$
    {\bf V}_H(\tilde u, u)+{\bf V}_H( u, v)={\bf V}_H(\tilde u, v)\, ,
$$
and
$$
	\big| {\bf V}_H(\tilde u, u)\big|
	\le 
	\| H\|_{L^\infty} \| u-\tilde u\|_{L^\infty} 
	\big[ {\bf D}_G(u)+{\bf D}_G(\tilde u)\big],
$$
where $G=\{ x\in B: u(x)\not= \tilde u(x)\}$.

\item  Consider a one-sided variation $u_\tau\in \SGA$,
$\tau\in[0,\eps)$, for which  the bound
\begin{equation}\label{uniform-bounds2}
    \sup_{0\le \tau <\eps}
    \Big(
    \|u_\tau\|_{W^{1,2}(B)}
    +
    \|\tfrac\partial{\partial\tau}u_\tau\|_{L^\infty(B)}
    +
    \|\tfrac\partial{\partial\tau}u_\tau\|_{C^0(\partial B)}
    \Big)
    <
    \infty
\end{equation}
holds true.
Then ${\bf V}_H(u_\tau,u)$ and ${\bf V}_H(u_\tau,v)$ are defined for
$\tau\in[0,\eps)$ and with the 
abbreviation $U(\tau,x):=u_\tau(x)$, the following homotopy formula holds:
\begin{align}\label{homotopy}\nonumber
	{\bf V}_H(u_\tau,v)-{\bf V}_H(u,v)
	&
	= {\bf V}_H(u_\tau,u) \\
  	&
	=
  	\int_B\int_0^\tau (H\circ U)\left\langle \Omega\circ U,
   U_s \wedge U_{x_1}\wedge U_{x_2}\right\rangle\, ds\, dx.
\end{align}
\end{enumerate}
\end{lemma}

\begin{proof}
For the proof of (i), we refer to \cite[Lemma
3.6\,(i)]{DuzaarSteffen:1999}. We turn our attention to the proof
of (ii). We define a $3$-current by 
\begin{equation*}
    Q_U(\phi)
    :=
    \int_B\int_0^\tau\langle\phi\circ U,U_s\wedge U_x\wedge
    U_y\rangle \,ds\,dx\,dy
    =\int_{[0,\tau]\times B}U^\#\phi
\end{equation*}
for every $\phi\in\mathcal{D}^3(\R^3)$. The idea of the proof is to
apply a construction 
similar to the one from \cite[Lemma 3.3\,(i)]{DuzaarSteffen:1999} to each of the functions $u_s:=U(s,\cdot)$ for any $s\in[0,\tau]$. To this end, we note that since $u_s\in\SGA\subseteq \Ss$, we can find $\varphi_s\in\mathcal{T}^*(\Gamma)$ with
\begin{equation*}
  u_s(e^{i\vartheta})=\widehat\gamma(\varphi_s(\vartheta)) \qquad\mbox{for
    each }\vartheta\in [0,2\pi].
\end{equation*}
Because $\widehat\gamma\colon \R\to\Gamma$ is a  local $C^1$-diffeomorphism, the
assumption~(\ref{uniform-bounds2}) implies
\begin{equation}
\label{uniform-bounds-phi}
  \sup_{0\le s\le\tau}
  \Big(
  \|\varphi_s\|_{W^{1/2,2}(0,2\pi)}
  +
  \|\tfrac\partial{\partial s}\varphi_s\|_{C^0([0,2\pi])}
  \Big)<\infty.
\end{equation}
Now we choose an arbitrary $\delta>0$ and define
$h_s\colon [1-\delta,1]\times S^1\to\R$ as the unique harmonic function with
boundary values given by 
\begin{equation*}
  h_s(1-\delta,e^{i\vartheta})=\varphi_s(\vartheta)-\vartheta
  \qquad
  \mbox{and}
  \qquad
  h_s(1,e^{i\vartheta})=0.
\end{equation*}
We note that these boundary traces are well-defined since
$\varphi_s(\cdot+2\pi)=\varphi_s+2\pi$ for every $\varphi_s\in\mathcal{T}^\ast(\Gamma)$. 
As a consequence of~(\ref{uniform-bounds-phi}), this function satisfies
\begin{equation}\label{uniform-bound-h-1}
  \sup_{0\le s\le\tau}\|h_s\|_{W^{1,2}}
  \le
  c \sup_{0\le s\le\tau}
  \big(
  \|\varphi_s\|_{W^{1/2,2}(0,2\pi)}
  +
  \|{\rm id}\|_{W^{1/2,2}(0,2\pi)}
  \big)<\infty.
\end{equation}
Moreover, the derivative $\frac\partial{\partial s}h_s$ is again a
harmonic function, with the boundary values given by
$\frac\partial{\partial s}\varphi_s$ on $\{1-\delta\}\times \partial B$ and
by  zero on $\{1\}\times \partial B$. The maximum principle
and~(\ref{uniform-bounds-phi}) therefore imply
\begin{equation}\label{uniform-bound-h-2}
    \sup_{0\le s\le\tau}
    \|\tfrac\partial{\partial s}h_s\|_{L^\infty}
  	 \le 
    \sup_{0\le s\le\tau}
    \|\tfrac\partial{\partial s}\varphi_s\|_{C^0([0,2\pi])}
    <\infty.
\end{equation}
Now we are in a position to define the functions $\tilde u_s:B\to\R^3$
by
\begin{equation*}
  \tilde u_s(\rho e^{i\vartheta})
  :=
  \begin{cases}
    u_s\big(\tfrac\rho{1-\delta}e^{i\vartheta}\big)
    &\mbox{for }0\le\rho<1-\delta,\\[5pt]
    \widehat\gamma\big(h_s(\rho,e^{i\vartheta})+\vartheta\big)
    &\mbox{for }1-\delta\le\rho\le1.
  \end{cases}
\end{equation*}
We note that the definition of $h_s$ ensures that $\tilde u_s\in W^{1,2}(B,\R^3)$ for each $s\in[0,\tau]$. 
Since $\widehat\gamma$ is a local $C^1$-diffeomorphism, the
bounds~(\ref{uniform-bounds2}), (\ref{uniform-bound-h-1}) and (\ref{uniform-bound-h-2}) imply
\begin{equation}\label{uniform-bounds-tilde}
   \sup_{0\le s <\tau}
   \Big(
   \|\tilde u_s\|_{W^{1,2}(B)}
   +
   \|\tfrac\partial{\partial s}\tilde u_s\|_{L^\infty(B)}
   \Big)
   <\infty.
\end{equation}
Moreover, $\tilde u_s(e^{i\vartheta})=\widehat\gamma(\vartheta)$ for each
$s\in[0,\tau]$ and $\vartheta\in\R$, so that $\tilde u_s|_{\partial B}$ is of class $C^1$. Finally, since $\tilde u_s$ is constructed
as a re-parametrization of the original variation $u_s$ and the
$C^1$ diffeomorphism $\widehat \gamma$ we also have $\tilde u_s(B)
\subseteq A$ for $s\in [0,\tau]$.
We abbreviate $\Tilde U(s,x):=\tilde u_s(x)$ and observe that 
\begin{equation*}
    \Tilde U\big([0,\tau]\times(B{\setminus}B_{1-\delta})\big)
    \subset\Gamma.
\end{equation*}
Since $\Tilde U\big|_{[0,\tau]\times B_{1-\delta}}$ is defined as a rescaled version of $U$ and $\Gamma$ is a one-dimensional curve, the above construction does not change the corresponding currents, more precisely we have
\begin{equation*}
     Q_{\Tilde U}=Q_U,
     \qquad 
     J_{\tilde u_o}=J_u
     \qquad
     \mbox{and}
     \qquad
     J_{\tilde u_\tau}=J_{u_\tau}.
\end{equation*}
We claim that $\partial Q_U=J_{u_\tau}-J_u$. To this end, we choose
$\omega\in\mathcal{D}^2(\R^3)$ and calculate, using Stokes' theorem:
\begin{align*}
    \partial Q_U(\omega)
     &
     =Q_U(d\omega)
     =Q_{\Tilde U}(d\omega)
     =\int_{[0,\tau]\times B}d(\Tilde U^\#\omega)\\
     &
     =\int_B\Tilde u_\tau^\#\omega - \int_B\Tilde u_o^\#\omega
      +\int_{[0,\tau]\times \partial B}\Tilde U^\#\omega.
\end{align*}
Here, the application of Stokes' theorem can be justified by an
approximation argument since we
have~(\ref{uniform-bounds-tilde}) and $\Tilde U$ is of class $C^1$
on $[0,\tau]\times \partial B$ . Next, we observe that the last
integral vanishes because $\frac\partial{\partial s}\Tilde U$
vanishes on $[0,\tau]\times \partial B$. We thereby deduce
\begin{align*}
    \partial Q_U(\omega)
     =J_{u_\tau}(\omega)-J_u(\omega)
     \qquad
     \mbox{for all }\omega\in\mathcal{D}^2(\R^3).
\end{align*}
The definition of the $H$-volume now yields the claim~(\ref{homotopy}).  
\end{proof}

\section{The time discrete variational formulation}
To set up the approximation scheme by time discretization  we shall use $H$-energy functionals with a suitable lower order perturbation
term of the form
\begin{equation}\label{discrete-H-energy}
    {\bf F}(u)
    := 
    {\bf D}(u) +2 {\bf V}_H(u,u_o)+ \tfrac{1}{2h}\int_B |u-z|^2\, dx
    \equiv 
    \mathbf F^{(h)}_{u_o,z}(u)
\end{equation}
defined for $u\in \SGA$, where $u_o\in\SGA$ is a given fixed reference surface;  see Definition~\ref{def-H-volume} for the notion
of the volume functional. Here, $h>0$ and $z\in \SGA$
are given. The $H$-volume term measures 
 the oriented volume enclosed by $u$ and the given fixed reference surface $u_o$ weighted with respect to $H$. In order not to overburden
the presentation of the results and proofs we prefer not to
indicate the dependence of the functional on the data $u_o$, $z$ and
$h>0$. We start with the following assertion concerning the first variation formulae.

%We begin by calculating the first variation of $\mathbf{F}$.
%We call $u_\tau$ a {\bf sufficiently regular variation} of $u$ if
%$u_0=u$,
% $u_\tau\in W^{1,2}(B,\R^3)$ and $\frac\partial{\partial\tau}u_\tau\in
%W^{1,2}(B,\R^3)\cap L^\infty(B,\R^3)$ for
%$|\tau|$ sufficiently small, together with the uniform bounds~(\ref{uniform-bounds2}). A similar notion
%will be used for {\bf one-sided variations}, that is variations 
%$u_\tau$ of $u$ with  $u_\tau\in W^{1,2}(B,\R^3)$ for
%$\tau \ge 0 $ sufficiently small.

\begin{lemma}\label{first-variation}
\strut

(i) Let $u_\tau\in\SGA$, $\tau\in [0,\varepsilon)$ by a one-sided variation of $u\in \SGA$ 
with initial vector field $\varphi\in L^\infty(B,\R^3)\cap
W^{1,2}(B,\R^3)$ and assume that it satisfies the bounds
\eqref{uniform-bounds2}. Then we have
\begin{align}\label{first-var-formula}
    \lim_{\tau\downarrow 0}\frac{{\bf F}(u_\tau )-{\bf F}(u )}{\tau}
        &
    =
    \int_B \Big[\frac{u-z}{h}\cdot\varphi+Du\cdot D\varphi 
    +
     2(H\circ u)D_1u\times D_2u \cdot \varphi \Big]\, dx.
\end{align}
%The same formula holds for sufficiently regular one-sided variations
%$u_\tau$ with $\tau\in[0,\eps)$ if we replace the derivative by the
%right-sided derivative with respect to $\tau$. 

(ii) If $u\in \SGA$ and $\varphi_\tau$ is the flow generated by a vector field $\eta\in \mathcal C^\ast (B)$, then
\begin{align}\label{first-inner-var-formula}
    \frac{d}{d\tau}&\Big|_{\tau =0}{\bf F}(u\circ\varphi_\tau )
    =
    \int_B
	 {\rm Re}\big(\frak h[u]\overline\partial\eta\big)\, dx
	 +
	 \tfrac1{h}\int_B (u-z)\cdot Du\,\eta\, dx.
\end{align}
\end{lemma}
\begin{proof}
The assertion (i) follows from a straightforward calculation, using
the homotopy formula~(\ref{homotopy}). For the claim (ii), 
in view of Lemma \ref{lem-invariant-volume} and \eqref{inner-D}
we only have to compute
\begin{equation*}
	\frac{d}{d\tau}\Big|_{\tau=0}
	\tfrac1{2h}\int_B |u\circ\varphi_\tau-z|^2\, dx
	=
	\tfrac1{h}\int_B (u-z)\cdot Du\eta\, dx.\qedhere
\end{equation*}
\end{proof}
The integral  $\delta {\bf F}(u;\varphi )$ in \eqref{first-var-formula} is called the {\bf first variation} of the functional
${\bf F}$ in direction $\varphi$ and the integral $\partial \mathbf{F}(u;\eta)$ from (ii) the {\bf first variation
of independent variables} (inner first variation) of $\mathbf F$
at $u$ in the direction $\eta$. The preceding lemma leads to the following

\begin{lemma}\label{lem:EulerMeasure}
Let  $A\subseteq\R^3$ be the closure of a convex $C^2$-domain in $\R^3$ and assume
that the principal curvatures of $\partial A$ are bounded with 
\begin{equation}\label{krumm-bed}
    |H(a)|\le \mathcal H_{\partial A}(a)
    \quad
    \mbox{for $a\in\partial A$,}
\end{equation}
where $\mathcal H_{\partial A}(a)$ denotes 
the minimum of the principle curvatures of $\partial A$ at the
point $a$ with respect to the inner unit normal $\nu (a)$.
Assume that $u\in\SGA$ minimizes $\mathbf{F}$ in the class
$\SGA$. Then it satisfies the variational inequality 
\begin{equation}\label{time-discrete-evolution-inequality}
    \int_B\Big[ \frac{u-z}{h}\cdot\varphi+Du\cdot D\varphi +2(H\circ
    u) D_1u\times D_2u\cdot\varphi\Big]\, dx
    \ge 0
\end{equation}
for all $\varphi\in T_{u}\mathcal S^\ast$, and
moreover, the stationarity condition 
\begin{equation}\label{stationary-discrete}
    \int_B
	 {\rm Re}\big(\frak h[u]
	 \overline\partial\eta\big)\, dx
	 +
	 \tfrac1{h}\int_B (u-z)\cdot Du\eta\, dx=0
\end{equation}
holds true for every  $\eta\in \mathcal C^\ast (B)$. 
\end{lemma}

\begin{proof}
The case of variation vector fields with zero boundary values
is covered in our
earlier work~\cite{BoegDuzSchev:2011}, cf. Lemma 4.3, Lemma 4.4\,(iii). More
precisely, under the condition~(\ref{krumm-bed}), we derived the Euler-Lagrange equation
\begin{equation}\label{time-discrete-evolution}
    \int_B\Big[ \frac{u-z}{h}\cdot\varphi+Du\cdot D\varphi +2(H\circ u) D_1u\times D_2u\cdot\varphi\Big]\, dx
    = 0
\end{equation}
for every $\mathbf F$-minimizer $u\in\SGA$ of $\mathbf{F}$ and any 
$\varphi\in L^\infty (B,\R^3)\cap W^{1,2}_0(B,\R^3)$. Therefore, it remains to prove
the corresponding inequality for vector fields
$w\in T_{u} \mathcal{S}^\ast$. To this end, we
employ Lemma~\ref{lem:rand-variationen}\,(ii) to construct 
a one-sided variation $u_\tau\in\SGA$, $\tau\in [0,\eps)$ with the
properties~(\ref{uniform-bounds2}), $u_0=u$ and
$\widetilde w:=\frac\partial{\partial \tau}\big|_{\tau=0}u_\tau\in
w+W^{1,2}_0(B,\R^3)$. Since the maps $u_\tau$ are admissible as
comparison maps for $u$, we infer from~(\ref{first-var-formula}) that
\begin{equation*}
  0\le 
    \lim_{\tau\downarrow 0}\frac{{\bf F}(u_\tau )-{\bf F}(u )}{\tau}
    =
    \int_B \Big[\frac{u-z}{h}\cdot \widetilde w+Du\cdot D\widetilde w
    +
     2(H\circ u)D_1u\times D_2u \cdot \widetilde w\Big]\, dx.
\end{equation*}
Moreover, $\varphi:=w-\widetilde w\in L^\infty\cap W^{1,2}_0(B,\R^3)$ is
admissible in~(\ref{time-discrete-evolution}). Joining this with the
above inequality, we infer~(\ref{time-discrete-evolution-inequality})
for $\varphi=w$, which completes the proof of (\ref{time-discrete-evolution-inequality}). 

For the second assertion~(\ref{stationary-discrete}), we consider the flow
$\phi_\tau$ of the vector field $\eta\in \mathcal C^\ast (B)$ with
$\phi_0={\rm id}$. Then,  the maps $u\circ\phi_\tau\in\SGA$
are admissible competitors for $u$ (see the derivation of \eqref{inner-D}), and  
(\ref{stationary-discrete}) follows from the inner
variation formula~(\ref{first-inner-var-formula}).
\end{proof}

\section{Existence of minimizers to the time-discrete problem}

The next lemma will be crucial for the construction of minimizers to
the time-discrete volume functional by the direct method of the
calculus of variations. It was proven in \cite[Lemma
4.1]{DuzaarSteffen:1999}, 
see also \cite[Lemma 4.1]{DuzaarGrotowski:2000} for a version in
higher dimensions. 
The main idea is to control the volume of possible
bubbles occurring in a minimizing sequence $u_i$ 
by replacing it by a new sequence $\tilde u_i$.
This new sequence agrees with the limit map $u$ 
outside of a small set $G$ on which bubbles may evolve, while on
this set, the energy of the $\tilde u_i$ is controlled by the bubble energy
(cf. (vi) below). The term $|V_H(\tilde u_i,u)|$ -- which can be
interpreted as the volume of the bubbles --  can  be bounded in terms of the Dirichlet energy by use of
the isoperimetric condition. This
 enables us to establish a lower semicontinuity property of the
time-discrete volume functional and thereby 
to prove the existence of $\mathbf F$-minimizers.

\begin{lemma}\label{lem-bubbels}
Assume that $u_i\wto u$ weakly in $W^{1,2}(B,\R^m)$ and
$u_i\big|_{\partial B}\to u\big|_{\partial B}$ uniformly on $\partial B$.
Then for every $\varepsilon >0$ there exist $R>0$, a measurable set
$G\subseteq B$, and maps $\tilde u_i\in W^{1,2}(B,\R^m)$, such that
after extraction of a subsequence there holds:
\begin{enumerate}
\item[{\rm (i)}] $\tilde u_i=u$ on $B\setminus G$ with $\mathcal L^2(G)<\varepsilon$;
\item[{\rm (ii)}] $\tilde u_i\big|_{\partial B}=u\big|_{\partial B}$ on $\partial B$;
\item[{\rm (iii)}] $\tilde u_i(x)=u_i(x)$ if $|u_i(x)|\ge R$ or $| u_i(x)-u(x)|\ge 1$;
\item[{\rm (iv)}] $\lim_{i\to\infty} \| \tilde u_i-u_i\|_{L^\infty(B,\R^m)}=0$;
\item[{\rm (v)}] $\tilde u_i\wto u$ weakly in $W^{1,2}(B,\R^m)$ in the limit $i\to\infty$;
\item[{\rm (vi)}] $\limsup_{i\to\infty} [{\bf D}_G(\tilde u_i)+{\bf D}_G(u)]\le\varepsilon +\liminf_{i\to\infty} [{\bf D}(u_i)-{\bf D}(u)]$;
\item[{\rm (vii)}] If the $u_i$ take values in a closed convex subset $A\subseteq\R^3$, then the $\tilde u_i$ can also be chosen to have values in $A$.
\end{enumerate}
\end{lemma}
The proof of (i) to (vi) was carried out in \cite[Lemma
4.1]{DuzaarSteffen:1999}. The assertion (vii) follows immediately
from the construction in
\cite{DuzaarSteffen:1999}, since the maps $\tilde u_i$ are defined as
convex combinations of $u_i$ and $u$, whose images are contained in
the convex set $A$. 

Lemma \eqref{lem-bubbels} enables us to prove the existence of $\mathbf F$-minimizers in the class 
\begin{equation*}
  \SGAs:=\big\{w\in\SGA\,:\,\mathbf D(w)\le\sigma\mathbf D(u_o)\big\}, 
\end{equation*}
where we choose $\sigma:=\frac{1+c}{1-c}$ if $s<\infty$ and
$\sigma=\infty$ otherwise. This choice of $\sigma$ is made in such a
way that~(\ref{Assum-uo}) implies $(1+\sigma)\mathbf{D}(u_o)\le s$. 

\begin{lemma}\label{lem:F-min}
Suppose $A\subseteq\R^3$ is a closed convex set and the function
$H\colon A\to\R$ is bounded and continuous and satisfies a spherical
isoperimetric condition of type $(c,s)$ on $A$ with $0<s\le\infty$ and
$0<c<1$. Moreover let
$u_o\in \SGA$ be a fixed reference surface with~(\ref{Assum-uo})
and $z$ a fixed surface in $\SGAs$ for $\sigma$ defined as above.
Then for every $h>0$, the variational
problem
\begin{equation*}
	\mathbf{F}(w)
	:=
	\mathbf{D}(w)+2\mathbf{V}_H(u,u_o)+\tfrac1{2h}\int_B |u-z|^2\, dx
	\to\min\quad \mbox{in $\SGAs$}
\end{equation*}
has a solution.

\end{lemma}
\begin{proof}
We first observe that for any $w\in\SGAs$ by the choice of $\sigma$ there holds
\begin{equation*}
	\mathbf{M}(J_w-J_{u_o})
	\le 
	\mathbf{D}(w)+\mathbf{D}(u_o)
	\le
	(\sigma +1)\mathbf{D}(u_o)
	\le
	s.
\end{equation*}  
Hence, the spherical isoperimetric condition of type $(c,s)$ gives 
\begin{align*}
	2|\mathbf{V}_H(w,u_o)|
	\le
	c\,{\bf M}(J_w-J_{u_o})\le c\big({\bf D}(w)+{\bf D}(u_o)\big).
\end{align*}
This implies in particular that the functional is bounded from below
on $\SGAs$ by 
\begin{equation}\label{Bound-below}
	\mathbf F(w)
	\ge (1-c)\mathbf D (w)-c\mathbf D(u_o)
	\ge -c\mathbf D(u_o).
\end{equation}
We now consider an $\mathbf F$-minimizing sequence $(u_i)$ of maps in $\SGAs$, that is
\begin{equation*}
	\lim_{i\to\infty} \mathbf{F}(u_i)=\inf\big\{\mathbf{F}(w):w\in\SGAs\big\}.
\end{equation*}
Applying~(\ref{Bound-below}) to $w=u_i$, we infer
\begin{equation*}
  \sup_{i\in\N} \mathbf D(u_i)
  \le
  \tfrac1{1-c}\big[ \sup_{i\in\N} \mathbf F(u_i) +c\mathbf D(u_o)\big]<\infty.
\end{equation*}
Lemma~\ref{lem:compact traces} thus implies that 
% Since the maps
% $u_i$ are in the class $\SGA$ they also satisfy the three point condition,
% that is $u(P_k)=Q_k$ for $k=1,2,3$.
% As is well known from the theory of parametric minimal surfaces,
% it then follows from a Lemma
% of Courant and Lebesgue and the Jordan curve property of $\Gamma$ that the sequence
the boundary traces $u_i\big|_{\partial B}$ are equicontinuous.
Passing to a subsequence
and taking Rellich's theorem into account we may therefore assume
that the maps $u_i$ converge weakly in $W^{1,2}(B,\R^3)$,
strongly in $L^2(B,\R^3)$, and almost
everywhere on $B$ to a surface $u\in W^{1,2}(B,A)$ with
$\mathbf D(u)\le \sigma \mathbf D(u_o)$.
Moreover, we have that $u_i\big|_{\partial B}\to u\big|_{\partial B}$
holds uniformly on $\partial B$.
Due to the uniform convergence on $\partial B$
and the fact that the sequences $u_i$ satisfy the three point condition $u_i(P_k)=Q_k$, also the limit surface
fulfills the three point condition $u(P_k)=Q_k$ for
$k=1,2,3$, and therefore we have $u\in \SGAs$.
We now apply Lemma \ref{lem-bubbels} with a given $0<\varepsilon<\frac12\mathbf{D}(u)$ to obtain, after passage to
another subsequence,  surfaces $\tilde u_i\in \SGA$.
Since $u,u_i$ and $u_o$ are in the class $\SGA$, 
Lemma~\ref{prop-h-vol}\,(i) 
and Lemma~\ref{lem-bubbels}\,(iv) and (vi) yield that 
\begin{equation}\label{volume-vanish}
   \big|\mathbf V_H(\tilde u_i,u_i)\big|
   \le \|H\|_{L^\infty}\|\tilde
   u_i-u_i\|_{L^\infty}\big[\mathbf{D}_G(\tilde u_i)+\mathbf{D}_G(u_i)\big]\to 0\quad\mbox{as $i\to\infty$.}
\end{equation}
Now, we infer from
Lemma \ref{lem-bubbels} (vi) and \eqref{mass-Ju-Jv} that 
for $i$ large enough there holds
\begin{align*}
	\mathbf M (J_{\tilde u_i}-J_u)
	&\le
	\mathbf D_G(\tilde u_i) +\mathbf D_G(u)\\
	&\le
	2\varepsilon +\mathbf D(u_i)-\mathbf D(u)
	<
	\sigma \mathbf D(u_o)\le s.
\end{align*}
Therefore, by the spherical isoperimetric condition
with $c<1$  we have
\begin{equation*}
	2\big| \mathbf V_H(\tilde u_i,u)\big|
	\le
	2\varepsilon +\mathbf D(u_i)-\mathbf D(u).
\end{equation*}
Moreover we have
\begin{equation*}
	\mathbf V_H(u_i,u_o)
        -
        \mathbf{V}_H(u,u_o)
	=
	\mathbf V_H(\tilde u_i,u)
	-
	\mathbf V_H(\tilde u_i,u_i).
\end{equation*}
This allows us to conclude -- with the help of the strong
convergence $u_i\to u$ in $L^2(B,\R^3)$ and (\ref{volume-vanish}) -- that there holds
\begin{align*}
	\mathbf F(u_i)
	&=
	\mathbf D(u_i)+2\mathbf V_H( u_i,u_o)
	+
	\tfrac1{2h}\int_B |u_i-z|^2\, dx\\
	&=
	\mathbf F(u) -\mathbf D(u) +\mathbf D(u_i)
	+
	2\mathbf V_H(\tilde u_i,u)-2\mathbf V_H(\tilde u_i,u_i)\\
	&\qquad\qquad
	+\tfrac1{2h}\int_B |u_i-z|^2\, dx
	-
	\tfrac1{2h}\int_B |u-z|^2\, dx\\
	&\ge
	\mathbf F(u)-3\varepsilon
\end{align*}
for sufficiently large $i\in\N$. 
Since $\varepsilon >0$ was arbitrary we conclude that
$u\in \SGAs$ minimizes the variational functional $\mathbf F$.
\end{proof}

The following regularity result for minimizers was established in \cite[Theorem
6.1]{BoegDuzSchev:2011}. We note that it is also a special case of the much
more involved result by Rivi\`ere (see Lemma~\ref{mod-Riviere}) for solutions to the Euler-Lagrange equation~(\ref{time-discrete-evolution}).
\begin{lemma}
Suppose $A$, $H$, $u_o$ and $z$ satisfy the hypotheses of
Lemma \ref{lem:F-min} with parameters $\sigma, s$ and $c$.
Then, every minimizer of the functional $\mathbf{F}$ in the class
$\SGAs$ is H\"older continuous in $B$ and continuous up to the boundary, that is $u\in C^0(\overline B,\R^3)$.
\end{lemma}
\section{A-priori estimates}
 In this section, we derive local a-priori estimates for solutions of
 problems satisfying a Plateau boundary condition -- i.e. for solutions
 to inequalities of the
 type~(\ref{time-discrete-evolution-inequality}) -- under the additional assumption $u\in
 W^{1,4}(B_R^+(x_o),\R^3)$. The validity of the $W^{1,4}$ assumption will be justified later. 
 Moreover, we restrict ourselves to the case
 $B_R^+(x_o)\cap\{P_1,P_2,P_3\}=\varnothing$. This will be sufficient
 for our purposes since the set $\{P_1,P_2,P_3\}$ has vanishing
 capacity.  More precisely, we
 consider maps $u\in\Ss\cap W^{1,4}(B_R^+(x_o),\R^3)$ with
 $\Delta u=F$ on $B_R^+(x_o)\subset B$ for some $x_o\in\partial B$ and
 $F\in L^1(B_R^+(x_o),\R^3)$. Additionally, we assume that $u$
 satisfies the natural boundary condition associated with the Plateau
 problem on $I_R(x_o)\subset\partial B_R^+(x_o)$, i.e. we assume that
 \begin{equation}\label{local-plateau}
   \int_{B_R^+(x_o)}Du\cdot Dw\dx+ \int_{B_R^+(x_o)}F\cdot w\dx\ge0
 \end{equation}
holds true for all $w\in T_u\mathcal{S}^\ast$ 
with $w=0$ on $S_R^+(x_o)$ in the sense of traces.
We recall that the condition $w\in
T_u\mathcal{S}^\ast$ is equivalent to the boundary representation
$w(e^{i\vartheta})=\widehat\gamma'(\varphi)(\psi-\varphi)$ for some
$\psi\in\mathcal T^\ast(\Gamma)$. Here, $\varphi\in\mathcal T^\ast(\Gamma)$ is defined
by $u(e^{i\vartheta})=\widehat\gamma(\varphi(\vartheta))$ for all
$\vartheta\in\R$. On the inhomogeneity $F$, we impose the growth condition
\begin{equation}
  \label{F-growth}
  |F|\le C_1\big(|Du|^2+|f|\big)\qquad\mbox{in }B_R^+(x_o),
\end{equation}
where  $C_1>0$ and $f\in L^2(B_R^+(x_o),\R^3)$ are given. Moreover,
we assume
\begin{equation}
  \label{time-deriv-bound}
  \int_{B_R^+(x_o)}|f|^2\dx\le K^2
\end{equation}
for some constant $K>0$.
We start  with the interior a-priori estimate. 

\begin{theorem}\label{thm:apriori-W22-int}
On $B_R(x_o)\Subset B$ consider a solution $u\in
W^{1,4}(B_R(x_o),\R^3)$ of $\Delta u=F$, where $F$
satisfies (\ref{F-growth}) for $C_1>0$ and $f\in L^2(B_R(x_o),\R^3)$. There exist $\eps_o=\eps_o(C_1)>0$ and $C=C(C_1)$ such that the smallness condition
\begin{equation}\label{small-energy-1}
    \int_{B_R(x_o)}|Du|^2\dx\le\eps_o^2
\end{equation}
implies $u\in W^{2,2}(B_{R/2}(x_o),\R^3)$ with
\begin{equation*}
    \int_{B_{R/2}(x_o)}|D^2u|^2\dx
    \le
    \frac C{R^2}\int_{B_R(x_o)}|Du|^2\dx
    +C\int_{B_R(x_o)}|f|^2\dx.
  \end{equation*}
\end{theorem}
We note that in \cite[Lemma 7.3]{BoegDuzSchev:2011}, this result was
established for more regular right-hand sides with $f\in W^{1,2}(B_R(x_o),\R^3)$. Here, we shall weaken this property to the natural assumption $f\in L^2(B_R(x_o),\R^3)$.

\begin{proof}
A standard application of the difference quotient
technique yields the following estimate for any
radii $s,t$ with $\frac R2\le s<t\le R$:
\begin{equation*}
    \int_{B_s(x_o)}|D^2u|^2\dx
     \le 
     \frac C{(t-s)^2}\int_{B_t(x_o)}|Du|^2\dx
     +
     C\int_{B_t(x_o)}|F|^2\dx.
\end{equation*}
Using the growth condition (\ref{F-growth}), the regularity
assumption $u\in W^{1,4}(B_R(x_o),\R^3)$ and the Gagliardo-Nirenberg
interpolation inequality (\ref{appl-gag-nir}), we can further
estimate
\begin{align*}
    \int_{B_t(x_o)}
    &|F|^2\dx
    \le 
    C \int_{B_t(x_o)}|Du|^4+|f|^2\dx\\
    &
    \le 
    C \int_{B_t(x_o)}|D^2u|^2 + \Big|\frac{Du}{R/2}\Big|^2\, dx\,
    \int_{B_R(x_o)}|Du|^2\dx+ C\int_{B_R(x_o)}|f|^2\dx,
\end{align*}
for a constant $C=C(C_1)$.
Combining the preceding two estimates, using the smallness
assumption (\ref{small-energy-1}) and also $t-s\le R/2$,
we arrive at
\begin{align*}
    \int_{B_s(x_o)}|D^2u|^2\dx
    &\le 
    C\eps_o^2\int_{B_t(x_o)}|D^2u|^2\dx\\
    &\phantom{\le}
    +\frac C{(t-s)^2}\int_{B_R(x_o)}|Du|^2\dx
    +
    C\int_{B_R(x_o)}|f|^2\dx.
\end{align*}
Choosing $\eps_o\in(0,1)$ small enough in dependence on
$C=C(C_1)$, we can therefore derive the
claim by an application of the Iteration Lemma \ref{lem:Giaq}. 
\end{proof}

For the boundary analogue of Theorem \ref{thm:apriori-W22-int},
we additionally have to assume small oscillation of $u$.
This is needed for the following extension result which we
shall employ for the construction of admissible testing functions. 

\begin{lemma}\label{lem:extension}
There is a radius $\rho_o=\rho_o( \Gamma)\in(0,1)$ 
such that the following holds. Consider maps $u_k\in
W^{1,2}(B_{r}^+(x_o),\R^3)$ for $k\in\mathcal I $ with
an index set $\mathcal{I}$ and $x_o\in\partial B$, such that
for some $p_o\in\Gamma$ we have 
\begin{equation}\label{kleines-Bild}
    u_k(B_{r}^+(x_o))\subset B_{\rho_o}(p_o)\subset\R^3
    \quad
    \mbox{for all $ k\in\mathcal I$.}
\end{equation}
Then there are maps $\Phi_k\in W^{1,2}(B_{r}^+(x_o),\R)$ with 
\begin{equation*}
    \widehat\gamma\circ\Phi_k=u_k
    \quad\mbox{on $I_{r}(x_o)$ for $k\in\mathcal I $,}
\end{equation*}
and which satisfy a.e. on $B_{r}^+(x_o)$ the estimates
\begin{align*}
  |\Phi_k-\Phi_\ell|&\le C|u_k-u_\ell|,\\
  |D\Phi_k|&\le C|Du_k|,\\
  |D\Phi_k-D\Phi_\ell|&\le C|Du_k-Du_\ell|+|Du_k|\,|u_k-u_\ell|
\end{align*}
for all $k,\ell\in\mathcal I$, where $C$ denotes a universal constant that
depends only on $\Gamma$. 
\end{lemma}

\begin{proof}
Since $\Gamma\subset\R^3$ is a Jordan-curve of class $C^3$, there is
a tubular neighborhood $U\subset\R^3$ of $\Gamma$ such that the
nearest-point retraction $\pi\colon U\to\Gamma$ is well-defined
and of class $C^2$ with $\|\pi\|_{C^2}\le C(\Gamma)$. For a $\rho_o>0$ sufficiently small, the assumption (\ref{kleines-Bild}) implies in particular $u_k(B_{r}^+(x_o))\subset U$ for $k\in\mathcal I $. Consequently, we may define
\begin{equation*}
    \widetilde\Phi_k:=\gamma^{-1}\circ \pi\circ u_k\colon B_{r}^+(x_o)\to S^1
    \quad\mbox{for $k\in\mathcal I $.}
\end{equation*}
Since $\gamma^{-1}\circ\pi$ is Lipschitz continuous
with Lipschitz constant $L$ depending only on $\Gamma$,
we may once more diminish $\rho_o>0$ in
dependence on $\Gamma$ such that $\gamma^{-1}\circ\pi(B_{\rho_o}(p_o))$
is contained in a half sphere $S^+:=\{e^{i\vartheta}\in
S^1\,:\,\vartheta_o<\vartheta<\vartheta_o+\pi\}$ for some
$\vartheta_o\in\R$.  On this half sphere, the function
$\arg\colon S^+\ni e^{i\vartheta}\mapsto 
\vartheta\in (\vartheta_o,\vartheta_o+\pi)$
is a $C^2$-diffeomorphism, which implies that
$\widehat\gamma^{-1}=\arg\circ\gamma^{-1}$ is of class $C^2$ with
$\|\widehat\gamma^{-1}\|_{C^2}\le C(\Gamma)$ on 
$\pi(B_{\rho_o}(p_o))$. Now, for  $k\in\mathcal I $ we define
\begin{equation*}
    \Phi_k
    :=
    \arg\,\circ\,\widetilde \Phi_k
    =
    \widehat\gamma^{-1}\circ\pi\circ
    u_k\colon B_{r}^+(x_o)\to\R.
\end{equation*}
Since $\widehat\gamma^{-1}\circ\pi$ is Lipschitz
continuous on $B_{\rho_o}(p_o)$, we deduce the first two of the
asserted estimates.
For the last one, we calculate $D\Phi_k=(T{\circ}u_k)Du_k$, where $T=D(\widehat\gamma^{-1}\circ\pi)$ is of class $C^1$ 
with $\|T\|_{C^1}\le C(\Gamma)$. This implies the remaining claim
by straightforward calculations. 
\end{proof}

Now we are in a position to prove the a-priori estimates up to the
boundary. 

\begin{theorem}\label{thm:apriori-W22-bdry}
There is a constant $\eps_1=\eps_1(C_1,\Gamma)\in(0,1)$ 
for which the following two criteria for $W^{2,2}$-regularity hold true: Whenever $u\in\Ss$ satisfies (\ref{local-plateau}) on a half-disk around $x_o\in\partial B$ with $B_R^+(x_o)\cap
\{P_1,P_2,P_3\}=\varnothing$ and $R<\frac12$, then we have:
\begin{enumerate}
\item[(i)] 
If $F=0$ and the solution satisfies
$\osc_{B_R^+(x_o)}u\le \eps_1$
as well as the Morrey space estimate
\begin{equation}\label{Morrey-assumption}
      \int_{B_r^+(y)} |Du|^2\dx
      \le 
      C_M\Big(\frac rR\Big)^{2\alpha}
      \quad
      \mbox{for all $y\in B_{R/2}^+(x_o)$ and $0<r<\tfrac R2$}
\end{equation}
for some $\alpha\in(0,1)$ and $C_M\ge1$, then we have 
$u\in W^{2,2}(B_{R/4}^+(x_o),\R^3)$ and with a constant
$C=C(\Gamma,\alpha)$ furthermore the quantitative estimate
\begin{equation*}
	\int_{B_{R/4}^+(x_o)}|D^2u|^2\dx
	\le 
    C C_M^{1/\alpha}\frac1{R^2}\int_{B_{R/2}^+(x_o)}|Du|^2\dx.
\end{equation*}
\item[(ii)] 
If $F$ satisfies (\ref{F-growth}) and $u$ satisfies
$u\in W^{1,4}(B_{R/2}^+(x_o),\R^3)$ and
the smallness conditions
\begin{equation}\label{small-Diri}
   \int_{B_R^+(x_0)}|Du|^2\dx<\eps_1^2
	\quad
	\mbox{and}
	\qquad
   \osc_{B_R^+(x_o)}u\le \eps_1
\end{equation}
we have $u\in W^{2,2}(B_{R/4}^+(x_o),\R^3)$ together 
with the quantitative estimate
\begin{equation*}
   \int_{B_{R/4}^+(x_o)}|D^2u|^2\dx
   \le 
   \frac C{R^2}\int_{B_{R/2}^+(x_o)}|Du|^2\dx
   +
   C\int_{B_{R/2}^+(x_o)}|f|^2\dx
\end{equation*}
with a universal constant $C=C(C_1,\Gamma)$.
\end{enumerate}
\end{theorem}

\begin{proof}
The first part of the proof is identical for both cases.
We will later choose $\eps_1\in(0,\rho_o)$ with the radius
$\rho_o(\Gamma)>0$ from Lemma~\ref{lem:extension}. This implies in
view of our assumption $\osc_{B_R^+(x_o)} u\le\eps_1$, that we have 
\begin{equation}\label{small-ball}
    u(B_{R}^+(x_o))\subset B_{\rho_o}(u(x_o)).
\end{equation}
For $h\in(-\frac R{8},\frac R{8})$, we introduce the notation
\begin{equation*}
    u_h(re^{i\vartheta})
    :=
    u(re^{i(\vartheta+h)})
    \quad
    \mbox{whenever $re^{i\vartheta}\in B_{3R/8}^+(x_o)$.}
\end{equation*}
We note that $re^{i\vartheta}\in B_{3R/8}^+(x_o)$ implies
$re^{i(\vartheta+h)}\in B_{R}^+(x_o)$, so that the
preceding definition makes sense. Further,
the inclusion (\ref{small-ball}) yields
\begin{equation*}
    u(B_{3R/8}^+(x_o)) \subset B_{\rho_o}(u(x_o))
    \quad
    \mbox{and}
    \quad
    u_{\pm h}(B_{3R/8}^+(x_o)) \subset B_{\rho_o}(u(x_o)).
\end{equation*}
Therefore, with Lemma~\ref{lem:extension} we find maps $\Phi,\Phi_h,\Phi_{-h}\in W^{1,2}(B_{3R/8}^+(x_o))$ such that
\begin{equation}\label{Phi-bounds}
    \left\{
    \begin{array}{l}
    |\Phi-\Phi_{\pm h}|\le C|u-u_{\pm h}|,\\[0.7ex]
    |D\Phi|\le C|Du|,\\[0.7ex]
%   |D\Phi_{\pm h}|\le C|Du_{\pm h}|,\\[0.7ex]
    |D\Phi-D\Phi_{\pm h}|\le C|Du-Du_{\pm h}|+C|Du|\,|u-u_{\pm h}|
	\end{array}
   \right.
\end{equation}
hold true a.e. on $B_{3R/8}^+(x_o)$, where $C=C(\Gamma)$, and moreover
\begin{equation*}
     \widehat\gamma\circ \Phi=u,\quad \widehat\gamma\circ \Phi_{\pm
       h}=u_{\pm h}
     \quad\mbox{on }I_{3R/8}(x_o).  
\end{equation*}
Since $u\in\Ss$, we can thus find a map $\varphi\in
\mathcal{T}^\ast(\Gamma)$ with
\begin{equation*}
    \varphi(\vartheta)=\Phi(e^{i\vartheta})
    \quad
    \mbox{whenever $e^{i\vartheta}\in I_{3R/8}(x_o)$.} 
\end{equation*}
Similarly, we define
\begin{equation*}
    \varphi_{\pm h}(\vartheta):=\Phi_{\pm h}(e^{i\vartheta})
    \quad
    \mbox{whenever $e^{i\vartheta}\in I_{3R/8}(x_o)$.} 
\end{equation*}
Next, for $v\colon B_{3R/8}^+(x_o)\to\R^k$ we define the {\bf angular difference quotient} by
\begin{equation*}
    \d_hv:=\tfrac1h(v_h-v).
\end{equation*}
Now we let $s,t$ be arbitrary radii with
$\frac R4\le s<t\le \frac{3R}8$
and choose a cut-off function $\eta\in C^\infty_0(B_t(x_o),[0,1])$ with $\eta\equiv1$ on $B_s(x_o)$ and
$|D\eta|\le\frac 2{t-s}$. Then we would like to test the
inequality~(\ref{local-plateau}) with the testing function
\begin{equation*}
    \d_{-h}\big(\eta^2\d_hu\big)
     = 
     \tfrac1{h^2}\big[\eta^2(u_h-u)+\eta^2_{-h}(u_{-h}-u)\big].     
\end{equation*}
With the abbreviation $\Tilde w:=\eta^2(u_h-u)$, this identity becomes
\begin{equation*}
    \d_{-h}\big(\eta^2\d_hu\big)=\tfrac1{h^2}\big[\Tilde w-\Tilde
    w_{-h}\big]=\tfrac1h\partial_{-h}\Tilde w.
\end{equation*}
Unfortunately,  $\Tilde w$ and $-\Tilde w_{-h}$ are not
admissible in (\ref{local-plateau}) since they might not attain the
right boundary values. Therefore, following Struwe \cite{Struwe:1988-2} we modify the function $\Tilde w$ to
\begin{equation*}
    w
    :=\Tilde w -\eta^2
    \int_\Phi^{\Phi_h}\int_\Phi^t\widehat\gamma''(s)\,ds\,dt
    =:
    \Tilde w- \eta^2g_+
    \quad\mbox{on $B_{3R/8}^+(x_o)$}
\end{equation*}
and extend $w$ by zero outside of $B_{3R/8}^+(x_o)$.
From~(\ref{Phi-bounds})$_1$ we infer 
\begin{equation}\label{g+BoundI}
    |g_+|\le C|\Phi-\Phi_h|^2\le C|u-u_h|^2,
\end{equation}
for a constant $C=C(\Gamma)$.
In order to estimate $Dg_+$, we calculate
\begin{align*}
  Dg_+
  &=D\big[\widehat\gamma(\Phi_h)-\widehat\gamma(\Phi)-\widehat\gamma'(\Phi)(\Phi_h-\Phi)\big]\\
  &=\big(\widehat\gamma'(\Phi_h)-\widehat\gamma'(\Phi)-\widehat\gamma''(\Phi)(\Phi_h-\Phi)\big)D\Phi
    +\big(\widehat\gamma'(\Phi_h)-\widehat\gamma'(\Phi)\big)\big(D\Phi_h-D\Phi\big).
\end{align*}
Since $\widehat\gamma$ is of class $C^3$ we
use the bounds~(\ref{Phi-bounds}) to obtain
\begin{align}\label{g+BoundII}
    |Dg_+|
    &\le C|\Phi-\Phi_h|^2|D\Phi|+C|\Phi-\Phi_h||D\Phi_h-D\Phi|\\\nonumber
    &\le C|u-u_h|^2|Du|+C|u-u_h||Du_h-Du|
\end{align}
a.e. on $B_{3R/8}^+(x_o)$, where $C=C(\Gamma)$. 
Next, we calculate the boundary values of $w$. 
By the choice of $\varphi$ and $\varphi_h$, we have for any
$e^{i\vartheta}\in I_{3R/8}(x_o)$ 
\begin{align*}
    w(e^{i\vartheta})
    &=
    \Tilde w(e^{i\vartheta})-\eta^2(e^{i\vartheta})
    \int_{\varphi(\vartheta)}^{\varphi_h(\vartheta)}
     \int_{\varphi(\vartheta)}^t\widehat\gamma''(s)\,ds\,dt\\
    &=
    \eta^2(e^{i\vartheta})\bigg[
      \widehat\gamma(\varphi_h(\vartheta))
      -\widehat\gamma(\varphi(\vartheta))
      -\int_{\varphi(\vartheta)}^{\varphi_h(\vartheta)}
     \int_{\varphi(\vartheta)}^t\widehat\gamma''(s)\,ds\,dt
      \bigg]\\
    &=
    \eta^2(e^{i\vartheta})
    \widehat\gamma'(\varphi(\vartheta))(\varphi_h(\vartheta)-
    \varphi	(\vartheta)).
\end{align*}
On the other hand, for $e^{i\vartheta}\not\in I_{3R/8}(x_o)$, the choice of $\eta$ implies
$w(e^{i\vartheta})=0$. Defining
\begin{equation*}
	\psi(\vartheta)
   :=
   \eta^2(e^{i\vartheta})\varphi_h(\vartheta)+(1-\eta^2(e^{i\vartheta}))
   \varphi(\vartheta),
\end{equation*}
we thereby deduce
\begin{equation*}
     w(e^{i\vartheta})
     =
     \widehat\gamma'(\varphi(\vartheta))(\psi(\vartheta)-\varphi(\vartheta))
     \quad\mbox{for all }\vartheta\in\R.
\end{equation*}
We observe that $\psi\in C^0\cap W^{1/2,2}(\R)$ is weakly monotone
and satisfies the periodicity condition $\psi(\cdot+2\pi)=\psi+2\pi$ because it is a convex combination of two functions with these properties.
Moreover, it satisfies the three-point condition $\widehat\gamma(\psi(\Theta_k))=Q_k$ since $\eta$ vanishes in $P_k=e^{i\Theta_k}$ because
of the assumption $B_R^+(x_o)\cap \{P_1,P_2,P_3\}=\varnothing$.
We conclude $\psi\in\mathcal{T}^\ast(\Gamma)$, which in turn
implies that $w=\widetilde w-\eta^2g_+\in T_u\mathcal{S}^\ast$ 
is an admissible testing function in (\ref{local-plateau}). From this we infer
\begin{align}\label{testI}
    -\int_{B_R^+(x_o)}&Du\cdot D \Tilde w\dx\\\nonumber
    &\le
    \int_{B_R^+(x_o)}F\cdot \Tilde w\dx 
     +\int_{B_R^+(x_o)}|Du|\,|D(\eta^2g_+)|+\eta^2|F|\,|g_+|\dx.
\end{align}
With $-\Tilde w_{-h}=\eta^2_{-h}(u_{-h}-u)$ we can repeat the preceding arguments with $h, \eta^2$ replaced by $-h,\eta_{-h}^2$. We obtain a function $g_-:B_{3R/8}^+(x_o)\to\R^3$
that satisfies the estimates
\begin{align}\label{g-BoundI}
  |g_-|
  &\le C|\Phi-\Phi_{-h}|^2\le C|u-u_{-h}|^2,\\\label{g-BoundII}
  |Dg_-|
  &\le C|u-u_{-h}|^2|Du|+C|u-u_{-h}||Du_{-h}-Du|,
\end{align}
with a constant $C=C(\Gamma)$, and for which 
$-\Tilde w_{-h}-\eta^2_{-h}g_-$ is admissible in the variational
inequality~(\ref{local-plateau}). This leads to
\begin{align*}
    \int_{B_R^+(x_o)}&Du\cdot D\Tilde w_{-h}\dx\\
    &
    \le
    -\int_{B_R^+(x_o)}F\cdot \Tilde w_{-h}\dx 
    +
    \int_{B_R^+(x_o)}|Du|\,|D(\eta^2_{-h}g_-)|+\eta^2_{-h}|F|\,|g_-|\dx.
\end{align*}
Adding the preceding inequality to~(\ref{testI}) and dividing by $h^2$, we deduce
\begin{align*}
	-\tfrac1h\int_{B_R^+(x_o)}&Du\cdot D\partial_{-h}\Tilde w\dx\\
	&
	\le
	\tfrac1h
	\int_{B_R^+(x_o)} F\cdot \partial_{-h}\Tilde w\dx
	+\tfrac1{h^2}\int_{B_R^+(x_o)}|F|\big(\eta^2|g_+|+\eta^2_{-h}|g_-|\big)\dx\\
	&\phantom{m}+\tfrac1{h^2}\int_{B_R^+(x_o)}|Du|\,\big(|D(\eta^2g_+)|
	+|D\big(\eta^2_{-h}g_-)|\big)\dx.
	% &
	% \le
	% \tfrac1h
	% \int_{B_R^+(x_o)} F\cdot \partial_{-h}\Tilde w\dx
	% +\tfrac1{h^2}\int_{B_R^+(x_o)}|F|\big(\eta^2|g_+|+\eta^2_{-h}|g_-|\big)\dx\\
	% &\phantom{m}+\tfrac1{h^2}\int_{B_R^+(x_o)}|Du|\,\big(\eta^2|Dg_+|
	% +\eta^2_{-h}|Dg_-|\big)\dx\\
	% &
	% \phantom{m}+\tfrac2{h^2}\int_{B_R^+(x_o)}|Du|\,\big(\eta|D\eta||g_+|
	% +\eta_{-h}|D\eta_{-h}||g_-|\big)\dx.
\end{align*}
Taking into account the definition $\frac1h \Tilde w =\eta^2\d_hu$, 
we can re-write the preceding inequality in the form
\begin{align}\label{W22-Bound-I}\nonumber
	-\int_{B_R^+(x_o)}&Du\cdot D\partial_{-h}(\eta^2\d_hu)\dx\\\nonumber
	&\le
	\int_{B_R^+(x_o)} |F|\,|\partial_{-h}(\eta^2\d_hu)|\dx
	+\tfrac1{h^2}\int_{B_R^+(x_o)}|F|\big(\eta^2|g_+|+\eta^2_{-h}|g_-|\big)\dx\\\nonumber
	&\phantom{m}+\tfrac1{h^2}\int_{B_R^+(x_o)}|Du|\,\big(\eta^2|Dg_+|
	+\eta^2_{-h}|Dg_-|\big)\dx\\
	&
	\phantom{m}+\tfrac2{h^2}\int_{B_R^+(x_o)}|Du|\,\big(\eta|D\eta||g_+|
	+\eta_{-h}|D\eta_{-h}||g_-|\big)\dx\\\nonumber
	&=I+II+III+IV,
\end{align}
with the obvious meaning of $I$ - $IV$. In the sequel we estimate 
these terms separately.  We start with the
{\bf estimate of $II$.} Here we use~(\ref{g+BoundI})
and~(\ref{g-BoundI}) and the fact that $0\le\eta\le 1$
to obtain
\begin{align*}
	II
	&\le
	C\int_{B_R^+(x_o)}|F|\big(\eta|\partial_hu|^2
	+\eta_{-h}|\partial_{-h}u|^2\big)\dx.
\end{align*}
Next we deduce the {\bf estimate for $IV$.} Here we use again~(\ref{g+BoundI})
and~(\ref{g-BoundI}) to obtain
\begin{align*}
	IV
	&\le
	C\int_{B_R^+(x_o)}|Du|\big(\eta |D\eta ||\partial_hu|^2
	+\eta_{-h}|D\eta_{-h}||\partial_{-h}u|^2\big)\dx\\
	&\le
	C\int_{B_R^+(x_o)}\tfrac{|Du|}{t-s} 
	\big(\eta |\partial_hu|^2
	+\eta_{-h}|\partial_{-h}u|^2\big)\dx.
	\end{align*}
The {\bf Estimate of $III$} is achieved as
follows. Using~(\ref{g+BoundII}), (\ref{g-BoundII}) and Young's inequality we find
\begin{align*}
	III
	&\le
	C\int_{B_R^+(x_o)}|Du|\big(\eta^2 |\partial_hu| |D\partial_hu|
	+\eta_{-h}^2|\partial_{-h}u||D\partial_{-h}u| \big)\dx\\
	&\phantom{\le}
	+C \int_{B_R^+(x_o)}|Du|^2\big(\eta |\partial_hu|^2
	+\eta_{-h}|\partial_{-h}u|^2 \big)\dx\\
	&\le
	\mu\int_{B_R^+(x_o)}\eta^2|D\partial_hu|^2+\eta_{-h}^2
	|D\partial_{-h}u|^2\dx\\
	&\phantom{\le}
	+
	C_\mu
	\int_{B_R^+(x_o)}|Du|^2\big(\eta |\partial_hu|^2
	+\eta_{-h}|\partial_{-h}u|^2 \big)\dx.
\end{align*}
Adding the estimates for $II$, $III$ and $IV$ and using
$2\tfrac{|Du|}{t-s}\le|Du|^2+\frac1{(t-s)^2}$ we deduce
\begin{align*}
	II+III+IV
	&\le
	\mu\int_{B_R^+(x_o)}\eta^2|D\partial_hu|^2+\eta_{-h}^2
	|D\partial_{-h}u|^2\dx\\
	&\phantom{\le}
	+
	C_\mu
	\int_{B_R^+(x_o)}\big(|Du|^2+\tfrac{1}{(t-s)^2}+|F|\big)\big(\eta |\partial_hu|^2
	+\eta_{-h}|\partial_{-h}u|^2 \big)\dx.
\end{align*}
Next, we consider the term $I$. Here, by Young's inequality
and a standard estimate for difference quotients we find
\begin{align*}
	I
	&
	=\int_{B_R^+(x_o)} |F|\,|\partial_{-h}(\eta^2\d_hu)|\dx\\
	&
	\le
	\mu\int_{B_R^+(x_o)}|D(\eta^2\d_hu)|^2\dx
	+C_\mu\int_{B_{t+h}^+(x_o)}|F|^2\dx.
\end{align*}
In the last line we used the fact that
$\spt (\eta^2\partial_hu)\subset B_t^+(x_o)$ and $e^{-ih}B_t^+(x_o)
\subset B_{t+h}^+(x_o)$.
For the  estimate of the left-hand side of~(\ref{W22-Bound-I}) from below we compute
\begin{align*}
  -\int_{B_R^+(x_o)}Du\cdot D\d_{-h} (\eta^2\partial_h u)\dx
  &=
  \int_{B_R^+(x_o)}D\d_hu\cdot D(\eta^2\d_hu)\dx\\
  &=\int_{B_R^+(x_o)}\eta^2|D\d_hu|^2
   +2\eta D\d_hu\cdot \d_hu\otimes D\eta\dx\\
   &\ge
   \tfrac12\int_{B_R^+(x_o)}\eta^2|D\d_hu|^2\dx
   -
   C\int_{B_R^+(x_o)}|\d_hu|^2|D\eta|^2\dx\\
   &\ge
   \tfrac12\int_{B_R^+(x_o)}\eta^2|D\d_hu|^2\dx
   -
   C\int_{B_t^+(x_o)}\tfrac{|\d_hu|^2}{(t-s)^2}\dx.
\end{align*}
Joining the preceding estimates we arrive at
\begin{align}\label{W22-Bound-II}\nonumber
   \tfrac12\int_{B_R^+(x_o)}&\eta^2|D\d_hu|^2\dx\\
   &
   \le2\mu\int_{B_R^+(x_o)}\eta^2|D\partial_hu|^2+\eta_{-h}^2
	|D\partial_{-h}u|^2\dx\\\nonumber
	&\phantom{\le}
	+
	C_\mu
	\int_{B_R^+(x_o)}\big(|Du|^2+\tfrac{1}{(t-s)^2}+|F|\big)\big(\eta |\partial_hu|^2
	+\eta_{-h}|\partial_{-h}u|^2 \big)\dx\\\nonumber
	&\phantom{\le}
	+C_\mu\int_{B_{t+h}^+(x_o)}|F|^2\dx
        +C\int_{B_t^+(x_o)}\tfrac{|\d_hu|^2}{(t-s)^2}\dx
        .
\end{align}
From the transformation $x=ye^{-ih}$ we infer the identity
\begin{equation*}
  \int_{B_R^+(x_o)}\eta^2|D\partial_hu|^2\dx=\int_{B_R^+(x_o)}\eta_{-h}^2
	|D\partial_{-h}u|^2\,dy.
\end{equation*}
Therefore, we can re-absorb the first integral from the right-hand
side of~(\ref{W22-Bound-II}) into the left after choosing $\mu\in(0,1)$
sufficiently small. 
Keeping in mind the properties of the cut-off function $\eta$, we deduce
\begin{align}\label{W22-Bound-III}\nonumber
    \int_{B_s^+(x_o)}|D\d_hu|^2\dx\nonumber
    &\le C\int_{B_R^+(x_o)}
    \big(|Du|^2
    +|F|\big)\,(\eta|\d_hu|^2
    +\eta_{-h}|\d_{-h}u|^2)\dx\\
    &\phantom{\le}+C\int_{B_{t+h}^+(x_o)}|F|^2\dx
    +C\int_{B_t^+(x_o)}\tfrac{|\d_hu|^2}{(t-s)^2}\dx
\end{align}
with a constant $C=C(\Gamma)$. For the bound of the right-hand side we distinguish between the two cases (i) and (ii).
We begin with the

{\bf Proof of (i).} In this case, we follow the strategy
from \cite[p. 73]{Struwe:1988-2} and first extend the map $u$
by reflection $u(x):=u(x/|x|^2)$ onto the full disk $B_{R/2}(x_o)$
and then cover $B_t(x_o)$ with disks $B_\rho(x_i)$ of radius
$\rho\in(0,\frac R{16})$. This can be done in such a way that every point
$x\in B_t(x_o)$ is contained in at most $N$ of the disks
$B_{2\rho}(x_i)$ with $N$ independent from $\rho$. We choose a standard cut-off function $\zeta\in C^\infty_0(B_2(0))$
with $\zeta\equiv1$ on $B_1(0)$ and let
$\zeta_i(x):=\zeta(\frac{x-x_i}\rho)$. Then we estimate
\begin{align*}
    \int_{B_R^+(x_o)}&
     |Du|^2(\eta|\d_hu|^2+\eta_{-h}|\d_{-h}u|^2)\dx\\
    &=\int_{B_t^+(x_o)}
     \eta(|Du|^2+|Du_h|^2)|\d_hu|^2\dx\\
    &
    \le
    \sum_i \int_{B_{2\rho}(x_i)\cap B_t(x_o)}
    (|Du|^2+|Du_h|^2)|\d_hu|^2\zeta_i^2\dx.
\end{align*}
Because of the Morrey type assumption~(\ref{Morrey-assumption}),
each of the latter integrals can be estimated by 
Lemma~\ref{lem:Morrey} with $\Omega=B_{2\rho}(x_i)\cap B_t(x_o)$,
$w=|Du|^2+|Du_h|^2$, $v_i=|\partial_hu|\zeta_i$ and $C_o=C_MR^{-2\alpha}$. This leads us to 
\begin{align*}
    \int_{B_R^+(x_o)}
    &
    |Du|^2(\eta|\d_hu|^2+\eta_{-h}|\d_{-h}u|^2)\dx \\
    &
    \le
    C C_M \Big(\frac\rho R\Big)^{2\alpha} \sum_i
    \int_{B_{2\rho}(x_i)\cap B_t(x_o)}|Dv_i|^2\\
    &
    \le
    C C_M N\Big(\frac\rho R\Big)^{2\alpha}
    \int_{B_t^+(x_o)} |D\d_hu|^2
     +\tfrac1{\rho^2}|\d_hu|^2\dx.
\end{align*}
At this stage we fix $\rho\in(0,\frac R{16})$ in the
form $\rho^{2\alpha}:=R^{2\alpha}/(2C C_MN)$. Plugging
the resulting estimate into~(\ref{W22-Bound-III}) and
keeping in mind $F=0$ and $C_M\ge1$, we arrive at
\begin{align*}
    \int_{B_s^+(x_o)}|D\d_hu|^2\dx
    &\le 
    \tfrac 12\int_{B_t^+(x_o)}|D\d_hu|^2\dx
    +C\tfrac {C_M^{1/\alpha}}{(t-s)^2}
    \int_{B_{R/2}^+(x_o)}|\d_hu|^2\dx.
\end{align*}
Here, we can re-absorb the first integral on the right-hand side
by means of Lemma~\ref{lem:Giaq}. Letting $h\downarrow0$, we thereby
deduce
\begin{equation}\label{W22-Bound(i)}
    \int_{B_{R/4}^+(x_o)}|D_\vartheta Du|^2\dx
    \le 
    C C_M^{1/\alpha}\tfrac 1{R^2}\int_{B_{R/2}^+(x_o)}|Du|^2\dx,
  \end{equation}
where $D_\vartheta$ denotes the angular derivative. It remains to
estimate the second radial derivative $D_r^2u$. For this we rewrite the Laplacian in polar coordinates as
$\Delta=D^2_ru+\frac1rD_ru+\frac1{r^2}D_\vartheta^2 u$. 
Since $u$ is harmonic on $B_R^+(x_o)$, this implies
$|D_r^2u|^2\le C|D_\vartheta Du|^2+C|Du|^2$ and the claim~(i)
follows from~(\ref{W22-Bound(i)}) with a constant
$C=C(\Gamma ,\alpha)$. Now we proceed to the

{\bf Proof of (ii).} Under the assumption $u\in
W^{1,4}(B_R^+(x_o),\R^3)$, the estimate~(\ref{W22-Bound-III})
readily implies
\begin{align*}
    \int_{B_s^+(x_o)} |\d_hDu|^2\dx
    \le
    C\int_{B_{t+h}^+(x_o)}|f|^2 
    +|Du|^4+\tfrac1{(t-s)^2}|Du|^2\dx
\end{align*}
by the bound~(\ref{F-growth}) on $F$, where $C=C(C_1,\Gamma)$. Since
the right-hand side is bounded independently from $h$, we infer
after letting $h\to0$  
\begin{equation}\label{angular-Deri}
    \int_{B_s^+(x_o)}|D_\vartheta Du|^2\dx
    \le 
    C\int_{B_t^+(x_o)}|f|^2 +|Du|^4+\tfrac1{(t-s)^2}|Du|^2\dx.
\end{equation}
Similarly as in the proof of (i), we use 
$\Delta=D^2_ru+\frac1rD_ru+\frac1{r^2}D_\vartheta^2 u$, the
equation $\Delta u=F$ on $B_R^+(x_o)$ in the sense of
distributions and also the assumption $R<\frac 12$ which implies $\frac1r<2$ on $B_R^+(x_o)$. Combining this with~(\ref{angular-Deri}), we deduce
\begin{align*}
    \int_{B_s^+(x_o)}|D_r^2u|^2\dx
    &\le C\int_{B_s^+(x_o)}|F|^2+|D_ru|^2+|D^2_\vartheta u|^2\dx\\
    &\le C\int_{B_t^+(x_o)}|f|^2 +|Du|^4+\tfrac1{(t-s)^2}|Du|^2\dx,
\end{align*}
with $C=C(C_1,\Gamma)$, where here, the last estimate follows from assumption (\ref{F-growth})
and (\ref{angular-Deri}). Combining this
with~(\ref{angular-Deri}) and applying the interpolation
inequality~(\ref{appl-gag-nir}), we arrive at
\begin{align*}
    &\int_{B_s^+(x_o)}|D^2u|^2\dx
    \le C\int_{B_t^+(x_o)}|f|^2
    +|Du|^4+\tfrac1{(t-s)^2}|Du|^2\dx\\
    &\qquad\le C\eps_1^2 \int_{B_t^+(x_o)}|D^2u|^2 +
    \tfrac1{R^2}|Du|^2\, dx
        +C\int_{B_t^+(x_o)}|f|^2 +\tfrac1{(t-s)^2}|Du|^2\dx\,.
\end{align*}
In the last step, we also used the assumption~(\ref{small-Diri}). If
we decrease once more the value of $\eps_1$ in such a way that
$C\eps_1^2\le\tfrac12$, we arrive at
\begin{align*}
    \int_{B_s^+(x_o)}&|D^2u|^2 \dx\\
    &\le 
    \tfrac12 \int_{B_t^+(x_o)}|D^2u|^2\,dx
    +\tfrac C{(t-s)^2}\int_{B_{R/2}^+(x_o)}|Du|^2\dx
    +C\int_{B_{R/2}^+(x_o)}|f|^2 \dx.
\end{align*}
Since this inequality holds for all $s,t$ with
$\frac R4\le s<t\le\tfrac{3R}8<\frac R2$, the Iteration
Lemma~\ref{lem:Giaq} now implies the claim (ii).
\end{proof}

\section{Calder\'on-Zygmund estimates for solutions}
In this section we remove the $W^{1,4}$-hypothesis from the last Chapter which was needed to
establish the local $W^{2,2}$-estimates in Theorems
\ref{thm:apriori-W22-int} and \ref{thm:apriori-W22-bdry} (ii). 

\subsection{Results for comparison problems}
In this section, we provide some results for harmonic maps with a
partial Plateau boundary condition. These maps  will serve as comparison maps later. More precisely, we consider minimizers of the Dirichlet energy in the class
\begin{equation*}
    \mathcal{S}_u^\ast(\Gamma):=\big\{w\in\Ss\,: 
    w=u\mbox{ on }B		\setminus B_R^+(x_o)\}
\end{equation*}
for some $x_o\in\partial B$. Minimizers $v\in\Ss$ of
this problem are harmonic on $B_R^+(x_o)$ and satisfy a 
weak Neumann type boundary condition on $I_R(x_o)$. More precisely, we
have
\begin{lemma}\label{lem:harmonic_plateau}
  Every minimizer $v$ of $\,\mathbf{D}$ in the class
  $\mathcal{S}_u^\ast(\Gamma)$ satisfies 
  \begin{equation}
   \label{Harmonic_Plateau}
   \int_{B_R^+(x_o)}Dv\cdot Dw\dx \ge0
\end{equation}
for all $w\in T_v\mathcal{S}^\ast$ 
with $w=0$ on $S_R^+(x_o)$ in the sense of traces.

\begin{proof}
The proof is a slight modification of
Lemma~\ref{lem:rand-variationen}\,(i). We assume that the boundary
values of $v$ are given by $\varphi\in\mathcal{T}^\ast(\Gamma)$ in
the sense $v(e^{i\vartheta})=\widehat\gamma(\varphi(\vartheta))$. For a
given $w\in T_v\mathcal{S}^\ast$ 
with $w=0$ on $S_R^+(x_o)$ in the sense of traces,
we can find a $\psi\in\mathcal{T}^\ast(\Gamma)$ with
$w(e^{i\vartheta})=\widehat\gamma'(\varphi)(\psi-\phi)$ for all
$\vartheta\in\R$ and $\psi(\vartheta)=\varphi(\vartheta)$
whenever $e^{i\vartheta}\not\in I_R(x_o)$. We define $h_s:B\to\R^3$,
$0\le s\ll 1$ as the unique minimizer of the Dirichlet energy with $h_s=u$ on $B\setminus B_R^+(x_o)$ and
$h_s(e^{i\vartheta})=\widehat\gamma(\varphi+s(\psi-\varphi))$ for all
$e^{i\vartheta}\in I_R(x_o)$. As in the proof of
Lemma~\ref{lem:rand-variationen}\,(i), we then define
\begin{equation*}
    v_s:= h_s+s(w-\tfrac{\partial h_s}{\partial s}\big|_{s=0})-(h_0-u)
\end{equation*}
and check that the maps $v_s\in\mathcal{S}_u^\ast(\Gamma)$ are admissible as competitors for $v$ with
$\tfrac{\partial}{\partial s}\big|_{s=0}v_s=w$. 
We conclude the claim \eqref{Harmonic_Plateau}
by $ 0\le \frac{d}{ds}\Big|_{s=0}\mathbf{D}(v_s)=\int_{B_R^+(x_o)} Dv\cdot Dw\dx$.\qedhere
\end{proof}
  
\end{lemma}
Next, we give an existence result concerning $\mathbf D$-minimizers
in the class $\mathcal{S}_u^\ast(\Gamma)$.
\begin{lemma}\label{lem:exist-comparison}
   For every map $u\in\Ss$ and every disk $B_R^+(x_o)$ with center
   $x_o\in\partial B$, there is a minimizer $v\in\Ss$ of the Dirichlet
   energy $\mathbf{D}$ in the class $\mathcal{S}_u^\ast(\Gamma)$.
 \end{lemma}
 
\begin{proof}
We choose a minimizing sequence $v_k\in \mathcal{S}^\ast_u(\Gamma)$
for $\mathbf{D}$. Since the boundary traces of the $v_k$ are
contained in the compact set $\Gamma$, the $W^{1,2}$-norms of $v_k$
are uniformly bounded. Therefore, we can assume $v_k\wto v$ in
$W^{1,2}(B,\R^3)$ and almost everywhere, as
$k\to\infty$. Furthermore, Lemma~\ref{lem:compact traces} implies
uniform convergence  $v_k|_{\partial B}\to v|_{\partial B}$ of the
boundary traces. From this we deduce that the limit map again satisfies $v\in\mathcal{S}^\ast_u(\Gamma)$. The lower
semicontinuity of the Dirichlet energy $\mathbf{D}$ with
respect to weak $W^{1,2}$-convergence then yields
the claim.
\end{proof}

The main result of this section are the following $W^{2,2}$-estimates
for solutions of the comparison problem. 

\begin{lemma}\label{lem:W22-comparison}
For  a map $u\in\Ss$, a center $x_o\in\partial B$ and a radius
$R\in(0,\frac12)$ with $B_R^+(x_o)\cap\{P_1,P_2,P_3\}=\varnothing$, we
consider a minimizer $v$ of the 
Dirichlet energy in the class $\mathcal{S}^\ast_u(\Gamma)$ with
$\int_{B_R^+(x_o)}|Du|^2\dx\le E_o$ for some constant $E_o\ge1$. There
is a constant $\eps_2=\eps_2(\delta,M,\{Q_i\})>0$ such that the
smallness condition
\begin{equation*}
    \osc_{B_R^+(x_o)}v\le \eps_2
\end{equation*}
implies $v\in W^{2,2}(B_{R/4}^+(x_o),\R^3)$, and for some
constant $C=C(\Gamma,E_o)$, we have the quantitative estimate
\begin{equation*}
    \int_{B_{R/4}^+(x_o)}|D^2v|^2\dx
    \le 
    \frac C{R^2}\int_{B_{R/2}^+(x_o)}|Dv|^2\dx.
\end{equation*}
\end{lemma}

\begin{proof}
Since the minimizer $v$ satisfies the Euler-Lagrange
equations~(\ref{Harmonic_Plateau}), the claim follows from
Theorem~\ref{thm:apriori-W22-bdry}\,(i) as soon as we have
established the Morrey-type bound~(\ref{Morrey-assumption}). 
To this end, we choose a radius $\rho_1=\rho_1(\Gamma)>0$ so small that every ball
$B^3_{\rho_1}(p)\subset\R^3$ contains at most one of the points
$Q_1,Q_2,Q_3$. Then we define
\begin{equation}\label{eps-choice}
    \eps_2:=\min\big\{\eps_1,\delta, \tfrac{\rho_1}M\big\}
\end{equation}
with the constant $\eps_1$ from Theorem~\ref{thm:apriori-W22-bdry}\,(i)
and the parameters $\delta>0$ and $M\ge1$ from the chord-arc
condition \eqref{chord-arc} of $\Gamma$. For a fixed $y\in I_{R/2}(x_o)
\subset\partial B$ we consider the function
\begin{equation*}
    \big[0, \tfrac R2\big]\ni r\mapsto \Phi(r)
    :=
    {\bf D}_{B_r^+(y)}(v)=\tfrac12 \int_0^r\int_{S_\rho^+(y)}
    \Big[\Big| \frac{\partial v}{\partial\rho}\Big|^2
    +\Big|\frac{\partial v}{\partial\omega}\Big|^2\Big]\, d\H^1\omega\, d\rho,
\end{equation*}
where $S_\rho^+(y)= \partial B_\rho(y)\cap\overline B$ and
$\H^1$ denotes the $1$-dimensional Hausdorff measure on $\R^2$.
Since the function $\Phi$ is absolutely continuous, we know that for
almost every $r\in [0,\frac R2]$ there holds
\begin{equation}\label{Psi}
    \Psi(r)
    :=
    \frac r2 
    \int_{S_r^+(y)}\Big|\frac{\partial v}{\partial\omega}
    \Big|^2d\H^1\omega \le r\Phi'(r).
\end{equation}
From now on we consider only such $r$ for which \eqref{Psi} holds,
so that the minimizer $v$ is continuous on $S_r^+(y)$ by the
Sobolev embedding theorem. Writing
$\{x_r,y_r\}:=S_r^+(y)\cap\partial B$, we can thereby estimate
\begin{equation}\label{osc-u-r}
    |v(x_r)-v(y_r)|^2
    \le 
    \bigg(\int_{S_r^+(x_o)}
    \Big|\frac{\partial v}{\partial\omega}\Big|
    \,d\H^1\omega\bigg)^2\le 
    2\pi\Psi(r)\le 2\pi r\Phi'(r).   
\end{equation}
Since $\osc_{B_R^+(x_o)} v\le\eps_2\le\delta$, the chord-arc condition
\eqref{chord-arc} implies the existence of a sub-arc
$\Gamma_r\subset\Gamma$ connecting the points $v(x_r)$ and $v(y_r)$ with 
\begin{equation}\label{length-bdry}
    L(\Gamma_r)
    \le 
    M|v(x_r)-v(y_r)|\le M\eps_2\le\rho_1,
\end{equation}
where we used the choice of $\eps_2$ in the last step. 
From the choice of $\rho_1$ we thereby infer that this sub-arc
$\Gamma_r$ contains at most one of the points $Q_1,Q_2,Q_3$, which
implies $\Gamma_r=v(I_r(y))$. Indeed, if this was not the case, the
sub-arc $\Gamma_r=v(\partial B\setminus I_r(y))$ would contain all
three of the points $Q_1,Q_2,Q_3$, which is a contradiction.
Combining~(\ref{length-bdry}) and~(\ref{osc-u-r}), we
thereby deduce
\begin{equation}\label{length-bdry2}
    L^2(v(I_r(y))
    \le M^2|v(x_r)-v(y_r)|^2\le 2\pi M^2 r\Phi'(r).
\end{equation}
Our next goal is to estimate ${\bf D}_{B_r^+(y)}(v)$ by
constructing a suitable comparison map. To this end, we define
$c_r\colon I_r(y)\to\Gamma_r$ as the orientation preserving parametrization of the sub-arc
$\Gamma_r=v(I_r(y))$ proportionally to arc length.  From this
choice of the parametrization, we infer
\begin{equation}\label{energy-bdry}
    r\int_{I_r(y)}|c_r'|^2\,d\H^1
    = 
    \frac {r\,L^2(c_r)}{L(I_r(y))}
    \le L^2(v(I_r(y)))
    \le 2\pi M^2r\Phi'(r),
\end{equation}
where we employed \eqref{length-bdry2} in the last step.  Next we
define Dirichlet boundary values on $\partial B_r^+(y)=
S_r^+(y)\cup I_r(y)$ by
\begin{equation*}
    g(x)
    :=
    \begin{cases}
      v(x)&\mbox{for }x\in S_r^+(y),\\
      c_r(x)&\mbox{for }x\in I_r(y),
    \end{cases}
\end{equation*}
and define $w\in W^{1,2}(B_r^+(y),\R^3)$ as the minimizer of the Dirichlet
energy with boundary data $g$. By comparing $w$ 
with a suitable cone with the same boundary values $g$, we infer 
\begin{align}\label{energy-w}
    {\bf D}_{B_r^+(y)}(w) 
    &\le 
    Cr\int_{\partial B_r^+(y)}
    \Big|\frac{\partial g}{\partial\omega}\Big|^2\,
    d\H^1\omega
    \\\nonumber 
    &= 
    Cr\int_{S_r^+(y)}
    \Big|\frac{\partial v}{\partial\omega}\Big|^2\, d\H^1\omega
    +Cr\int_{I_r(y)}|c_r'|^2\,d\H^1\omega 
    \\\nonumber 
    &\le 
    C\Psi(r)+
    CM^2 r\Phi'(r)\le CM^2r\Phi'(r),
\end{align}
where we used the definition of $\Psi$ and the estimates
\eqref{energy-bdry} and \eqref{Psi} in the last two steps. 
On the other hand, extending $w$ by $v$ outside of $B_r^+(y)$, we
get an admissible comparison map for $v$, so that the minimizing
property of $v$ implies
\begin{equation*}
    \Phi(r)
    =
    {\bf D}_{B_r^+(y)}(v)\le {\bf D}_{B_r^+(y)}(w)\le mr\Phi'(r)
\end{equation*}
with a constant $m=m(\Gamma)$, or equivalently
\begin{align*}
    \tfrac{d}{d\tau }\big(\tau^{-\frac1{m}}\Phi (\tau )\big)
    =
    \tfrac{1}{m}\tau ^{-\frac{1}{m}-1}\big[m\tau \Phi'(\tau
    )-\Phi (\tau )\big] \ge 0.
\end{align*}
Integrating over $[\rho ,\frac R2]$ and abbreviating
$\alpha:=\frac1{2m}$, we arrive at
\begin{equation}\label{Morrey-bdry}
    \rho^{-2\alpha}\int_{B_\rho^+(y)}|Dv|^2\dx 
    \le
    \big(\tfrac R2\big)^{-2\alpha}\Phi(\tfrac R2)
    \le CR^{-2\alpha}\int_{B_R^+(x_o)}|Dv|^2\dx.
\end{equation}
This is the desired Morrey estimate for points $y\in
I_{R/2}(x_o)\subset\partial B$ and radii $\rho\in (0,\frac{ R}2]$. Now we consider an arbitrary point $y\in B_{R/2}^+(x_o)$ and arbitrary radii $\rho\in(0,\frac R4]$. 
We abbreviate $R_y:=1-|y|\in(0,\frac R2]$ and $y':=\frac y{|y|}$ and
distinguish between the cases $0<\rho< R_y$ and $\rho\ge R_y$. In the
latter case, there holds $B_\rho^+(y)\subset B_{2\rho}^+(y')$ and
therefore, we deduce from~(\ref{Morrey-bdry}) 
\begin{equation}\label{Morrey-near-bdry}
    \rho^{-2\alpha}\int_{B_\rho^+(y)}|Dv|^2\dx 
    \le
    \rho^{-2\alpha}\int_{B_{2\rho}^+(y')}|Dv|^2\dx 
    \le
    CR^{-2\alpha}\int_{B_R^+(x_o)}|Dv|^2\dx.
\end{equation}
We turn our attention to the remaining case $\rho< R_y$. First, for
$\rho\le\frac12R_y$ we use the mean value property of
harmonic maps with the result 
\begin{align*}
    \int_{B_\rho(y)}|Dv|^2\dx
     &
     \le 
     C\rho^2\|Dv\|_{L^\infty(B_{R_y/2}(y))}^2
     \le 
     C\Big(\frac\rho {R_y}\Big)^2\int_{B_{R_y}(y)}|Dv|^2\dx.
\end{align*}
For $\rho\in(\frac12 R_y,R_y)$, the same estimate holds
trivially. Combining this with~(\ref{Morrey-near-bdry}) for
$\rho=R_y$, we deduce 
\begin{align*}
    \int_{B_\rho(y)}|Dv|^2\dx
     \le 
     C\Big(\frac\rho {R_y}\Big)^2\Big(\frac
     {R_y}R\Big)^{2\alpha}\int_{B_R^+(x_o)}|Dv|^2\dx
     \le
     C\Big(\frac\rho {R}\Big)^{2\alpha}
     \int_{B_R^+(x_o)}|Dv|^2\dx,
\end{align*}
where we used $\rho< R_y$ in the last step. Summarizing, for
every $y\in B_{R/2}^+(x_o)$ and every $\rho\in(0,\frac R4]$ we infer
the bound 
\begin{equation*}
    \rho^{-2\alpha}\int_{B_\rho^+(y)}|Dv|^2\dx
    \le
    CR^{-2\alpha}\int_{B_R^+(x_o)}|Dv|^2\dx,
\end{equation*}
and for $\rho\in(\frac R4,\frac R2]$, it holds trivially. This
completes the proof of the desired Morrey
estimate~(\ref{Morrey-assumption}) with
$C_M=C\int_{B_R^+(x_o)}|Dv|^2\dx\le CE_o$ and therefore,
Theorem~\ref{thm:apriori-W22-bdry}\;(i) yields the claimed $W^{2,2}$-estimate.
\end{proof}

\subsection{$W^{1,4}$-regularity for solutions}
We start with a comparison estimate for two solutions
of~(\ref{local-plateau}).
\begin{lemma}\label{lem-Comp-Est}
We consider a disk $B_r^+(x_o)$ with $x_o\in\partial B$  and 
$F_1,F_2\in L^1(B_r^+(x_o),\R^3)$.
Assume that $u_k\in\Ss$, $k=1,2$, are solutions to 
\begin{equation}\label{local-plateau-II}
   \int_{B_r^+(x_o)}Du_k\cdot Dw\dx+ 
   \int_{B_r^+(x_o)}F_k\cdot w\dx		
   \ge
   0
\end{equation}
for all $w\in T_{u_k}\mathcal{S}^\ast$ 
with $w=0$ on $S_r^+(x_o)$ in the sense of traces.
Moreover, we assume $u_1=u_2$ on $S_r^+(x_o)$ in the trace sense
and
\begin{equation}\label{small-image}
    u_k(B_r^+(x_o))\subset B_{\rho_o}(p_o),\quad \mbox{$k=1,2$, for some $\,p_o\in\Gamma$,}
\end{equation}
where $\rho_o=\rho_o(\Gamma)$ is the radius determined in
Lemma~\ref{lem:extension}.
Then, with a universal constant $C=C(\Gamma)$ the
following comparison estimate  holds true:
\begin{align*}
   \int_{B_r^+(x_o)}&|Du_1-Du_2|^2\dx\\
   &
   \le 
   C\|u_1-u_2\|_{L^\infty(B_r^+(x_o))}
   \int_{B_r^+(x_o)}|F_1|+|F_2|+|Du_1|^2+|Du_2|^2\dx.
\end{align*}
\end{lemma}
\begin{proof}
Lemma~\ref{lem:extension} guarantees the existence of maps
$\Phi_1,\Phi_2\in W^{1,2}(B_r^+(x_o),\R)$ with
\begin{equation}
    \label{Phi-bounds2}
    \left\{
    \begin{array}{l}
    |\Phi_1-\Phi_2|\le C|u_1-u_2|,\\[0.7ex]
    |D\Phi_1|+|D\Phi_2|\le C(|Du_1|+|Du_2|)\\[0.7ex]
    %|D\Phi_1-D\Phi_2|\le C|Du_1-Du_2|+C|Du_1|\,|u_1-u_2|,
	 \end{array}
    \right.
\end{equation}
a.e. on $B_r^+(x_o)$ with a constant $C=C(\Gamma)$, 
and 
\begin{equation*}
     \widehat\gamma\circ \Phi_k=u_k,
     \quad\mbox{on $I_r(x_o)$ for $k=1,2$.}  
\end{equation*}
Since $u_k\in\Ss$, we can find maps $\varphi_k\in
\mathcal{T}^\ast(\Gamma)$ with
\begin{equation*}
    \varphi_k(\vartheta)
    =
    \Phi_k(e^{i\vartheta})\quad\mbox{provided
    $e^{i\vartheta}\in I_r(x_o)$. }
\end{equation*}
We define a testing function by
\begin{equation*}
    w_1
    := 
    u_2-u_1
    -
    \int_{\Phi_1}^{\Phi_2}\int_{\Phi_1}^t\widehat\gamma''(s)\,ds\,dt
    =:
    u_2-u_1- g_1.
\end{equation*}
In order to check that this function is admissible in the inequality
for $u_1$, we calculate the
boundary values of $w$ in  points $e^{i\vartheta}\in I_r(x_o)$ by
\begin{align*}
    w_1(e^{i\vartheta})
    &=
    \widehat\gamma(\varphi_2(\vartheta))
    -
    \widehat\gamma(\varphi_1(\vartheta))
    -
    \int_{\varphi_1(\vartheta)}^{\varphi_2(\vartheta)}
    \int_{\varphi_1(\vartheta)}^t\widehat\gamma''(s)\,ds\,dt\\
    &=
    \widehat\gamma'(\varphi_1(\vartheta))
    (\varphi_2(\vartheta)-\varphi_1(\vartheta)).
\end{align*}
Since $\varphi_2\in\mathcal{T}^\ast(\Gamma)$, this implies that
$w_1\in T_{u_1}\mathcal{S}^\ast$. 
Moreover, from~(\ref{Phi-bounds2}) and $u_1=u_2$ on $S_r^+(x_o)$ 
we infer that also $\Phi_1=\Phi_2$ on $S_r^+(x_o)$, in
the sense of traces.
By definition, we thus have $w_1=0$ on $S_r^+(x_o)$ in the trace
sense. 
Therefore, $w_1=u_2-u_1-g_1$ is an admissible testing function
in the inequality~(\ref{local-plateau-II}) for $u_1$. This provides
us with the estimate
\begin{equation}\label{plateau-u1}
   \int_{B_r^+(x_o)}Du_1\cdot D(u_1-u_2)\dx
   \le
   \int_{B_r^+(x_o)}F_1\cdot (u_2-u_1-g_1)-Du_1\cdot Dg_1\dx.
\end{equation}
Similarly as above, one checks that
\begin{equation*}
    w_2
    := 
    u_1-u_2
    -
    \int_{\Phi_2}^{\Phi_1}\int_{\Phi_2}^t\hat\gamma''(s)\,ds\,dt
    =:
    u_1-u_2- g_2
\end{equation*}
is an admissible testing function in the
inequality~(\ref{local-plateau-II}) for $u_2$. This implies 
\begin{equation}\label{plateau-u2}
   -\int_{B_r^+(x_o)}Du_2\cdot D(u_1-u_2)\dx
   \le 
   \int_{B_r^+(x_o)}F_2\cdot (u_1-u_2-g_2)-Du_2\cdot Dg_2\dx.
\end{equation}
Adding the inequalities~(\ref{plateau-u1}) and~(\ref{plateau-u2}), we
arrive at
\begin{align}\label{CompI}
   &\int_{B_r^+(x_o)} |Du_1-Du_2|^2\dx\\\nonumber
   &\qquad\le \int_{B_r^+(x_o)}(F_1-F_2)\cdot (u_2-u_1) 
     -\sum_{k=1}^2(F_k\cdot g_k+Du_k\cdot Dg_k)\dx.
\end{align}
Next, we observe that the
definition of $g_k$ and the bounds~(\ref{Phi-bounds2}) imply 
for $k=1,2$ almost everywhere on $B_r^+(x_o)$ that
\begin{align*}
   |g_k|
   &\le C|\Phi_1-\Phi_2|^2\le C|u_1-u_2|^2,\\
   |Dg_k|
   &\le C(|D\Phi_1|+|D\Phi_2|)|\Phi_1-\Phi_2|
   \le C\big(|Du_1|+|Du_2|\big)
   |u_1-u_2|,
\end{align*}
holds true. Here $C=C(\Gamma)$. In particular
we have $|u_1-u_2|\le 2\rho_o=C(\Gamma)$. Using the
preceding bounds in~(\ref{CompI}), we obtain
 \begin{align*}
   \int_{B_r^+(x_o)}& |Du_1-Du_2|^2\dx\\
   &\le
   C\|u_1-u_2\|_{L^\infty(B_r^+(x_o))}
   \int_{B_r^+(x_o)}|F_1|+|F_2|+|Du_1|^2+|Du_2|^2\dx,
\end{align*}
and this establishes the claimed comparison estimate.
\end{proof}

We use the preceding comparison estimate in the following theorem for
the derivation of Calder\'on-Zygmund type estimates for the
gradient. For the proof, we use techniques going back to
Caffarelli and Peral \cite{CaffarelliPeral:1998}. Actually,
our proof is inspired by arguments of Acerbi and
Mingione \cite{Acerbi-Mingione:2007,Mingione:2007}. In the following
theorem, we are dealing both with the boundary case $x_o\in\partial B$
and the interior case $x_o\in B$.

\begin{theorem}\label{thm:Calderon-Zygmund}
Assume that $u\in\Ss$ satisfies~(\ref{local-plateau}) on
a half-disk $B_R^+(x_o)$ with $x_o\in \overline B$, $R\in(0,1)$ and
$B_R^+(x_o)\cap\{P_1,P_2,P_3\}=\varnothing$, under the
assumption~(\ref{F-growth}). We assume that $\mathbf{D}(u)\le E_o$ for some constant $E_o\ge1$. 
Then there is a constant $\eps_3=\eps_3(C_1,\Gamma,E_o)>0$ such that the smallness condition
\begin{equation*}
    \osc_{B_R^+(x_o)}u\le \eps_3
\end{equation*}
implies $u\in W^{1,4}(B_{R/4}^+(x_o),\R^3)$, with the
corresponding quantitative estimate 
\begin{equation*}
    \int_{B_{R/4}^+(x_o)}|Du|^4\dx
    \le  
    \frac C{R^2} 
    \bigg(\int_{B_{R/2}^+(x_o)}|Du|^2\dx\bigg)^2
    +
    C \int_{B_{R/2}^+	(x_o)}|f|^2\dx
\end{equation*}
for  a universal constant $C=C(C_1,\Gamma,E_o)$. 
\end{theorem}

\begin{proof}
We shall later fix the constant $\eps_3>0$ such that 
$\eps_3\le\min\{\eps_2,\rho_o\}$ with the constant
$\eps_2$ from Lemma~\ref{lem:W22-comparison} and the radius
$\rho_o$ from Lemma~\ref{lem:extension}. In particular,
this implies 
\begin{equation}\label{small-image}
    u(B_R^+(x_o))\subset B_{\eps_3}^3(p_0)\subset B_{\rho_o}^3(p_0)
\end{equation}
for $p_o=u(x_o)\in\R^3$.

\textbf{Step 1: Covering of super-level sets.}
For every $r\in(0,\frac R2)$ and $\lambda>0$ we define super-level sets
\begin{equation*}
    E(r,\lambda)
    :=\left\{x\in B_r^+(x_o):\mbox{$x$ is a Lebesgue point
      of $Du$ with }|Du(x)|>\lambda\right\}.
\end{equation*}
We let
\begin{equation}\label{lambda0-choice}
    \lambda_o^2:= \mint_{B_{R/2}^+(x_o)}|Du|^2+ |f|\dx
\end{equation}
and fix two radii $s,t\in [\frac R4,\frac R2]$ with $s<t$. Then
we define
\begin{equation}\label{lambda1-choice}
    \lambda_1:=\frac{35 R}{t-s}\,\lambda_o\,.
\end{equation}
For a fixed $\lambda\ge\lambda_1$, we consider a point $x_1\in
E(s,\lambda)$. For any radius $r$ with $\tfrac1{70}(t-s)<r<t-s$, we can estimate by the definition of $\lambda_o$
\begin{equation*}
    \mint_{B_r^+(x_1)}|Du|^2+|f|\dx
    \le 
    \Big(\frac R{2r}\Big)^2 \lambda_o^2
    < \Big(\frac{35R}{t-s}\Big)^2 \lambda_o^2
    =\lambda_1^2\le\lambda^2.
\end{equation*}
On the other hand, since $x_1\in E(s,\lambda)$, the definition of
$E(s,\lambda)$ implies
\begin{equation*}
    \lim_{r\downarrow 0}\,\mint_{B_r^+(x_1)}|Du|^2+|f|\dx
    \ge |Du(x_1)|^2>\lambda^2.
\end{equation*}
The preceding two estimates and the absolute continuity of
the integral enable us to define $r_1\in
(0,\tfrac1{70}(t-s))$ (depending on $x_1$)
as the maximal radius with the property
\begin{equation}\label{Int=lambda}
    \mint_{B_{r_1}^+(x_1)}|Du|^2+|f|\dx=\lambda^2.
\end{equation}
The maximality of the radius implies in particular
\begin{equation}\label{Int<lambda}
    \mint_{B_r^+(x_1)}|Du|^2+|f|\dx
    < 
    \lambda^2\qquad\mbox{for all $r_1<r\le 70r_1$.}
\end{equation}
Proceeding in this way with every $x\in E(s,\lambda)$, we obtain a family of disks covering $E(s,\lambda)$, each of which
satisfies~(\ref{Int=lambda}) and~(\ref{Int<lambda}). By Vitali's
covering theorem, we may extract countably many, pairwise disjoint
disks $B_{k}^+:=B_{r_k}^+(x_k)$ with centers $x_k\in E(s,\lambda)$
and $0<r_k<\tfrac1{70}(t-s)$ for $k\in\N$, and with
\begin{equation*}
    E(s,\lambda)\subset \bigcup_{k\in\N} 5B_k^+\subset B_t^+(x_o). 
\end{equation*}
Here and in what follows, we
  use the notation $\sigma B_r^+(x)=B_{\sigma r}^+(x)$ for any
  $\sigma>0$. By the choice  of $B_k^+$, the
  formulae~(\ref{Int=lambda}) and~(\ref{Int<lambda}) are valid for
  each of the $B_k^+$, which means in particular that for each
  $k\in\N$ there holds
  \begin{equation}
    \label{Int_k=lambda}
    \mint_{B_k^+}|Du|^2+|f|\dx=\lambda^2
  \end{equation}
  and at the same time,
\begin{equation}\label{Int_k<lambda}
    \mint_{\sigma B_k^+}|Du|^2+|f|\dx<
    \lambda^2\qquad\mbox{for }\sigma\in\{5,10,70\}.
  \end{equation}

\textbf{Step 2: Comparison estimates.}
For each $k\in\N$, we distinguish whether we are in the \textbf{interior situation}
$10B_k\Subset B$ or in the \textbf{boundary situation}
$10\overline B_k\cap\partial B\neq\varnothing$. We first consider the
interior situation, in which $10B_k^+=10 B_k$. In this case, 
we choose the comparison map as the harmonic function with $w_k\in
u+W^{1,2}_0(10 B_k,\R^3)$. Since $w_k$ is harmonic and its boundary
values are contained in $B_{\eps_3}(p_o)$ by~(\ref{small-image}), the maximum principle implies 
$w_k(10 B_k)\subset B_{\eps_3}(p_o)$. Testing the equations $\Delta u=F$ and $\Delta w_k=0$ on
$10B_k$ with $w_k-u\in W^{1,2}_0(10 B_k,\R^3)$, we therefore infer the
\textbf{comparison estimate}
\begin{align}\label{Comp-interior}
  \mint_{10 B_k}&|Dw_k-Du|^2\dx\\
  &
  \le 
  2{\eps_3}\mint_{10 B_k}|F|\dx
  \le 
  2{\eps_3}\,
  C_1\,\mint_{10 B_k}|Du|^2+|f|\dx
  \le 
  C(C_1){\eps_3}\lambda^2,\nonumber
\end{align}
where we used assumption~(\ref{F-growth}) and~(\ref{Int_k<lambda}) for
the two last estimates. 
Furthermore, since $w_k$ is harmonic and therefore energy minimizing,
we have for every $q\in[1,\infty)$
\begin{align}\label{Hi-Int-interior}
  \bigg(\mint_{5B_k}|Dw_k|^{q}\dx\bigg)^{\frac 2q}
  &
  \le
  \sup_{5B_k} |Dw_k|^2
  \le
  C\mint_{10 B_k}|Dw_k|^2\dx \\
  &
  \le 
  C\mint_{10 B_k}|Du|^2\dx
  \le 
  C\lambda^2,\nonumber
\end{align}
where we used~(\ref{Int_k<lambda}) in the last step. 
Next, we turn our attention to the boundary case, in which there
exists a point $y_k\in 10 \overline B_k\cap\partial B$. Writing
$\widetilde B_k^+:=B_{r_k}^+(y_k)$, we have
\begin{equation*}
  5B_k^+\subset 15\Tilde B_k^+\subset 60\Tilde B_k^+\subset 70 B_k^+.
\end{equation*}
As comparison map on $60\Tilde B_k^+$ we choose a minimizer $w_k\in
W^{1,2}(B,\R^3)$ of the Dirichlet energy in the class
\begin{equation*}
   \big\{w\in\Ss\,:\, w=u\mbox{ on }B\setminus 60\Tilde B_k^+ \big\}.
\end{equation*}
This minimizer $w_k$ exists by Lemma~\ref{lem:exist-comparison} and
by Lemma~\ref{lem:harmonic_plateau}, it satisfies the differential
inequality~(\ref{local-plateau-II}) on $60\Tilde B_k^+$ with $F=0$.
Moreover, its image is contained in the ball $B_{\eps_3}(p_o)$ by the
convex hull property of the Dirichlet energy. We thus infer from
the Comparison Lemma~\ref{lem-Comp-Est} that 
\begin{align}\label{Comp-bdry}
  \mint_{5 B_k^+}|Dw_k-Du|^2\dx
  &
  \le C\mint_{60 \Tilde B_k^+}|Dw_k-Du|^2\dx\\\nonumber
  &
  \le C(\Gamma){\eps_3}\,
  \mint_{60 \Tilde B_k^+}|F|+|Du|^2+|Dw_k|^2\dx\\\nonumber
  &
  \le C(C_1,\Gamma){\eps_3}\,\mint_{60 \Tilde B_k^+}|f|+|Du|^2\dx
  \le C{\eps_3}\lambda^2,  
\end{align}
where in the last line, we used first the minimizing property of
$w_k$ and then the bound~(\ref{Int_k<lambda}) with $\sigma=70$
together with the inclusion $60 \Tilde B_k^+\subset 70
B_k^+$. Moreover, from Lemma~\ref{lem:W22-comparison}, applied on
$60\Tilde B_k^+$, we infer the
following bound for every $q\in[1,\infty)$.
\begin{align}\label{Hi-Int-bdry}
  \bigg(\mint_{5B_k^+}&|Dw_k|^q\dx\bigg)^{\frac 2q}
  \le
  C_q\bigg(\mint_{15\Tilde B_k^+}|Dw_k|^q\dx\bigg)^{\frac 2q}\\\nonumber
  &
  \le C_q\,\mint_{15\Tilde B_k^+}r_k^2|D^2w_k|^2+|Dw_k|^2\dx\\\nonumber
  &
  \le C_q(E_o,\Gamma)\mint_{30 \Tilde B_k^+}|Dw_k|^2\dx
  +
  C_q\mint_{15\Tilde B_k^+}|Dw_k|^2\dx\\\nonumber
  &
  \le C_q(E_o,\Gamma)\,\mint_{60\Tilde B_k^+}|Du|^2\dx
  \le C_q(E_o,\Gamma)\lambda^2,
\end{align}
where the last bound is a consequence of~(\ref{Int_k<lambda}) with
$\sigma=70$, since $60 \Tilde B_k^+\subset 70 B_k^+$.

\textbf{Step 3: Energy estimates on super-level sets.}
The property~(\ref{Int_k=lambda}) of the sets $B_k^+$ implies
\begin{equation}
  \label{ball-size-I}
  \big|B_k^+\big|= \frac1{\lambda^2}\int_{B_k^+}|Du|^2\dx+\frac1{\lambda^2}\int_{B_k^+}|f|\dx.
\end{equation}

In the first integral on the right-hand side, we decompose
the domain of integration into $B_k^+\cap\{|Du|>\lambda/2\}$ and
$B_k^+\cap\{|Du|\le\lambda/2\}$, with the result
\begin{equation*}
  \frac1{\lambda^2}\int_{B_k^+}|Du|^2\dx
  \le \frac1{\lambda^2}\int_{B_k^+\cap\{|Du|>\lambda/2\}}|Du|^2\dx
      +\tfrac14\big|B_k^+\big|.
\end{equation*}
Similarly, by distinguishing the cases $|f|>\lambda^2/4$ and
$|f|\le\lambda^2/4$, we deduce
\begin{equation*}
  \frac1{\lambda^2}\int_{B_k^+}|f|\dx
  \le \frac1{\lambda^2}\int_{B_k^+\cap\{|f|>\lambda^2/4\}}|f|\dx
      +\tfrac14\big|B_k^+\big|.
\end{equation*}
Plugging the preceding two estimates into~(\ref{ball-size-I}) and
re-absorbing the resulting term $\tfrac12|B_k^+|$ into the left-hand
side, we arrive at
\begin{equation}
  \label{ball-size-II}
  \big|B_k^+\big|
   \le\frac2{\lambda^2}\int_{B_k^+\cap\{|Du|>\lambda/2\}}|Du|^2\dx
      +\frac2{\lambda^2}\int_{B_k^+\cap\{|f|>\lambda^2/4\}}|f|\dx
\end{equation}
for every $k\in\N$. Since the sets
$5B_k^+$ cover the super-level set $E(s,\lambda)\supset E(s,L\lambda)$
for every parameter $L\ge1$, there holds
\begin{equation*}
  \int_{E(s,L\lambda)}|Du|^2\dx
  \le \sum_{k\in\N}\,\int_{5B_k^+\cap E(s,L\lambda)}|Du|^2\dx.
\end{equation*}
Each of the terms in the above sum can be estimated as follows:
\begin{align*}
  &\int_{5B_k^+\cap E(s,L\lambda)}|Du|^2\dx\\
  &\phantom{M}
  \le 
  2\int_{5B_k^+}|Du-Dw_k|^2\dx
  +
  \frac 2{(L\lambda)^{4/3}}\int_{5B_k^+\cap
   E(s,L\lambda)}|Dw_k|^2|Du|^{4/3}\dx\\
  &\phantom{M}
  \le 
  2\int_{5B_k^+}|Du-Dw_k|^2\dx
  +\frac C{(L\lambda)^4}
  \int_{5B_k^+}|Dw_k|^6\dx+\tfrac12\int_{5B_k^+\cap
   E(s,L\lambda)}|Du|^2\dx,
\end{align*}
where we used Young's inequality in the last step. Here, we re-absorb the last
integral into the left-hand side and estimate the other two integrals
in the preceding line by~(\ref{Comp-interior})
and~(\ref{Hi-Int-interior}) if we are in the interior situation,
respectively by~(\ref{Comp-bdry}) and~(\ref{Hi-Int-bdry}) in the
boundary situation. This leads us to
\begin{align*}
  \int_{5B_k^+\cap E(s,L\lambda)}|Du|^2\dx
  \le C\big({\eps_3}+L^{-4}\big)\lambda^2\big|B_k^+\big|,
\end{align*}
with $C=C(C_1,\Gamma ,E_o)$.
Summing over $k\in\N$ and then applying (\ref{ball-size-II}), we arrive at
\begin{align}\label{energy-level-set}
  \int_{E(s,L\lambda)}&|Du|^2\dx\nonumber\\\nonumber
  &\le
  C\big({\eps_3}{+}L^{-4}\big)\lambda^2\sum_{k\in\N}\big|B_k^+\big|\\
  &\nonumber\le C\big({\eps_3}{+}L^{-4}\big)
     \sum_{k\in\N}\bigg[\int_{B_k^+\cap\{|Du|>\lambda/2\}}|Du|^2\dx
      +\int_{B_k^+\cap\{|f|>\lambda^2/4\}}|f|\dx\bigg]\\
  &\le C\big({\eps_3}{+}L^{-4}\big)
     \bigg[\int_{B_t^+(x_o)\cap\{|Du|>\lambda/2\}}|Du|^2\dx
      +\int_{B_t^+(x_o)\cap\{|f|>\lambda^2/4\}}|f|\dx\bigg].
\end{align}
In the last step we used the fact that the sets $B_k^+$ are pairwise
disjoint and contained in $B_t^+(x_o)$. We recall that this estimate holds true for all $\lambda\ge\lambda_1$. 

\textbf{Step 4: The final estimate.}
We define truncations
\begin{equation*}
  |Du|_\ell:=\min\{|Du|,\ell\}\quad\mbox{for every $\ell\in\N$.}
\end{equation*}
Fubini's theorem yields for every $\ell\in\N$ that there holds:
\begin{align*}
  \int_{B_s^+(x_o)}|Du|_\ell^2\,|Du|^2\dx 
  &=2\int_{B_s^+(x_o)}\int_0^{|Du|_\ell}\lambda\,d\lambda\, |Du|^2\dx\\
  &=
  2\int_0^\ell \lambda \int_{B_s^+(x_o)\cap \{|Du|_\ell>\lambda\}}|Du|^2\dx\,d\lambda.
\end{align*}
Clearly, for $\lambda\le\ell$, the condition $|Du|_\ell>\lambda$ is
equivalent to $|Du|>\lambda$. We use this to calculate by a change of variables
\begin{align*}
  &\int_{B_s^+(x_o)}|Du|_\ell^2\,|Du|^2\dx\\
  &\qquad=2L^2\int_0^{\ell/L}\lambda
  \int_{E(s,L\lambda)}|Du|^2\dx\,d\lambda\\
  &\qquad\le 2L^2\int_{\lambda_1}^{\ell/L}\lambda
  \int_{E(s,L\lambda)}|Du|^2\dx\,d\lambda
  +L^2\lambda_1^2\int_{B_{R/2}^+(x_o)}|Du|^2\dx=:II+III.
\end{align*}
It remains to estimate the term $II$. 
For this aim we recall the estimate~(\ref{energy-level-set}), which holds for any
$\lambda\ge\lambda_1$. This leads us to
\begin{align*}
  II
  &\le 
  C\big({\eps_3} L^2+L^{-2}\big)\int_{\lambda_1}^{2\ell}
  \lambda\int_{B_t^+(x_o)\cap\{|Du|>\lambda/2\}}|Du|^2\dx\,d\lambda\\
  &\phantom{\le}
  +C\big({\eps_3} L^2+L^{-2}\big)\int_{\lambda_1}^\infty
  \lambda\int_{B_t^+(x_o)\cap\{|f|>\lambda^2/4\}}|f|\dx\,d\lambda\\
  &
  =:C\big({\eps_3} L^2+L^{-2}\big)\big(II_1+II_2\big),
\end{align*}
with the obvious labeling of $II_1$ and $II_2$.
For the estimation of the first term, we calculate by a change of
variables and Fubini's theorem
\begin{align*}
  II_1
  &\le C\int_0^{\ell}\lambda\int_{B_t^+(x_o)\cap\{|Du|_\ell>\lambda\}}|Du|^2\dx\,d\lambda\\
  &=C\int_{B_t^+(x_o)}\int_0^{|Du|_\ell}\lambda \,d\lambda |Du|^2\dx
  = C\int_{B_t^+(x_o)}|Du|_\ell^2\,|Du|^2\dx.
\end{align*}
Similarly, now by the change of variables 
$\mu=\lambda^2/4$ we estimate:
\begin{align*}
  II_2
  &\le C\int_0^\infty\int_{B_t^+(x_o)\cap\{|f|>\mu\}}|f|\dx\,d\mu\\
  &=C\int_{B_t^+(x_o)}\int_0^{|f|} \,d\mu\, |f|\dx
  = C\int_{B_t^+(x_o)}|f|^2\dx.
\end{align*}
Collecting the estimates, we arrive at
\begin{align*}
  \int_{B_s^+(x_o)}&|Du|_\ell^2\,|Du|^2\dx\\
  &\le 
  C\big({\eps_3}
  L^2+L^{-2}\big)\bigg(\int_{B_t^+(x_o)}|Du|_\ell^2\,|Du|^2\dx
  +\int_{B_t^+(x_o)}|f|^2\dx\bigg)\\
  &\qquad\qquad+L^2\lambda_1^2\int_{B_{R/2}^+(x_o)}|Du|^2\dx,
\end{align*}
where here, $C=C(C_1,\Gamma,E_o)$.
Now we choose first the parameter $L\ge1$ so large that
$CL^{-2}\le\frac14$ and then ${\eps_3}\in(0,1)$ so small that
$C{\eps_3} L^2\le\frac14$. This fixes the parameters $L$ and $\eps_3$ in
dependence on $C_1$, $\Gamma$ and $E_o$. Using the
above choice of parameters and the choice of $\lambda_1$ in~(\ref{lambda1-choice}), the preceding inequality becomes
\begin{align*}
  \int_{B_s^+(x_o)}|Du|_\ell^2\,|Du|^2\dx
  &\le \tfrac12 \int_{B_t^+(x_o)}|Du|_\ell^2\,|Du|^2\dx
  +\tfrac12\int_{B_{R/2}^+(x_o)}|f|^2\dx\\
  &\qquad+C\frac{R^2}{(t-s)^2}\,\lambda_o^2\int_{B_{R/2}^+(x_o)}|Du|^2\dx,
\end{align*}
whenever $\frac R4\le s<t\le\frac R2$. Therefore, the Iteration
Lemma~\ref{lem:Giaq} is applicable and yields
\begin{equation*}
  \int_{B_{R/4}^+(x_o)}|Du|_\ell^2\,|Du|^2\dx
  \le C\int_{B_{R/2}^+(x_o)}|f|^2\dx
  +C\lambda_o^2\int_{B_{R/2}^+(x_o)}|Du|^2\dx
\end{equation*}
for each $\ell\in\N$. Letting $\ell\to\infty$, we deduce by Fatou's
lemma, keeping in mind the definition of $\lambda_o$ in~(\ref{lambda0-choice}),
\begin{align*}
  \int_{B_{R/4}^+(x_o)}&|Du|^4\dx\\
  &\le 
  C\int_{B_{R/2}^+(x_o)}|f|^2\dx
  +
  C\mint_{B_{R/2}^+(x_o)}|Du|^2+|f|\dx
   \int_{B_{R/2}^+(x_o)}|Du|^2\dx.
\end{align*}
This implies the claim by Young's and Jensen's inequalities
with a constant $C$ having the dependencies indicated in the
formulation of the lemma. 
\end{proof}

\section{Uniform $W^{2,2}$-estimates}

We begin with the interior $W^{2,2}$-estimates, which will be crucial
for the boundary estimates since they will imply continuity of the
solutions up to the boundary. 
A similar result was proven in \cite[Lemma 7.3]{BoegDuzSchev:2011} 
for right-hand sides with $f\in W^{1,2}(B_R(x_o),\R^3)$. Here, we
weaken this assumption to $f\in L^2(B_R(x_o),\R^3)$. 

\begin{lemma}\label{lem-interior-W22}
On $B_R(x_o)\Subset B$, consider a weak solution
$u\in W^{1,2}\cap C^0(B_R(x_o),\R^3)$ of $\Delta u=F$, 
where $F$ satisfies (\ref{F-growth}) for some constant $C_1>0$
and $f\in L^2(B_R(x_o),\R^3)$. There
exists $\eps_o=\eps_o(C_1)>0$ and $C=C(C_1)$ such that
the smallness condition
\begin{equation}\label{small-energy}
    \int_{B_R(x_o)}|Du|^2\dx\le\eps_o^2
\end{equation}
implies $u\in W^{2,2}(B_{R/2}(x_o),\R^3)$ with the quantitative estimate
\begin{equation}\label{W22-interior}
    \int_{B_{R/2}(x_o)}|D^2u|^2\dx
    \le
    \frac C{R^2}\int_{B_R(x_o)}|Du|^2\dx
    +C\int_{B_R(x_o)}|f|^2\dx.
\end{equation}
\end{lemma}

\begin{proof}
We choose the constant $\eps_o>0$ as in
Theorem \ref{thm:apriori-W22-int}. In view of this theorem, it only remains to establish $u\in W^{1,4}_{\rm loc}(B_R(x_o),\R^3)$.
To this end, for any
$y\in B_R(x_o)$, we first exploit the continuity of $u$ in order to choose a radius $\rho>0$ small enough to have that
$\osc_{B_\rho(y)}u\le\eps_3$ for the constant $\eps_3$ determined in
Lemma~\ref{thm:Calderon-Zygmund}. From this lemma, we then infer
$u\in W^{1,4}(B_{\rho/4}(y),\R^3)$. Since the point $y\in B_R(x_o)$ was arbitrary, this implies $u\in W^{1,4}_{\rm loc}(B_R(x_o),\R^3)$. Therefore, we may apply the
  a-priori estimates from Theorem~\ref{thm:apriori-W22-int} for any radius
  $\widetilde R<R$ and let $\widetilde R\uparrow R$ in order to arrive
  at the claimed estimate. 
\end{proof}

The first important implication of the preceding lemma is the
following result that will guarantee small oscillation of the
solutions, which we assumed in the preceding sections. A similar
result has been used in \cite[Lemma 3.1]{Scheven:2006} in the context
of a free boundary condition. We point out that similar arguments
yield continuity of $u$ up to the boundary if the boundary values are continuous,
cf. Hildebrandt \& Kaul \cite{HildebrandtKaul:1972}, but the modulus
of continuity would depend on the absolute continuity of the Dirichlet
energy and would therefore not be suitable for our purposes. 

\begin{lemma}\label{lem:small-osc}
Assume that $u\in W^{1,2}(B_R^+(x_o),\R^3)
\cap C^0(B_R^+(x_o),\R^3)$ weakly solves $\Delta u=F$
on $B_R^+(x_o)$, where $x_o\in\partial B$ and
$R\in(0,1)$, and suppose that 
$F$ satisfies (\ref{F-growth}) and
(\ref{time-deriv-bound}) for some constants $C_1,K>0$.
Moreover, we assume that $u$ maps $I_{R/2}(x_o)$ into a subset 
$G\subset\R^3$. Then there exists $\eps_o=\eps_o(C_1)>0$
and $C=C(C_1)$ such that the smallness condition
\begin{equation*}
    \eps^2:=\int_{B_R^+(x_o)}|Du|^2\dx\le\eps_o^2
\end{equation*}
implies
\begin{equation}\label{small-dist}
    \dist(u(y),G)
    \le
    C(\eps+RK)\quad\mbox{for all }y\in B_{R/2}^+(x_o).
\end{equation}
If additionally, the boundary trace $u|_{I_R(x_o)}$ is continuous
with modulus of continuity $\omega\colon[0,\infty)\to[0,\infty)$,
then there holds
\begin{equation*}
    \osc_{B_{R/2}^+(x_o)}\,u\le C(\eps+RK)+2\omega(R).
\end{equation*}

\end{lemma}
\begin{proof}
We fix an arbitrary $y\in B_{R/2}^+(x_o)$ and let
$r:=\frac14\dist(y,\partial B)<\frac R4$. We choose
the constant $\eps_o=\eps_o(C_1)>0$ as in Lemma  
\ref{lem-interior-W22}, which enables us to apply this lemma on 
$B_{2r}(y)\Subset B_R^+(x_o)$. We infer $u\in W^{2,2}(B_r(y),\R^3)\hookrightarrow C^{0,\alpha}(B_r(y),\R^3)$ for an arbitrary $\alpha\in(0,1)$, with the corresponding estimate
\begin{align*}
    r^{2\alpha}[u]_{C^{0,\alpha}(B_r(y))}^2
    &
    \le
    C(\alpha)\int_{B_r(y)}r^2|D^2u|^2+|Du|^2\dx\\
    &
    \le
    C(\alpha,C_1)\bigg[\int_{B_{2r}(y)}|Du|^2\dx
    +r^2\int_{B_{2r}(y)}|f|^2\dx\bigg]\\
    &
    \le
    C(\alpha,C_1)\big(\eps^2+R^2K^2\big).
\end{align*}
In particular, we know that for every $x\in B_r(y)$, there holds
\begin{equation}\label{Holder-Est}
    |u(x)-u(y)|
    \le
    Cr^{-\alpha}\big(\eps+RK\big)|x-y|^\alpha\le C\big(\eps+RK\big).
\end{equation}
Here, we may eliminate the dependence of the constant on $\alpha$ by
fixing $\alpha=\frac12$. 
Moreover, since $s\mapsto u(sx)$ is absolutely continuous for
a.e. $x\in B_R^+(x_o)$ and $u(\frac x{|x|})\in G$, we conclude
\begin{align*}
    \mint_{B_r(y)\cap\partial B_{|y|}}\dist(u(x),G)\,d\H^1
    &
    \le 
    \mint_{B_r(y)\cap \partial B_{|y|}}\big|u(x)-u\big(\tfrac
    x{|x|}\big)\big|\,d\H^1x\\
    &
    \le
    \frac Cr 
    \int_{B_r(y)\cap \partial B_{|y|}}\int_{|y|}^1
    \big|\tfrac{\partial u}{\partial r}(\rho
    \tfrac x{|x|})\big|\,d\rho\,d\H^1x\\
    &
    \le 
    C\bigg(\int_{B_R^+(x_o)}|Du|^2\dx\bigg)^{1/2}\le C\eps,
\end{align*}
where we used $1-|y|=\dist(y,\partial B)= 4r$. Combining this
with~(\ref{Holder-Est}), we arrive at
\begin{align*}
    \dist(u(y),G)
    \le
    \mint_{B_r(y)\cap\partial B_{|y|}(0)}\big[\dist(u(x),G)
    +|u(x)-u(y)|\big]\dx
    \le 
    C\big(\eps+RK\big),
\end{align*}
which is the first assertion (\ref{small-dist}).
If we assume moreover that $u|_{I_R(x_o)}$ is continuous
with modulus of continuity $\omega$, then $u(I_{R/2}(x_o))$ is contained in $G\cap B_{\omega(R)}(p)$ for $p=u(x_o)$.
Consequently, we infer the estimate (\ref{small-dist}) with
$G\cap B_{\omega(R)}(p)$ instead of $G$,
which implies that $u(B_{R/2}^+(x_o))$ is contained in a ball of
radius $C(\eps+RK)+\omega(R)$. This yields the second assertion of the
lemma.
\end{proof}

Now we are in a position to extend the $W^{2,2}$-estimates of
Lemma~\ref{lem-interior-W22} up to the boundary.

\begin{theorem}\label{boundary-W22}
Consider a solution $u\in\Ss\cap C^0(B,\R^3)$
of~(\ref{local-plateau}) on a half-disk centered in $x_o\in\partial
B$ with $R\in(0,\frac12)$ and $B_R^+(x_o)\cap\{P_1,P_2,P_3\}=\varnothing$. We suppose that the
assumptions~(\ref{F-growth}) and~(\ref{time-deriv-bound}) are in
force and that $u|_{I_R(x_o)}$ is continuous with modulus of
continuity $\omega:[0,\infty)\to[0,\infty)$. Then there is a
constant $\eps_4=\eps_4(C_1,\Gamma,\omega(\cdot))\in(0,1)$ and a radius $R_o=R_o(C_1,\Gamma,K,\omega(\cdot))\in(0,1)$ such that the
smallness conditions
\begin{equation}
    \label{smallness-conditions}
    \int_{B_R^+(x_o)}|Du|^2\dx\le\eps_4^2\quad\mbox{and}\quad
    R\le R_o
\end{equation}
imply $u\in W^{2,2}(B_{R/4}^+(x_o),\R^3)$. Moreover, with a universal
constant $C=C(C_1,\Gamma)$ we have the quantitative estimate
  \begin{equation}\label{W22-final}
    \int_{B_{R/4}^+(x_o)}|D^2u|^2\dx\le
    \frac C{R^2}\int_{B_R^+(x_o)}|Du|^2\dx
    +C\int_{B_R^+(x_o)}|f|^2\dx.
  \end{equation}
\end{theorem}

\begin{proof}
We will later fix $\eps_4\in(0,1)$ so small that
$\eps_4\le\min\{\eps_o,\eps_1\}$ with the constants $\eps_o$ and $\eps_1$ determined in Lemma~\ref{lem:small-osc}, respectively in
Theorem~\ref{thm:apriori-W22-bdry}. Lemma~\ref{lem:small-osc} then implies
\begin{equation*}
  \osc_{B_{R/2}^+(x_o)}u\le C(\eps_4+RK)+2\omega(R_o).
\end{equation*}
Therefore we can achieve -- by choosing $0<\eps_4<\min\{\eps_o,\eps_1\}$  and $R_o\in(0,1)$ sufficiently small -- that
\begin{equation}\label{little-osc}
    \osc_{B_{R/2}^+(x_o)}u\le \min\{\eps_1,\eps_3\},
\end{equation}
where $\eps_3$ denotes the constant from Theorem~\ref{thm:Calderon-Zygmund} for the choice $E_o=1$. We note that the choice of
$\eps_4$ can be performed  in dependence on $C_1,\Gamma$ and
$\omega(\cdot)$, while $R_o$ may depend additionally on $K$. The
small oscillation property~(\ref{little-osc}) together with the fact
$\int_{B_R^+(x_o)}|Du|^2\dx\le1$ enables us to apply the
Calder\'on-Zygmund Theorem~\ref{thm:Calderon-Zygmund}, from which we
infer $u\in W^{1,4}_{\rm loc}(B_{R/2}^+(x_o),\R^3)$. Therefore and because of the smallness properties~(\ref{smallness-conditions})
and~(\ref{little-osc}), we may apply Theorem~\ref{thm:apriori-W22-bdry}\,(ii), which yields the desired estimate~(\ref{W22-final}).
This concludes the proof of the theorem. 
\end{proof}

\section{Concentration compactness principle}
In this section we consider sequences of maps $u_k\in \SGA$ satisfying the Euler-Lagrange system, the  weak Neumann type boundary condition
and the stationarity condition. To be precise, for 
$f\in L^2(B,\R^3)$ we consider solutions $u\in\SGA$ of the system 
\begin{align}\label{Euler-Lagrange}
	\int_B Du\cdot D\varphi + 2(H\circ u)
	D_1u\times D_2u\cdot \varphi-f\cdot \varphi\, dx=0
\end{align}
 for any $\varphi\in C^\infty_0(B,\R^3)$.
We note that by Rivi\`ere's result in the form of Theorem 
\ref{mod-Riviere}, such maps are of class $C^{0,\alpha}(B,\R^3)$
for some $\alpha\in (0,1)$, and by the result of Hildebrandt and Kaul
from \cite[Lemma 3]{HildebrandtKaul:1972} also continuous up to the boundary of $B$, i.e. $u\in C^0(\overline B,\R^3)$.
Further, we say that $u\in \SGA$ satisfies the Neumann type  boundary condition associated to the Plateau boundary condition
in the weak sense, if for any $w\in T_u\mathcal S^\ast$
there holds
\begin{equation}\label{Plateau-weak}
	0
	\le 
	\int_{B} 
	\big[Du\cdot D w+\Delta u\cdot  w\big]\,dx.
\end{equation}
Finally, we call $u\in \SGA$ stationary (with respect to inner variations), if for any vector field $\eta\in\mathcal C^\ast (B)$
there holds:
\begin{align}\label{stationary}
    \int_B{\rm Re} \big( \frak h[u] \,\overline\partial\eta\big) \, dx
     -
     \int_B f\cdot Du\,\eta\, dx=0.
\end{align}
We note that \eqref{Euler-Lagrange}, \eqref{Plateau-weak} and
\eqref{stationary} are satisfied for minimizers of the functionals
$\mathbf F$ defined in~(\ref{discrete-H-energy}) with
$f=\frac1h(u-z)\in L^2(B,\R^3)$, see Lemma~\ref{lem:EulerMeasure}.

\begin{lemma}\label{convergence-non-linearity}
Assume that $u_k\in \SGA\cap C^0(\overline B, \R^3)$ and $f_k\in L^2(B,\R^3)$
for $k\in\N$.
Moreover, suppose that
$u_k$ fulfills the Euler-Lagrange system
\eqref{Euler-Lagrange}, the weak Neumann condition
\eqref{Plateau-weak} and the stationarity condition \eqref{stationary}
with $(u_k, f_k)$ instead of $(u,f)$.
Finally, suppose that
$u_k\to u$ strongly in $L^2(B,\R^3)$ and
\begin{equation}\label{bound:uk-zk}
    \sup_{k\in\N}\int_B|Du_k|^2+|f_k| ^2 \, dx <\infty.
\end{equation}
Then the following holds:

\begin{enumerate}
\item If furthermore
  \begin{equation}\label{assump-uk-zk}
    f_k\wto f\quad\mbox{weakly in $L^2(B,\R^3)$,}
  \end{equation}
  then the limit map $u\in \SGA$ solves the
  Euler-Lagrange system~(\ref{Euler-Lagrange}), fulfills
  the boundary
  condition~(\ref{Plateau-weak}) and 
  the stationarity condition~(\ref{stationary}).

\item The non-linear $H$-term converges in the sense of distributions 
  (even without the assumption
  \eqref{assump-uk-zk}), that is for every $\varphi\in
  C^\infty_0(B,\R^3)$ we have
  \begin{align}\label{convergence-H-term}
    \int_B (H\circ u) D_1u\times D_2u\cdot\varphi\, dx &=
    \lim_{k\to\infty}\int_B (H\circ u_k) D_1u_{k}\times
    D_2u_{k}\cdot\varphi\, dx.
  \end{align}
\end{enumerate}

\end{lemma}

\begin{proof} We first prove the claim (i) and therefore assume that
\eqref{assump-uk-zk} is valid. We start with the observation
that by \eqref{bound:uk-zk}, the maps $u_k\in \SGA$ admit 
$\sup_{k\in\N} \mathbf D(u_k)<\infty$. Moreover, by the definition of the class
$\SGA$ they also satisfy the three point condition,
that is $u_k(P_j)=Q_j$ for $j=1,2,3$.
As is well known from the theory of parametric minimal surfaces,
it then follows from the  Courant-Lebesgue Lemma and the Jordan curve property of $\Gamma$ that the sequence
of boundary traces $u_k\big|_{\partial B}$ is equicontinuous
(cf. Lemma~\ref{lem:compact traces}) and therefore all the maps $u_k$
admit the same modulus of continuity $\omega$ 
on $\partial B$. Therefore, we may assume that
$u\in \SGA$ and $u_k\to u$ uniformly on $\partial B$. 

We define a sequence of Radon measures $\mu_k$ on $\R^2$ by
\begin{equation*}
    \mu_k:=\mathcal L^2\edge |Du_k|^2\, .
\end{equation*}
Since $(u_k)_{k\in\N}$ is a bounded sequence in $W^{1,2}(B,\R^3)$ 
by \eqref{bound:uk-zk}, we have
$$
	\sup_{k\in\N} \mu_k(\R^2)
	=
	2\sup_{k\in\N} \mathbf D(u_k)
	<\infty.
$$
Therefore, passing to a not relabeled
subsequence we can assume that $\mu_k\wto\mu$ in the sense of Radon
measures, for a Radon measure $\mu$ on $\R^2$ with 
$\mu (\R^2)<\infty$. We note that $\mu\edge (\R^2\setminus \overline B)=0$ by construction. Next we define 
the {\bf singular set} $\Sigma$ of $\mu$ by
\begin{equation*}
    \Sigma :=\left\{ x_o\in \overline
    B: \mu (\{ x_o\})\ge \varepsilon\right\}\cup\{P_1,P_2,P_3\},
\end{equation*}
where $\epsilon:=\min\{\epsilon_o,\epsilon_4\}>0$ for the constants 
$\epsilon_o$ and $\epsilon_4$ from
Lemma \ref{lem-interior-W22}, respectively Theorem~\ref{boundary-W22}.
We mention that ${\rm card} (\Sigma )<\infty$, since $\mu (\overline
B)<\infty$. Now, for any $x_o\in \overline
B\setminus\Sigma$ there exists a radius $\varrho_{x_o}>0$ such that $B_{\varrho_{x_o}} (x_o)\cap \Sigma=\varnothing$
and $\mu (\overline{B_{\varrho_{x_o}} (x_o)})<\varepsilon$.
Since $P_1,P_2,P_3\in \Sigma$, we have in particular that
$B_{\varrho_{x_o}} (x_o)\cap\{ P_1,P_2,P_3\}
=\varnothing$. In the case of a center $x_o\in B$ we choose the disk
in such a way that $B_{\varrho_{x_o}} (x_o)\subseteq B$, while in 
the boundary case $x_o\in \partial B$ we choose $\rho_{x_o}\le
R_o$, where $R_o$ is the radius from Lemma \ref{boundary-W22}. We note
that the radius $R_0$ can be chosen only in dependence on
$\|H\|_{L^\infty},\Gamma,\omega$ and $K:=\sup_k\int_B|f_k|^2\dx$, and
in particular independent from $k\in\N$. The dependence
on $\|H\|_{L^\infty}$ results from the non-linear term
$2(H\circ u)D_1u\times D_2u$ which can be estimated by 
$\|H\|_{L^\infty}|Du|^2$. Since
\begin{equation*}
    \limsup_{k\to\infty} \mu_k (B_{\varrho_{x_o}} (x_o))\le \mu (\overline{B_{\varrho_{x_o}} (x_o)})<\varepsilon,
\end{equation*}
we can find $k_o\in\N$ such that
\begin{equation*}
	\int_{B_{\varrho_{x_o}}^+ (x_o)} |Du_k|^2\, dx
   =
   \mu_k (B_{\varrho_{x_o}} (x_o))
   <
   \varepsilon
   \quad
   \mbox{for any $k\ge k_o$.}
\end{equation*}
Therefore,  the smallness hypotheses
\eqref{small-energy} of Lemma~\ref{lem-interior-W22} 
respectively \eqref{smallness-conditions} of Theorem \ref{boundary-W22}
are fulfilled for $u_k$ with $k\ge k_o$. 
Finally, by assumption we have $u_k\in C^0(\overline
B,\R^3)$ and as stated above, the boundary traces $u_k\big|_{\partial
  B}$ are equicontinuous. Therefore, the application of 
Lemma~\ref{lem-interior-W22}, respectively of Theorem~\ref{boundary-W22},
yields the estimate 
\begin{align*}
    \int_{B_{{\varrho_{x_o}}/4}^+(x_o)}|D^2 u_k|^2\, dx
    &
    \le
    C\bigg[
    \varrho_{x_o}^{-2}\int_{B_{\varrho_{x_o}}^+(x_o)}|Du_k|^2\dx      
    +
    \int_{B} |f_k|^2\, dx
    \bigg]\nonumber\\
    &
    \le
     C\bigg[
     \varrho_{x_o}^{-2}\,\varepsilon
     +
     \sup_{k\in\N}\int_{B} |f_k|^2\, dx
     \bigg]
     =:C.
\end{align*}
for any $k\ge k_o$, with a constant $C$ independent from $k$. This implies the uniform bound
\begin{equation}\label{W22-uniform:klarga}
    \sup_{k\ge k_o}\| u_k\|_{W^{2,2}( B_{{\varrho_{x_o}}/4}^+(x_o),\R^3)}<\infty.
\end{equation}
Here we set $B_{{\varrho_{x_o}}/4}^+(x_o)=B_{{\varrho_{x_o}}/4}(x_o)$ for an interior point $x_o\in B$.
Hence, passing again to a not relabeled subsequence
we have $u_k\wto u$
weakly in  $W^{2,2}( B_{{\rho_{x_o}}/4}^+(x_o),\R^3)$ and strongly
in $W^{1,q}( B_{{\rho_{x_o}}/4}^+(x_o),\R^3)$
for any $q\ge 1$. Because of the Sobolev embedding
$W^{1,q}\hookrightarrow L^\infty$ that holds for $q>2$,
we  moreover have $u_k\to u$ uniformly on $B_{{\rho_{x_o}}/4}^+(x_o)$.
Therefore,  for any $\varphi\in C_0^\infty (B_{{\rho_{x_o}}/4}^+(x_o),\R^3)$ we have
\begin{align}\label{local-Euler}
    &\int_{B_{{\rho_{x_o}}/4}^+(x_o)} Du\cdot D\varphi +2(H\circ u)D_1u\times D_2u\cdot\varphi -f\cdot\varphi\, dx\nonumber\\
    &\phantom{m}=
    \lim_{k\to\infty}
        \int_{B_{{\rho_{x_o}}/4}^+(x_o)}
        Du_k\cdot D\varphi +2(H\circ u_k)D_1u_k\times D_2u_k\cdot\varphi -f_k\cdot\varphi\, dx
    =
    0.
\end{align}
Now we consider the case of a boundary point $x_o\in\partial B\setminus \Sigma$. We choose $w\in T_{u}^\ast\mathcal S$ and
$\zeta\in C^\infty_0(B_{{\rho_{x_o}}/4}(x_o),[0,1])$. By definition
we have $w(e^{i\vartheta})=\widehat \gamma'
(\varphi)(\psi -\varphi)$, where $\varphi$ is defined by
$u(e^{i\vartheta})=\widehat\gamma (\varphi (\vartheta))$
and $\psi\in \mathcal T^\ast (\Gamma)$. For the maps $u_k$ we have the corresponding representations
$u_k(e^{i\vartheta})=\widehat\gamma (\varphi_k (\vartheta))$
for some $\varphi_k\in \mathcal T^\ast (\Gamma)$.
Due to the uniform convergence $u_k\to u$ on $\partial B$
we know $\varphi_k\to\varphi$
uniformly. We then define $w_k$ on $\partial B$ by 
$w_k(e^{i\vartheta}):=\widehat \gamma'(\varphi_k)(\psi -\varphi_k)$.
Its harmonic extension, which we also denote by $w_k$,
clearly is in $W^{1,2}(B,\R^3)\cap L^\infty (B,\R^3)$, because its
boundary trace is contained in $W^{\frac12 ,2}\cap C^0$. Since
\begin{equation*}
	\zeta (e^{i\vartheta})w_k(e^{i\vartheta})
	=
	\widehat \gamma'(\varphi_k)
	\big( \zeta (e^{i\vartheta})\psi 
	+(1-\zeta (e^{i\vartheta}))\varphi_k
	-\varphi_k\big),
\end{equation*}
we see that $\zeta w_k\in T_{u_k}\mathcal{S}^\ast$ and $\spt (\zeta w_k)\subseteq B_{{\rho_{x_o}}/4}^+(x_o)$. Testing the weak Neumann boundary condition \eqref{Plateau-weak} for $u_k$
with $\zeta w_k$ we deduce 
\begin{equation*}
	0
	\le 
	\int_{B_{{\rho_{x_o}}/4}^+(x_o)} 
	\big[Du_k\cdot D(\zeta w_k)+\Delta u_k\cdot (\zeta w_k)\big]\,dx.
\end{equation*}
Since $u_k\in W^{2,2}(B_{{\rho_{x_o}}/4}^+(x_o),\R^3)$ the Gauss-Green theorem  leads us to
\begin{equation*}
	0
	\le 
	\int_{I_{{\rho_{x_o}}/4}(x_o)} 
	\frac{\partial u_k}{\partial r}\cdot (\zeta w_k)\, d\mathcal H^1.
\end{equation*}
In the boundary integral we can pass to the limit $k\to\infty$, since
we have $\frac{\partial u_k}{\partial r}\to \frac{\partial u}{\partial r}$
in $L^2(I_{{\rho_{x_o}}/4}(x_o),\R^3)$ and $w_k\to w$ uniformly on
$I_{{\rho_{x_o}}/4}(x_o)$. This yields
\begin{equation*}
	0
	\le 
	\int_{I_{{\rho_{x_o}}/4}(x_o)} 
	\frac{\partial u}{\partial r}\cdot (\zeta w)\, d\mathcal H^1.
\end{equation*}
Using again the Gauss-Green theorem we finally arrive at
\begin{equation}\label{local-Plateau}
	0
	\le 
	\int_{B_{{\rho_{x_o}}/4}^+(x_o)} 
	\big[Du\cdot D(\zeta w)+\Delta u\cdot (\zeta w)\big]\,dx,
\end{equation}
whenever $w\in T_u\mathcal{S}^\ast$ and $\zeta\in C^\infty_0(B_{{\rho_{x_o}}/4}(x_o),[0,1])$. By a partition of unity argument we conclude
from~(\ref{local-Euler}) and~(\ref{local-Plateau}) that $u$ solves
\begin{align}\label{limit-system-B-Sigma}
     -\Delta u+2 (H\circ u)D_1u\times D_2u=f\quad \mbox{weakly in $B\setminus \Sigma$}
\end{align}
and for any $w\in T_u\mathcal{S}^\ast$ with $\spt w\subseteq
\overline B\setminus\Sigma$ the Neumann type boundary condition
\begin{equation}\label{Plateau-1}
	0
	\le 
	\int_{B} 
	\big[Du\cdot D w+\Delta u\cdot  w\big]\,dx
\end{equation}
holds true.

Next, we wish to establish that \eqref{limit-system-B-Sigma}
holds on the whole of $B$. This can be shown by a capacity argument
along the lines of the proof of \cite[Lemma 7.5]{BoegDuzSchev:2011}. 
Once we know that \eqref{limit-system-B-Sigma} holds on $B$
we can apply the modification of Rivi\`ere's
result from Theorem \ref{mod-Riviere}
to conclude that $u\in C^{0,\alpha}_{\rm loc}(B,\R^3)$
for some $\alpha\in (0,1)$. Since $u\big|_{\partial B}$ is continuous, the same result
yields $u\in C^0(\overline B,\R^3)$.
From \eqref{limit-system-B-Sigma} we conclude that
$\Delta u\in L^1(B,\R^3)$. This allows us to apply again a capacity argument to conclude that \eqref{Plateau-1} holds for any $w\in T_u\mathcal{S}^\ast$, without any restriction on the support of $w$.
To summarize, we have shown that \eqref{limit-system-B-Sigma}
holds on $B$ and \eqref{Plateau-1} holds for any
$w\in T_u\mathcal{S}^\ast$. 

To conclude the proof of (i) we finally show that the limit $u$ also fulfills the stationarity condition \eqref{stationary}. This can be achieved as follows: By assumption
we have the stationarity of the maps $u_k$ in the sense that
for any $k\in\N$ and  every $\eta\in \mathcal C^\ast (B)$
there holds:
\begin{equation}\label{stationary-uk}
    \int_B
	 {\rm Re}
	 \big( \frak h[u_k]\,
	 \overline\partial\eta\big)\, dx
	 -
	 \int_B f_k\cdot Du_k\,\eta\, dx=0.
\end{equation}
Since $f_k\wto f$ weakly in $L^2(B,\R^3)$ and 
$u_k\to u$ strongly in $W^{1,q}(B_{\rho_{x_o}/4}^+(x_o),\R^3)$, we
easily see that for any $\eta\in\mathcal C^\ast (B)$ 
with support in $\overline{B_{\rho_{x_o}/4}^+(x_o)}$, the above
identity is preserved in the limit, that is
\begin{equation}\label{stationary-u}
    \int_B
	 {\rm Re}
	 \big( \frak h[u]\,
	 \overline\partial\eta\big)\, dx
	 -
	 \int_B f\cdot Du\,\eta\, dx=0,
\end{equation}
provided $\spt\eta\subset \overline{B_{\rho_{x_o}/4}^+(x_o)}$. 
A partition of unity argument then yields~(\ref{stationary-u})
for all vector fields $\eta\in \mathcal C^\ast (B)$
with support contained in $\overline B\setminus \Sigma$.
Since $u\in C^0(\overline B,\R^3)$ as noted above,
Theorem \ref{thm:Calderon-Zygmund} yields $u\in W^{1,4}(\Omega,\R^3)$
for any $\Omega\Subset B\setminus\{P_1,P_2,P_3\}$, which implies in
particular $\mathfrak{h}[u]\in L^2(\Omega)$ for any such $\Omega$. In this situation again a capacity argument
implies that $u$ is stationary in the sense of \eqref{stationary}
for any vector field $\eta$ with support compactly contained in
$\overline B\setminus\{P_1,P_2,P_3\}$.
The case of a general vector field $\eta\in\mathcal C^\ast (B)$ is
treated by the following approximation argument.
We choose a cut-off function
$0\le \xi\in C^1_0([0,1])$ with $\xi\equiv 1$ on $[0,\frac12]$
and $|\xi'|\le 3$. For $0<\delta\ll 1$ we consider
\begin{equation*}
	\eta_\delta :=\eta (1-\xi_\delta)
	\quad
	\mbox{where}
	\quad
	\xi_\delta(x):= \sum_{j=1}^3\xi \big(
	\tfrac{|x-P_j|}{\delta}\big).
\end{equation*}
Then, \eqref{stationary} holds true with
$\eta_\delta$. Since $\eta(P_j)=0$ for $j=1,2,3$ and $ \spt (\eta \otimes D\xi_\delta )\subset \bigcup_{j=1}^3
B_\delta (P_j)$, we calculate
$|\eta\otimes D\xi_\delta|\le C\|D\eta\|_{L^\infty}$ and consequently,
 $\|D\eta_\delta\|_{L^\infty}\le C\|D\eta\|_{L^\infty}$ for any $0<\delta\ll1$.
Combining this with $D\eta_\delta \to D\eta$ on
$\overline B \setminus \{ P_1,P_2,P_3\}$ as $\delta\downarrow 0$,
the dominated convergence theorem implies
that \eqref{stationary} holds for $\eta\in\mathcal{C}^\ast(B)$. This proves (i).

Finally, the claim (ii) can be obtained as
follows: Due to the bound
\eqref{bound:uk-zk}, by passing to a non-relabeled subsequence, we may assume that $u_k\wto u$ weakly in $W^{1,2}(B,\R^3)$
and $f_k\wto f$ weakly in $L^2(B,\R^3)$.
Therefore, we can apply the claim (i), which implies together with the Euler-Lagrange system \eqref{Euler-Lagrange} for
the maps $u_k$ that
\begin{align*}
    \int_B 2(H\circ u) D_1u\times D_2u\cdot \varphi\,
    dx
    &=\int_B -Du\cdot D\varphi +f\cdot\varphi\, dx\\
    &=\lim_{k\to\infty} \int_B -Du_k\cdot D\varphi +f_k\cdot\varphi\, dx\\
    &=\lim_{k\to\infty}\int_B 2(H\circ u_k) D_1u_k\times D_2u_k\cdot \varphi\, dx
\end{align*}
holds true for any $\varphi\in C^\infty_0(B,\R^3)$. Since the left-hand side is independent from the subsequence,
the last equality must hold for the whole sequence.
This proves   (ii).
\end{proof}

\section{The approximation scheme}
In this section we follow a method due to Moser \cite{Moser:2009} for the
construction of solutions to  the evolutionary Plateau problem for $H$-surfaces by a time discretization
approach. This method is also known as Rothe's method. This technique has been
applied in \cite{BoegDuzSchev:2011} for the construction of global
weak solutions to the heat flow for surfaces with prescribed mean
curvature with a Dirichlet boundary condition on the lateral boundary.
Since the arguments in this section are somewhat similar
to those in \cite{BoegDuzSchev:2011} we only sketch the proofs
and avoid reproductions. 
Throughout this section, we suppose that the general assumptions
listed in Section \ref{sec:intro} are in force. In particular, we
assume that the prescribed mean curvature function $H$ satisfies an
isoperimetric condition of type $(c,s)$. By 
$u_o\in\SGA$ we denoted a fixed reference surface for which the inequality $\mathbf D(u_o)\le \frac12s(1-c)$ holds true.
We recall that by $\SGAs$ we denoted
the class of all surfaces $w\in\SGA$ with
$\mathbf{D}(w)\le\sigma\mathbf{D}(u_o)$, for $\sigma=\frac{1+c}{1-c}$.
Now, consider $j\in \N_0$ and $h>0$.
We define sequences of energy functionals ${\bf F}_{j,h}$ and maps $u_{j,h}\in \SGAs$ according to the following
recursive iteration scheme: We set $u_{o,h}=u_o$. Once $u_{j-1,h}$ is constructed, the map $u_{j,h}\in\SGAs$
is chosen as a minimizer of the variational problem
\begin{equation}\label{min-j}
    {\bf F}_{j,h}(\tilde u)\longrightarrow\min\quad\mbox{in $\SGAs$,}
\end{equation}
for the energy functional
\begin{align*}
    {\bf F}_{j,h}(\tilde u)&:= {\bf D}(\tilde u)+2{\bf V}_H(\tilde u, u_o) +\tfrac{1}{2h}\int_B \big| \tilde u -  u_{j-1,h}\big|^2\, dx.
\end{align*}
Lemma \ref{lem:F-min} guarantees the existence of such a minimizer $u_{j,h}\in \SGAs$.
We have ${\bf D}(u_{j,h})\le \sigma {\bf D}(u_o)$. Actually, we have the strict inequality
${\bf D}(u_{j,h})< \sigma {\bf D}(u_o)$ for any $j\in\N$, and the same
proof also yields an estimate for the discrete time derivative. 
This follows exactly as in \cite[Lemma 8.1]{BoegDuzSchev:2011} and therefore we state only the result.

\begin{lemma}\label{Mass-max-principle}
Assume that the assumptions of Lemma \ref{lem:F-min} are in force and
$\sigma=\frac{1+c}{1-c}$. Then the minimizers $u_{j,h}\in\SGAs$ of
$\,\mathbf{F}_{j,h}$ satisfy the strict inequality 
\begin{align}\label{fundamental-2}
    {\bf D}(u_{j,h})
    <
     \frac{1+c}{1-c}\, {\bf D}(u_o)=\sigma {\bf D}(u_o).
\end{align}
and
\begin{equation}
  \label{fundamental}
  \sum_{\ell=1}^{j}\tfrac{1}{2h}\int_B \big| u_{\ell,h} -  u_{\ell-1,h}\big|^2\, dx\le 2{\bf D}(u_{o})
\end{equation}
for any $j\in\N$.
\end{lemma}

Since ${\bf D}(u_{j,h})<\sigma{\bf D}(u_{o})$, all
variations that were used in the proof of Lemma \ref{lem:EulerMeasure}
remain admissible also under the additional constraint ${\bf D}(v)\le \sigma{\bf D}(u_{o})$.
It follows that the minimizers $u_{j,h}$ are actually solutions of the
Euler-Lagrange system as stated below.

\begin{theorem}\label{lem:Euler-Lagrange}
Suppose that the assumptions of Lemma \ref{lem:F-min}
are in force and that $\partial A$ is of class $C^2$ with bounded principal
curvatures. Further, assume that 
\begin{equation}\label{curvature-condition}
    |H(a)|\le \mathcal H_{\partial A}(a)\quad\mbox{for $a\in\partial A$.}
\end{equation}
Then any minimizer $u_{j,h}\in\SGAs$ with  $j\in \N$ satisfies the time-discrete Euler-Lagrange system weakly on $B$, that is
\begin{align*}
    \int_B \Big[\frac{u_{j,h}-u_{j-1,h}}{h}\cdot\varphi
    +D u_{j,h}\cdot D\varphi &+2 (H\circ u_{j,h}) D_1u_{j,h}\times D_2u_{j,h}\cdot\varphi
    \Big]\, dx =0
\end{align*}
whenever $\varphi\in L^\infty (B,\R^3)\cap W^{1,2}_0(B,\R^3)$. Moreover, $u_{j,h}$ fulfills the weak form of the Neumann type boundary condition, i.e. we have
\begin{equation*}
	0
	\le 
	\int_{B} 
	\big[Du_{j,h}\cdot D w+\Delta u_{j,h}\cdot  w\big]\,dx
\end{equation*}
for any $w\in T_{u_{j,h}}\mathcal S^\ast$.
Finally, the maps $u_{j,h}$ are stationary
in the sense that there holds
\begin{align*}\label{stationary}
    \int_B&
	 {\rm Re} \big( \frak h [u_{j,h}] \,\overline\partial\eta\big)
	 \, dx+
	 \tfrac{1}{h}\int_B (u_{j,h}-u_{j-1,h})\cdot Du_{j,h}\,\eta\, dx=0
\end{align*}
whenever $\eta\in\mathcal C^\ast (B)$.
\end{theorem}

We now define the approximating sequence, which will lead to the desired global weak
solution  in the limit $h\downarrow 0$. We let
\begin{equation*}
    u_h(x,t):= u_{j,h}(x) \quad\mbox{for $(j-1)h< t\le jh$, $j\in\N$ and $x\in B$}
\end{equation*}
and $u_h(\cdot,t)=u_o$ for $t\le0$. 
Using the finite difference quotient operator in time, that is
\begin{equation*}
  \Delta_t^h v(x,t):=\frac{v(x,t)-v(x,t-h)}{h},
\end{equation*}
we can re-write the Euler-Lagrange system from above in the form
\begin{equation}\label{system-uh}
    \Delta_t^h u_h-\Delta u_{h}+2 (H\circ u_{h})D_1u_{h}\times D_2u_{h}=0\quad \mbox{in $B\times (0,\infty)$.}
\end{equation}
Moreover, we have the stationarity of $u_h$ in the form
\begin{align*}
	\int_{B\times \{t\}}
	{\rm Re} \big( \frak h[u_{h}] \,\overline\partial\eta\big)
	 \, dx+
	\int_{B\times\{t\}} \Delta_t^hu_h\cdot Du_{h}\eta\, dx=0
\end{align*}
whenever $t>0$ and $\eta\in\mathcal{C}^\ast(B)$. Here, $\frak h_{h}:=|D_1u_{h}|^2-|D_2u_{h}|^2 -2\mathbf i D_1u_{h}\cdot D_2u_{h}$. Finally, we have the weak Neumann type boundary condition for the map $u_h$ for any $t>0$, that is
\begin{equation*}
	0
	\le 
	\int_{B\times \{t\}} 
	\big[Du_{h}\cdot D w+\Delta u_{h}\cdot  w\big]\,dx
\end{equation*}
for any $w\in T_{u_{h}(\cdot ,t)} \mathcal S^\ast$.
We mention that $u_h(\cdot ,t)\in\SGA$ for any $t\ge 0$. 
The bounds (\ref{fundamental-2}) and \eqref{fundamental} imply
the energy estimate
\begin{equation}\label{fundamental-bound}
    \sup_{h>0}\sup_{T>0}\bigg[ {\bf D}(u_h(\cdot,T))
     +\tfrac12 \int_0^T\int_B \big|\Delta_t^h u_h\big|^2\, dxdt\bigg]
    \le
    C\,{\bf D}(u_o).
\end{equation}
A version of Poincar\'e's inequality moreover implies
\begin{align*}
    \|u_h(\cdot,T )\|_{L^2 (B)}
    &\le C\|Du_h(\cdot,T )\|_{L^2 (B)} +C
    \|u_h (\cdot, T)\|_{L^2 (\partial B)}\\
    &\le C\sqrt{{\bf D}(u_h(\cdot,T ))} +C(\Gamma)
\end{align*}
for any $h,T>0$, 
which combined with the uniform energy bound \eqref{fundamental-bound} yields 
\begin{equation}\label{L-unendlich-W12}
    \sup_{h>0} \| u_h \|_{L^\infty ((0,\infty),W^{1,2}(B,\R^3))}\le C\| Du_o \|_{L^2(B,\R^3)}+C(\Gamma).
\end{equation}
Next, arguing exactly as in \cite[Chapter 8]{BoegDuzSchev:2011}
we deduce
the following continuity property of $u_h$ with respect to the time direction:
% To this end,
%we let $0\le s<t$ and $0<h<t-s$. In the case $0<h<t-s$ we find $jh\le s<(j+1)h\le mh\le t<(m+1)h$ 
%and obtain
%\begin{align*}
%    \| u_h(\cdot,t)-u_h(\cdot,s )\|_{L^2(B)}
%    &= \| u_{m,h}-u_{j,h}\|_{L^2(B)}
%    \le \sum_{\ell =j+1}^{m} \| u_{\ell,h}-u_{\ell-1,h}\|_{L^2(B)}\\
%    &\le \bigg[ \sum_{\ell =j+1}^{m} \| u_{\ell,h}-u_{\ell-1,h}\|_{L^2(B)}^2\bigg]^\frac12 \sqrt{m-j}\\
%    &= \bigg[ \sum_{\ell =j+1}^{m} \tfrac{1}{2h}\int_B\big | u_{\ell,h}-u_{\ell-1,h}\big|^2\, dx\bigg]^\frac12 \sqrt{2h(m-j)}\\
%    &\le 2\sqrt{2{\bf D}(u_o)}\sqrt{t-s}.
%\end{align*}
%Here we used in the last line the bound \eqref{fundamental} and the
%fact $h(m-j)=h(m-(j+1)) +h\le 2(t-s)$. In the other case, that is when $0<t-s\le h$, we either find
%$j\in\N_o$ such that $jh\le s<t<(j+1)h$ or $jh<s<(j+1)h\le t <(j+2)h$. Here, in the first case we trivially have $\| u_h(\cdot ,t)-u_h(\cdot ,s)\|_{L^2(B)}=0$ since $u_h(\cdot ,t)=u_h(\cdot ,s)$, while in the second case we have
%\begin{equation*}
%	\| u_h(\cdot ,t)-u_h(\cdot ,s)\|_{L^2(B)} =\| u_{j+1,h}(\cdot ,t)-u_{j,h}(\cdot ,s)\|_{L^2(B)}
%	\le 2\sqrt{{\bf D}(u_o)}\sqrt{h}.
%\end{equation*}
%For any $h>0$ that there holds
\begin{equation*}
	\| u_h(\cdot ,t)-u_h(\cdot ,s)\|_{L^2(B)}
	\le 4\sqrt{{\bf D}(u_o)}\big[\sqrt{t-s}+\sqrt{h}\big]
	\qquad\forall\, h>0,\ t>s\ge0. 
\end{equation*}
As in \cite[Lemma 4.1]{Duzaar-Mingione:2005} we can conclude from 
\cite[Theorem 3]{JacquesSimon} that there exists a sequence $h_i\downarrow 0$ and a map
$
    u\in C^{0,\frac12} ([0,\infty); L^2(B,\R^3))\cap L^\infty ([0,\infty); W^{1,2}(B,\R^3))
$
such that
$$
    u_{h_i}\to u\quad \mbox{in $C^0([0,T]; L^2(B,\R^3))$ as
      $i\to\infty$, for all $T>0$.}
$$
Further, we can also achieve
$Du_{h_i}\wto Du$ in $L^2(B\times (0,T),\R^{3\cdot2})$ for every $T>0$,
as $i\to\infty$. Moreover, Lemma~\ref{lem:compact traces} implies
$u_{h_i}(\cdot,t)\to u(\cdot,t)$ uniformly on $\partial B$, from which
we infer $u(\cdot,t)\in\SGA$ for a.e. $t>0$.

%Further, form the uniform bound for $\lambda_h\times\mathcal L^1$ on finite cylinders $B\times (0,T)$ we conclude, possibly passing to a further
%subsequence, that
%\begin{equation*}
%    \lambda_{h_i}\times \mathcal L^1\to \Gamma
%\end{equation*}
%weakly in the sense of Radon measures on $B\times (0,\infty)$.
Now, for  $\varphi\in C^\infty_0(B\times (0,\infty),\R^3)$
we have
\begin{align*}
    \int_0^\infty\int_B u\cdot\partial_t\varphi\, dxdt
%    &=\lim_{i\to\infty} \int_0^\infty\int_B u_{h_i}\cdot \Delta_t^{-h_i}\varphi \, dxdt\\
     =-\lim_{i\to\infty} \int_0^\infty\int_B \Delta_t^{h_i}u_{h_i}\cdot \varphi \, dxdt
%    &\le \| \varphi\|_{L^2(B\times (0,\infty))}\limsup_{i\to\infty} \bigg( \int_{{\rm spt}\varphi} \big|\Delta_t^{h_i}u_{h_i}\big|^2\, dz\bigg)^\frac12 \\
    &\le \sqrt{C\,{\bf D}(u_o)}\, \| \varphi\|_{L^2(B\times (0,\infty))}.
\end{align*}
Here we performed a partial integration with respect to difference quotients in time, applied
the Cauchy-Schwarz inequality and finally used the uniform bound
\eqref{fundamental-bound}. This implies the existence of the weak time
derivative $\partial_tu\in L^2 (B\times (0,\infty),\R^3)$ with 
\begin{equation}\label{bound-ut}
    \int_0^\infty\int_B |\partial_t u| ^2\, dxdt\le C\,{\bf D}(u_o).
\end{equation}
Moreover, we have
\begin{equation}\label{convergence-dt-uhi}
    \Delta_t^{h_i}u_{h_i}\wto \partial_tu\quad\mbox{weakly in $L^2(B\times (0,\infty ),\R^3)$.}
\end{equation}
Next, from  \eqref{fundamental-bound} we conclude that for $0\le t_1<t_2$ there holds
\begin{equation}\label{sup-D(uhi)}
    \sup_{i\in\N}\,\int_{t_1}^{t_2} \int_B |Du_{h_i}|^2\, dxdt\le C\,{\bf D}(u_o)(t_2-t_1)
\end{equation}
and
\begin{equation}\label{sup-delta-uhi}
    \sup_{i\in\N}\,\int_{t_1}^{t_2} \int_B |\Delta_t^{h_i}u_{h_i}|^2\, dxdt\le 2C\,{\bf D}(u_o).
\end{equation}
%Furthermore
%\begin{align*}
%      \int_{t_1}^{t_2}\lambda_{h_i}(t)(B)\, dt
%      &\le C\, T.
%\end{align*}
By Fatou's Lemma we therefore have
\begin{align*}
    \int_{t_1}^{t_2}\liminf_{i\to\infty}& \int_B |Du_{h_i}|^2+|\Delta_t^{h_i}u_{h_i}|^2\, dx \, dt
%    &\le \liminf_{i\to\infty}\int_{t_1}^{t_2} \int_B |Du_{h_i}|^2+|\Delta_t^{h_i}u_{h_i}|^2\, dx\, dt\\
    \le C\,{\bf D}(u_o)\big(2+ t_2-t_1\big)<\infty
\end{align*}
and for almost every $t\in(t_1,t_2)$  we conclude 
\begin{equation}\label{estimate-slice}
    \liminf_{i\to\infty}\int_B |Du_{h_i}(\cdot, t)|^2+|\Delta_t^{h_i}u_{h_i}(\cdot,t)|^2\, dx <\infty.
\end{equation}
Hence, for fixed $t\in(t_1,t_2)$ satisfying \eqref{estimate-slice} and a non-relabeled subsequence (possibly depending on $t$)
we have
\begin{equation*}
    \sup_{i\in\N}\int_B |Du_{h_i}(\cdot, t)|^2+|\Delta_t^{h_i}u_{h_i}(\cdot,t)|^2\, dx <\infty.
\end{equation*}
Next, we consider $k_i\in \N$ such that $(k_i-1)h_i< t \le k_ih_i$.
Then $u_{h_i}(x,t)= u_{k_i,h_i}(x)$ is a minimizer of the functional
\begin{align*}
    {\bf F}_{k_i,h_i}(\tilde u)= {\bf D}(\tilde u)+2{\bf V}_H(\tilde u, u_o) +\tfrac{1}{2h_i}\int_B \big| \tilde u -  u_{k_i-1,h_i}\big|^2\, dx
\end{align*}
in the class $\SGAs$. From Theorem~\ref{lem:Euler-Lagrange} we thereby
infer that $u_{k_i,h_i}$ solves the Euler-Lagrange
system~(\ref{Euler-Lagrange}) with
\begin{equation*}
    f_k:=-\frac{u_{k_i,h_i}-u_{k_i-1,h_i}}{h_i} =-\Delta_t^{h_i}u_{h_i}(\cdot,t),
\end{equation*}
and moreover, the weak form of the Neumann boundary
condition~(\ref{Plateau-weak}) and the stationarity
condition~(\ref{stationary}), again with $f_k$ defined as above. 
Finally, for a fixed time $t\in(t_1,t_2)$ we can pass once more to a subsequence -- which may depend on $t$ -- 
such that $f_{k_i}\wto:f(\cdot,t)$ weakly in $L^2(B,\R^3)$ as $i\to\infty$.
Therefore, all assumptions of Lemma \ref{convergence-non-linearity}\,(i) are fulfilled and we conclude that
the limit $u(\cdot, t)$ satisfies the limit system
\begin{equation}\label{EL-slicewise-f}
    -\Delta u(\cdot, t)+2(H\circ u(\cdot, t))D_1u(\cdot, t)\times D_2u(\cdot, t)=f(\cdot ,t)
\end{equation}
weakly on $B$. Moreover, $u(\cdot ,t)$ fulfills the weak Neumann type boundary condition
\begin{equation}\label{Neumann-slicewise}
	0
	\le 
	\int_{B} 
	\big[Du(\cdot ,t)\cdot D w+\Delta u(\cdot ,t)\cdot  w\big]\,dx
\end{equation}
for any $w\in T_{u(\cdot ,t)}\mathcal S^\ast$, and $u(\cdot ,t)$
is stationary in the sense that
\begin{equation}\label{stationary-f}
    \int_B
	 {\rm Re} \big( \frak h [u(\cdot ,t)] \,\overline\partial\eta\big)
	 \, dx
	 -
	 \int_B f(\cdot,t) \cdot Du(\cdot ,t)\eta\, dx=0
\end{equation}
holds true whenever $\eta\in\mathcal C^\ast (B)$.
We note that this holds whenever $t>0$ is chosen such that \eqref{estimate-slice} holds.
However, since the subsequence chosen above may depend on $t$ this is not enough to identify $-f(\cdot ,t)$ as $\partial_t u(\cdot ,t)$ and
to guarantee that $u$ is the desired global weak solution.
Therefore, for given $a>0$ and $i\in\N$ we define the {\bf set of bad time slices} by
\begin{equation*}
    \Lambda_{i,a}:=\left\{ t\in(t_1,t_2)\,:\, \int_B |Du_{h_i}(\cdot, t)|^2+|\Delta_t^{h_i}u_{h_i}(\cdot,t)|^2\, dx>a\right\}.
\end{equation*}
By \eqref{sup-D(uhi)} and \eqref{sup-delta-uhi}, the measure
$|\Lambda_{i,a}|$ is bounded by 
\begin{align}\label{measure-estimate}
   |\Lambda_{i,a}|
    &\le\frac{ C\,{\bf D}(u_o)\big(2+ t_2-t_1\big)}{a}.
\end{align}
We now define modified sequences $(\tilde u_{h_i})_{i\in\N}$ and $(\tilde f_{h_i})_{i\in\N}$ according to
\begin{equation*}
    \tilde u_{h_i}(x,t)
    :=
    \left\{
    \begin{array}{cl}
        u(x,t) &\mbox{if $t\in \Lambda_{i,a}$,}\\[5pt]
        u_{h_i}(x,t)&\mbox{if $t\not\in \Lambda_{i,a}$,}
    \end{array}
    \right.
\end{equation*}
and
\begin{equation*}
    \tilde f_{h_i}(x,t)
    :=
    \left\{
    \begin{array}{cl}
        f(x,t) &\mbox{if $t\in \Lambda_{i,a}$,}\\[5pt]
        -\Delta_t^{h_i}u_{h_i}(x,t)&\mbox{if $t\not\in \Lambda_{i,a}$.}
    \end{array}
    \right.
\end{equation*}
We observe that for each fixed
$a>0$ we still have $\tilde u_{h_i}\to u$ in $L^\infty
([t_1,t_2];L^2(B,\R^3))$. Furthermore, for a.e. $t\in (t_1,t_2)$ we have that $\tilde u_{h_i}(\cdot ,t)$ solves
\begin{equation*}
    -\Delta \tilde u_{h_i}(\cdot ,t) +2(H\circ \tilde u_{h_i}(\cdot,t))D_1\tilde u_{h_i}(\cdot,t)\times D_2\tilde u_{h_i}(\cdot,t)=\tilde f_{h_i}(\cdot,t)
   ,
\end{equation*}
weakly on $B$. From Theorem~\ref{mod-Riviere}, we therefore infer $\tilde
u_{h_i}(\cdot,t)\in C^0(\overline B,\R^3)$ for a.e. $t\in(t_1,t_2)$.
Moreover, the maps $\tilde u_{h_i}(\cdot ,t)$ fulfill the Neumann type boundary condition and are stationary in the sense
\begin{align*}\label{stationary}
    \int_B&
	 {\rm Re} \big[ \frak h [\tilde u_{h_i} (\cdot ,t)] \,\overline\partial\eta\big]
	 \, dx
	 -
	 \int_B \tilde f_{h_i}(\cdot,t) \cdot D\tilde u_{h_i}
	 (\cdot ,t)\eta\, dx=0
\end{align*}
for any $\eta\in \mathcal C^\ast (B)$. 
% Now, for $t\not \in \Lambda_{i,a}$ we have
% \begin{align*}
%     \int_B |D\tilde u_{h_i}(\cdot, t)|^2+|\tilde f_{h_i}(\cdot,t)|^2\, dx=\int_B |Du_{h_i}(\cdot, t)|^2+|\Delta_t^{h_i}u_{h_i}(\cdot,t)|^2\, dx\le a,
% \end{align*}
% while for a.e. $t \in \Lambda_{i,a}$ we have
% \begin{align*}
%     \int_B |D\tilde u_{h_i}(\cdot, t)|^2+|\tilde f_{h_i}(\cdot,t)|^2\, dx=\int_B |Du(\cdot, t)|^2+|f(\cdot,t)|^2\, dx<\infty.
% \end{align*}
% Combining the last and the second last inequality we obtain that the sequences $\tilde u_{h_i}$ and $\tilde f_{h_i}$ satisfy
%  the following inequality for a.e. $t\in [t_1,t_2]$:
From the definition of $\tilde u_{h_i}$ and $\tilde f_{h_i}$ it is
clear that we have 
\begin{align*}
    \sup_{i\in\N} \int_B &|D\tilde u_{h_i}(\cdot, t)|^2+|\tilde f_{h_i}(\cdot,t)|^2\, dx\nonumber
    \le\max\left\{ a, \int_B |Du(\cdot, t)|^2+|f(\cdot,t)|^2\, dx\right\}<\infty
\end{align*}
for a.e. $t\in(t_1,t_2)$. 
Therefore, we may apply Lemma \ref{convergence-non-linearity}\,(ii) to the
sequences $\tilde u_{h_i}\in\SGA\cap C^0(\overline B,\R^3)$ and
$\tilde f_{h_i}\in L^2(B,\R^3)$ for $i\in\N$, with the result
\begin{equation*}
    \int_{B\times\{t\} } 2(H\circ u)D_1u\times D_2u\cdot\varphi\, dx  
    =
   \lim_{i\to\infty}\int_{B\times \{t\}} 2(H\circ \tilde u_{h_i})D_1\tilde u_{h_i}\times D_2\tilde u_{h_i}\cdot\varphi\, dx
\end{equation*}
for all $\varphi\in C^\infty_0(B,\R^3)$. 
Furthermore, we have
\begin{align*}
    &\int_{B\times \{t\}} 2(H\circ \tilde u_{h_i})D_1\tilde u_{h_i}\times D_2\tilde u_{h_i}\cdot\varphi\, dx\\
    &\qquad\le \| H\|_{L^\infty} \|\varphi\|_{L^\infty} \int_{B} \big| D\tilde u_{h_i}(\cdot ,t)\big|^2\, dx 
     \le  C(H,\varphi)\max\left\{ a, \int_B |Du(\cdot, t)|^2\, dx\right\}
\end{align*}
for a.e. $t\in(t_1,t_2)$. 
Since the right-hand side is in $L^1([t_1,t_2],\R)$, the last two
formulae imply by the dominated convergence theorem that we have the convergence
\begin{align}\label{limit-non-linear-tilde-uhi}
    \int_{t_1}^{t_2}\int_{B }& 2(H\circ u)D_1u\times D_2u\cdot\varphi\, dx\,dt\nonumber\\
    &=\lim_{i\to\infty}\int_{t_1}^{t_2}\int_{B} 2(H\circ \tilde
    u_{h_i})D_1\tilde u_{h_i}\times D_2\tilde u_{h_i}\cdot\varphi\,
    dx\,dt
\end{align}
whenever $\varphi\in C^\infty_0(B\times (t_1,t_2),\R^3)$.
It remains to replace the functions $\tilde u_{h_i}$ on the right-hand
side by the original sequence $u_{h_i}$. To this end, we recall 
the uniform bound \eqref{fundamental-bound} in order to estimate
\begin{align*}
   \int_{B\times\{t\} }& \big| 2(H\circ  u_{h_i})D_1 u_{h_i}\times D_2 u_{h_i}\cdot\varphi\big|\, dx\\
   &\le 2\| H\|_{L^\infty} \| \varphi\|_{L^\infty} \int_{B }| Du_{h_i}(\cdot ,t)|^2\, dx
   \le C\,\| H\|_{L^\infty} \| \varphi\|_{L^\infty}{\bf D}(u_o).
\end{align*}
We integrate this with respect to $t$ over $\Lambda_{i,a}$ and use the
measure estimate \eqref{measure-estimate} to get
\begin{align*}
   &\int_{\Lambda_{i,a}}\int_{B } \big| 2(H\circ  u_{h_i})D_1 u_{h_i}\times D_2 u_{h_i}\cdot\varphi\big|\, dxdt\\
   &\qquad\le C\,\| H\|_{L^\infty} \| \varphi\|_{L^\infty}{\bf D}(u_o)|\Lambda_{i,a}|
     \le \Tilde C\, \frac{2+ t_2-t_1}{a}
\end{align*}
with a constant $\Tilde C$ independent from $i$ and $a$.
Similarly, since $\tilde u_{h_i}(\cdot, t) \equiv u(\cdot, t)$ for $t\in \Lambda_{i,a}$, we have
\begin{align*}
   \int_{\Lambda_{i,a}}\int_{B } \big|& 2(H\circ \tilde u_{h_i})D_1\tilde u_{h_i}\times D_2\tilde u_{h_i}\cdot\varphi\big|\, dx\,dt
   \le \Tilde C \int_{\Lambda_{i,a}}  \int_{B }| D u(\cdot, t)|^2\, dx\,dt.
\end{align*}
Joining the last two estimates we obtain
\begin{align*}
    \bigg|\int_{t_1}^{t_2}&\int_{B } \big[
    2(H\circ \tilde u_{h_i})D_1\tilde u_{h_i}\times D_2\tilde u_{h_i}
    -
    2(H\circ  u_{h_i})D_1 u_{h_i}\times D_2 u_{h_i}\big]\cdot\varphi\, dx\,dt\bigg|\\
    &=\bigg| \int_{\Lambda_{i,a}}\int_{B } \big[\dots\big]\, dx\,dt\bigg|
    \le \widetilde C\bigg[ \frac{1}{a} + \int_{\Lambda_{i,a}}   \int_{B }| D u(\cdot, t)|^2\, dx\,dt\bigg]
\end{align*}
for a constant $\widetilde C$ independent of $i$ and $a$.
At this stage we let $a\to \infty$. In view of \eqref{limit-non-linear-tilde-uhi}, we infer
\begin{align*}
    \int_{t_1}^{t_2}\int_{B } 2&(H\circ u)D_1u\times D_2u\cdot\varphi\, dxdt\nonumber\\
    &=\lim_{i\to\infty}\int_{t_1}^{t_2}\int_{B} 2(H\circ  u_{h_i})D_1 u_{h_i}\times D_2 u_{h_i}\cdot\varphi\, dxdt
\end{align*}
whenever $\varphi\in C^\infty_0(B\times (t_1,t_2),\R^3)$. The last
identity, the weak convergence $Du_{h_i}\wto Du$ in $L^2(B\times
(t_1,t_2),\R^{3\cdot2})$ and the weak 
convergence  $\Delta_t^{h_i}u_{h_i}\wto \partial_tu$ in $L^2(B\times (0,\infty), \R^3)$ from \eqref{convergence-dt-uhi} allow us to pass to the limit
$i\to\infty$ in \eqref{system-uh}, and this proves that $u$ solves the limit system
\begin{equation}\label{limit-system}
  -\Delta u+2 (H\circ u)D_1u\times D_2u=-\d_t u\qquad\mbox{weakly on } B\times(0,\infty).
\end{equation}
%Note that $-f =\partial_t u$ and therefore
%\begin{align*}
%      0\le  \lambda_\ast \times\mathcal L^1
%     &\le\mathcal L^2\times\mathcal L^1\edge \Big[-f \cdot(\nu\circ u)
%     + \big( |H\circ u| -\mathcal H_{\partial A}\circ u\big)|Du|^2\Big]_+
%\end{align*}
%on $u^{-1}\partial A$. Here we used the fact that $\partial_tu\cdot (\nu\circ u)=0$ on $u^{-1}\partial A$.
The above construction yields
$u\in C^{0} ([0,\infty); L^2(B,\R^3))\cap L^\infty ([0,\infty); W^{1,2}(B,\R^3))$ and $\partial_tu\in L^2(B\times (0,\infty),\R^3)$.
Since the weak limit $u$ satisfies the weak Neumann type boundary
condition~(\ref{Neumann-slicewise}), we only have to show the
stationarity condition
\begin{equation}\label{stationary-dtu}
    \int_B
	 {\rm Re} \big( \frak h [u(\cdot ,t)] \,\overline\partial\eta\big)
	 \, dx
	 +
	 \int_B \d_tu(\cdot,t) \cdot Du(\cdot ,t)\eta\, dx=0
\end{equation}
for any $\eta\in\mathcal{C}^\ast(B)$ and a.e. $t>0$.
In view of~(\ref{stationary-f}), it suffices
to show $f(\cdot,t)=-\d_tu(\cdot,t)$ for a.e. $t>0$. But this easily
follows by joining the equations~(\ref{EL-slicewise-f}) and
(\ref{limit-system}).
The proof of Theorem~\ref{main} is thus completed.

\section{Convergence to a stationary solution; proof of Theorem \ref{main2}} 
Here we study the asymptotics of the flow as $t\to\infty$, more
precisely, for a suitable sequence of times $t_k\to\infty$ we wish to show
convergence of the maps $(u(\cdot ,t_k))_{k\in \N}$ to a conformal
$H$-surface $u_*$ satisfying the Plateau boundary condition, i.e. a solution to~(\ref{Plateau-probl-infty}).
% \begin{align}\label{Plateau-stationary}
%   \left\{
%   \begin{array}{c}
%    -\Delta u_*
%      = -2(H\circ u_*)D_1u_*\times D_2u_* 
%      \quad
%      \mbox{weakly in $B$,}
%      \\[7pt]
%      u_*\in\SGA,\\[7pt]
%      |D_1u_*|^2-|D_2u_*|^2=0=D_1u_*\cdot D_2u_*.
%    \end{array}\right.
% \end{align}
Since $\d_tu\in L^2(B\times(0,\infty),\R^3)$ we can find a sequence of
times $t_k\to\infty$ with
\begin{equation}\label{dtu_vanish}
  \int_B|\d_t u(\cdot,t_k)|^2\dx\to 0\quad\mbox{as }k\to\infty.
\end{equation}
Further, we can choose the times $t_k$ in such a way that the partial
maps $u(\cdot,t_k)$ satisfy the Euler-Lagrange
system~(\ref{Euler-Lagrange}), the weak Neumann-type boundary
condition~(\ref{Plateau-weak}) and the stationarity
condition~(\ref{stationary}) with $u$ replaced by $u(\cdot,t_k)$ and
$f$ replaced by $f_k:=-\d_tu(\cdot,t_k)$. 
Since $u\in L^\infty ([0,\infty), W^{1,2}(B,\R^3))$ we have
\begin{equation}\label{Linfty-W12}
    \sup_{k\in\N}\int_B |Du(\cdot ,t_k)|^2+|u(\cdot ,t_k)|^2\, dx<\infty,
\end{equation}
so that we can achieve strong convergence
$u(\cdot,t_k)\to u_*$ with respect to the $L^2$-norm 
and almost everywhere on $B$ for some limit
map $u_*\in W^{1,2}(B,A)$. Lemma~\ref{lem:compact traces} implies
the  uniform convergence $u(\cdot,t_k)|_{\partial B}\to
u_*|_{\partial B}$ of the boundary traces, from which we conclude $u_*\in\SGA$. 
The property~(\ref{Plateau-probl-infty})$_1$ now follows from an application of 
Lemma~\ref{convergence-non-linearity}\,(i) with $f_k=-\d_tu(\cdot,t_k)\to0$ strongly in
$L^2(B,\R^3)$ by~(\ref{dtu_vanish}). Furthermore, the same lemma
yields the stationarity condition
\begin{equation}\label{stationary-limit}
    0=\int_B
	 {\rm Re} \big( \frak h [u_*] \,\overline\partial\eta\big)
	 \, dx
	 =\partial \mathbf{D}(u_*;\eta)
\end{equation}
for any $\eta\in\mathcal{C}^\ast(B)$. Next, we claim that the
conformal invariance of $\mathbf{D}$ yields this
equation in fact for every $\eta\in\mathcal{C}(B)$. To this end,
we define $\varphi_\tau$ as the flow generated by a general vector field
$\eta\in\mathcal{C}(B)$ with $\varphi_0=\mathrm{id}$. For every
$\tau\in(-\eps,\eps)$ we choose the conformal diffeomorphism
$g_\tau\colon \overline B\to \overline B$ defined by $g_\tau(P_j)=\varphi_\tau^{-1}(P_j)$ for
$j=1,2,3$ and define a new variational vector field
$\tilde\eta:=\frac\partial{\partial\tau}\big|_{\tau=0}(\varphi_\tau\circ
g_\tau)$. 
We note that this definition implies $g_0=\mathrm{id}$ and 
$\tilde\eta(P_j)=0$ for $j=1,2,3$, so that
$\tilde\eta\in\mathcal{C}^*(B)$ is admissible
in~(\ref{stationary-limit}). Combining this fact with the conformal
invariance of $\mathbf{D}$, we calculate 
\begin{equation*}
  \partial\mathbf{D}(u_*;\eta)
  =
  \frac d{d\tau}\Big|_{\tau=0}\mathbf{D}(u_*\circ\varphi_\tau)
  =
  \frac d{d\tau}\Big|_{\tau=0}\mathbf{D}(u_*\circ\varphi_\tau\circ
  g_\tau)
  =
  \partial\mathbf{D}(u_*;\tilde\eta)=0.
\end{equation*}
It is well known that the validity of this equation for every
$\eta\in\mathcal{C}(B)$ implies the claimed
conformality~(\ref{Plateau-probl-infty})$_3$ of the limit map $u_*$,
cf. \cite[Sect. 4.5]{Hildebrandt.et.al} or
\cite[Cor. 2.2]{DuzaarSteffen:1999}. 
This completes the proof of~(\ref{Plateau-probl-infty}).

Concerning the regularity of $u_*$, we first infer from the result of Rivi\`ere \cite[Theorem I.2]{Riviere:2007} 
that the limit map $u_*$ is continuous in $B$. 
From classical elliptic bootstrap arguments one then concludes $u\in C^{1,\alpha}_{\rm
  loc}(B,\R^3)$ for every $\alpha\in(0,1)$, cf. e.g. \cite[Lemma 7.2]{BoegDuzSchev:2011},
and an argument by Hildebrandt and Kaul
\cite{HildebrandtKaul:1972} implies even continuity up to the
boundary, i.e. $u_*\in C^0(\overline B,\R^3)$. 

Assuming the prescribed mean curvature  function $H$  to be H\"older, the classical Schauder theory yields $u_*\in C^{2,\beta}_{\rm
  loc}(B,A)$ for some $\beta\in(0,1)$, and $u_*$ is a surface with
mean curvature given by $H$.
The boundary regularity can then be retrieved from \cite[Sect. 7.3, Thm. 2]{Hildebrandt.et.al}, with the result $u_*\in
C^{2,\beta}(\overline B,A)$. For classical solutions $u_*$ to the
$H$-surface equation it is moreover well known that $u_*\in\Ss$
implies that $u\big|_{\partial B}:\partial B\to \Gamma$ is a
homeomorphism, i.e. (\ref{Plateau-probl})$_2$,
cf. \cite[Proof of Satz 3]{Hildebrandt:1969} or 
\cite[Prop. 2.7]{SteffenWente:1978}. 
 We  refer to \cite[Theorem
5.3]{Steffen:1976} for a brief summary of regularity results on $H$-surfaces, 
as well as to \cite{Morrey:1966, Heinz-Tomi:1969, Tomi:1969}.

\end{document}